\documentclass[11pt,a4paper,twoside]{report}
\usepackage{shortfr}

\title{\bsc{Généricité au sens probabiliste dans les difféomorphismes du cercle}\\
\vspace{2em}
\[\includegraphics[scale=.8]{languesr6}\]}
\author{\bsc{Michele Triestino}}
\date{}

\begin{document}
\pagenumbering{gobble}
\newgeometry{margin=2cm}
\begin{titlepage}
	\maketitle
	\begin{abstract}
	What kind of dynamics do we observe in general on the circle? It depends somehow on the
interpretation of ``in general''. Everything is quite well understood in the topological (Baire) setting,
but what about the probabilistic sense? The main problem is that on an infinite dimensional group
there is no analogue of the Lebesgue measure, in a strict sense. There are however some analogues,
quite natural and easy to define: the Malliavin-Shavgulidze measures provide an example and constitute the main character of this text. The first results
show that there is no actual disagreement of general features of the dynamics in the topological and
probabilistic frames: it is the realm of hyperbolicity!
The most interesting questions remain however unanswered...

This work, coming out from the author's Ph.D.~thesis, constitutes an opportunity to review interesting results in mathematical
topics that could interact more often: stochastic processes and one-dimensional dynamics.

\bigskip

After an introductory overview, the following three chapters are a pedagogical summary of classical results about
measure theory on topological groups,
Brownian Motion,
theory of circle diffeomorphisms.

Then we present the construction of the Malliavin-Shavgulidze measures on the space of interval and circle $C^1$ diffeomorphisms, and discuss their key property of quasi-invariance.

The last chapter is devoted to the study of dynamical features of a random Malliavin-Shavgulidze diffeomorphism.
	\end{abstract}
\end{titlepage}

\clearpage
%\cleardoublepage
\pagenumbering{roman}
\chapter{Introduction}
\label{ch:intro}
\section{Généricité topologique et généricité probabiliste}

Dans la théorie moderne des systèmes dynamiques, l'un des points cruciaux est de comprendre leur propriétés \emph{génériques}, puisqu'une connaissance globale s'avère tout simplement impossible pour l'être humain.  Durant les dernières années, d'importants et nombreux résultats ont élucidé le paysage, pqrticulièrement celui de la dynamique décrite par un difféomorphisme $\C{1}$ générique d'une variété compacte fermée, sous l'impulsion du programme de Palis \cite{beyond}. La notion de généricité communément utilisée est \emph{topologique}: dans un espace de Baire (comme par exemple l'espace $\mathrm{Diff}^r(M)$ des difféomorphismes $\C{r}$ d'une variété, $r\in [1,+\infty]$) un ensemble  est générique (ou \emph{résiduel}) s'il contient une intersection dénombrable d'ouverts denses. 

Cependant, atteindre cette seule description n'offre pas une vision satisfaisante et il faut lui accoster une étude \emph{probabiliste}: dans un espace mesuré $(X,\mu)$ (par exemple le $n$-cube unité avec la mesure de Lebesgue) un ensemble est générique -- au sens probabiliste -- si son complémentaire est de $\mu$-mesure nulle. Problème : dans un espace dynamiquement intéressant, comme $\mathrm{Diff}^r(M)$, quelle mesure $\mu$ choisir ? Il n'y en a pas d'intrinsèque et on doit souvent avoir recours aux approximations de dimension finie (où l'on dispose de la mesure de Lebesgue, justement). À présent, la notion qui formalise ce genre d'approche est celle de \emph{prévalence} \cite{prevalence}, inspirée par Kolmogorov \cite{kol54}. 

Dans ce mémoire nous présentons en quelque sorte une << rareté >> : lorsque la variété $M$ est de dimension $1$ (un intervalle ou un cercle), il existe une classe de candidats jouant le rôle de la mesure de Lebesgue sur $\mathrm{Diff}^1(M)$ : les \emph{mesures de Malliavin-Shavgulidze} (MS), apparues dans la littérature physico-mathématique vers la fin des année 1980. Nous entreprenons l'étude des propriétés génériques d'un difféomorphisme tiré au hasard selon l'une de ces mesures.

Malheureusement, l'existence de classes de mesures possédant des propriétés semblables n'est pas connue, ni pour des groupes de difféomorphismes de variétés de dimension majeure, ni pour des espaces d'endomorphismes différentiables.

\section{Mesures de type Haar}

Les mesures MS ont la propriété remarquable d'être de \emph{type Haar} : elles sont
\begin{enumerate}
\item compatibles avec la topologie (elles sont des \emph{mesures de Radon}),
\item \emph{quasi-invariantes} par l'action de translation (à gauche) d'un sous-groupe suffisamment grand, à l'occasion $\mathrm{Diff}^2(M)$ -- à savoir les translations préservent les ensembles de mesure nulle.
\end{enumerate}
De fait, l'existence d'une mesure de Haar sur un groupe topologique (donc, une mesure de Radon quasi-invariante par l'action de translation du groupe \emph{entier} sur lui-même) force le groupe à être \emph{localement compact}.

Nous ferons des rappels plus détaillés dans ce texte. Expliquons plutôt ici l'avantage théorique et pratique de disposer d'une mesure de Haar sur un groupe, qui est le moteur des travaux autour de Shavgulidze. Il est bien connu que tout groupe fini est isomorphe à un sous-groupe du groupe symétrique sur un nombre d'objets égal à son ordre. Ainsi, le groupe symétrique sur $n$ éléments contient une copie de tous les groupes de cardinal au plus $n$. Notons que l'hypothèse de finitude n'est pas nécessaire : tout groupe s'injecte dans le groupe des bijections d'un ensemble de cardinal égal à l'ordre du groupe. En revanche, lorsque le cardinal n'est plus fini, le groupe symétrique est trop grand pour être compris : nous nous en sortons avec l'emploi de l'analyse. Si un groupe $G$ possède une mesure de Haar $\mu$, on peut regarder l'espace de Hilbert des fonctions $L^2$ sur $G$ :
\[L^2(G,\mu)=\left \{f:G\longrightarrow \mathbf{C}\trm{ mesurable}\,\middle\vert\,\int_G f^2\,d\mu<\infty\right \}.\]
Or, l'action de $G$ par multiplication à gauche induit une action linéaire sur $L^2(G,\mu)$ : étant donné $g\in G$, on associe à $f\in L^2(G,\mu)$ la fonction $L_gf$ qui fait correspondre à tout $x\in G$ la valeur $f(g^{-1}x)$. Puisque la mesure $\mu$ est invariante, cette action est aussi unitaire. La conséquence que l'on en tire est la suivante : tout groupe localement compact est isomorphe à un sous-groupe du groupe des opérateurs unitaires dans un espace de Hilbert. D'après Shavgulidze, on peut développer une telle analyse harmonique aussi pour un groupe de dimension infinie tel que le groupe des difféomorphismes du cercle.

Une telle vision est utile lorsque l'on s'intéresse à la théorie des représentations linéaires d'un groupe, mais elle est rarement d'utilité lorsque l'on veut étudier la dynamique non-linéaire. Une (rare ?) exception apparait lorsque l'on se place dans le monde des groupes de Lie (on pense ici aux grands résultats ergodiques obtenus pour les réseaux depuis Margulis). 
Mais encore, le dynamicien peut-il être satisfait si un problème apparemment simple comme celui de comprendre l'image de la mesure de Haar par l'application d'élévation à puissance $g\mapsto g^{\,k}$, est difficilement abordable ? Même pour le groupe linéaire $GL_n(\R)$ cette question n'est pas si simple (heureusement la décomposition spectrale nous sauve !). Les mesures de Haar ne sont pas proprement dynamiques : les objets naturels en dynamique sont ceux qui sont invariants par conjugaison, non par composition d'un coté !\footnote{On pourrait argumenter qu'une mesure de Haar $\mu$ invariante à gauche permet de construire une mesure bi-invariante (et donc invariante par conjugaison) : si $\iota:G\to G^{op}$ est le morphisme vers le groupe opposé, $\iota(g)=g^{-1}$, alors $\mu*\iota_*\mu$ est bi-invariante. Mais il s'agit d'une construction artificielle.}

Cependant, en reprenant le discours initial, une mesure « homogène » comme la mesure de Haar est assez appropriée pour décrire le comportement générique au sens probabiliste : très souvent on rencontre les mêmes phénomènes génériques sous les deux différents points de vue (topologique et probabiliste), même si parfois les deux descriptions peuvent être contradictoires. Pour donner un exemple fortement relié à la dynamique sur le cercle, les nombres diophantiens forment une partie de mesure de Lebesgue totale dans $\R$, alors que les nombres de Liouville forment un ensemble résiduel.

\smallskip

D'une nature un peu différente était la motivation du travail autour de Malliavin : au lieu d'une analyse harmonique en dimension infinie, le \emph{calcul de Malliavin} offre (entre autres) le bon cadre pour une théorie des distributions de Schwartz en dimension infinie. En dimension finie, l'intégration par parties est au cœur de la définition des distributions : en commençant par comprendre les transformations qui laissent quasi-invariante une mesure sur un espace de dimension infinie, il est possible d'étendre l'intégration par parties et d'étudier la régularité des transformations de l'espace. 

\section{Les mesures de Malliavin-Shavgulidze}

Esquissons brièvement la construction de Malliavin et Shavgulidze : l'idée est brillante, l'explication très simple ! Depuis Wiener, nous savons choisir uniformément une fonction continue sur le cercle de moyenne nulle : nous appelons une telle fonction aléatoire \emph{pont brownien}. Ensuite on définit le difféomorphisme aléatoire par intégration, en choisissant uniformément par rapport à la mesure de Lebesgue l'image d'un point donné (par exemple $0$) : en termes plus précis, si $B=(B_t)_{t\in [0,1]}$ est un pont brownien et $\lambda$ une variable aléatoire uniforme sur $[0,1]$ indépendante de $B$, nous définissons le difféomorphisme aléatoire de Malliavin-Shavgulidze par
\beqn{diffeomorphismeMS}{
f(t)=\dfrac{\int_0^t\exp(B_s)\,ds}{\int_0^1\exp(B_s)\,ds}+\lambda.
}
Par suite, nous disposons d'une mesure $\mu_{MS}$ sur le groupe $\Dc$ que l'on appelle \emph{mesure de Malliavin-Shavgulidze}.

Un résultat de Cameron et Martin classique montre que la mesure de Wiener sur l'espace des fonctions continues $C_0(\T)$ est de \emph{type Haar}.
Il existe un autre théorème, moins connu, de Cameron et Martin \cite{cameron-martin2}, qui implique entre autres que la mesure de Malliavin-Shavgulidze est de type Haar pour $\Dc$ : le sous-groupe des difféomorphismes $C^2$ préserve la classe de la mesure MS.

\smallskip

Dans le chapitre \ref{chapter:1}, nous avons recueilli les résultats principaux sur les mesures MS, en rajoutant des détails aux preuves qui en nécessitent. En particulier nous présentons le théorème de Shavgulidze qui explicite la dérivée de Radon-Nikodym associée à l'action par multiplication à gauche de $\Dt$ sur $(\Dc,\mu_{MS})$ :
\[
\dfrac{d(L_\vf)_*\mu_{MS}}{d\mu_{MS}}(f)=\exp\left \{\int_{\T}\mathcal{S}_\vf(f(t))\,(f'(t))^2\,dt\right \},
\]
puis nous rappelons le théorème de Kosyak qui montre l'ergodicité de cette action.

\smallskip

Bien évidemment, lorsque l'on dispose d'une mesure quasi-invariante à gauche comme $\mu_{MS}$, on peut définir la mesure $d\tilde{\mu}_{MS}(f)=d\mu_{MS}(f^{-1})$, qui est naturellement quasi-invariante \emph{à droite}. Ensuite, par convolution on peut obtenir aussi une mesure \emph{quasi-bi-invariante} $\mu_{MS}*\tilde{\mu}_{MS}$ : elle est peut-être plus intéressante d'un point de vue dynamique, puisqu'elle est également quasi-invariante par l'action de conjugaison. Par rapport à cette mesure, un difféomorphisme aléatoire s'écrit sous la forme $f_+\circ f_-^{-1}$, où les $f_{\pm}$ sont de difféomorphismes aléatoires du type \eqref{eq:diffeomorphismeMS}.

Il ne faudra pas cacher que beaucoup d'autres mathématiciens ont construit des mesures sur les groupes de difféomorphismes des variétés. Malgré tous les efforts, il est très difficile d'obtenir de belles mesures lorsque les variétés sont de dimension plus grande que $1$ : l'une des difficultés majeures est la nature algébrique de ces groupes de Lie de dimension infinie, qui devient trop compliquée, alors qu'une autre provient de la régularité qui s'affaiblit lorsque l'on monte en dimension (comme par exemple dans le cas des immersions de Sobolev). Il serait très intéressant de disposer d'une construction simple de mesure quasi-invariante sur les difféomorphismes du tore ou du disque, qui préservent l'aire. À ce propos, rappelons que Shavgulidze a défini d'autres mesures quasi-invariantes, par exemple sur l'espace des difféomorphismes $\C{4}$ du disque, mais il s'agit de constructions assez indirectes.

Nous n'essaierons pas de décrire d'autres exemples de mesures, car chacune porte ses caractéristiques spécifiques et il est presque impossible de les présenter sous une théorie unificatrice.
La simplicité, la propriété de quasi-invariance et les propriétés de régularité nous ont convaincus que les mesures MS méritent un certain regard.

\section{Étude de la dynamique générique}

Dans le chapitre \ref{chapter:2}, nous entreprenons l'étude \emph{dynamique}. Comme d'habitude, on essaie de décrire le comportement d'un difféomorphisme à conjugaison près.
Depuis le travail fondateur de Poincaré \cite{poincare}, il est très naturel de comprendre, en premier lieu, quelle pourrait être la distribution du \emph{nombre de rotation}, qui est l'invariant de conjugaison le plus célèbre.\footnote{Pour le lecteur qui n'est pas initié à la théorie des difféomorphismes du cercle, nous parcourons les aspects principaux dans le chapitre~\ref{sc:diffeos}.} Il s'avère que ce problème semble difficile d'accès.

À cause de la présence de difféomorphismes avec orbites périodiques hyperboliques, les difféomorphismes avec orbites périodiques (et donc nombre de rotation \emph{rationnel}), constituent une partie de mesure strictement positive.

Que peut-on dire sur les orbites d'un difféomorphisme choisi uniformément par rapport à la mesure MS ? Existe-t-il presque toujours une orbite périodique ?

\begin{figure}[ht]
\[
\includegraphics[scale=.3]{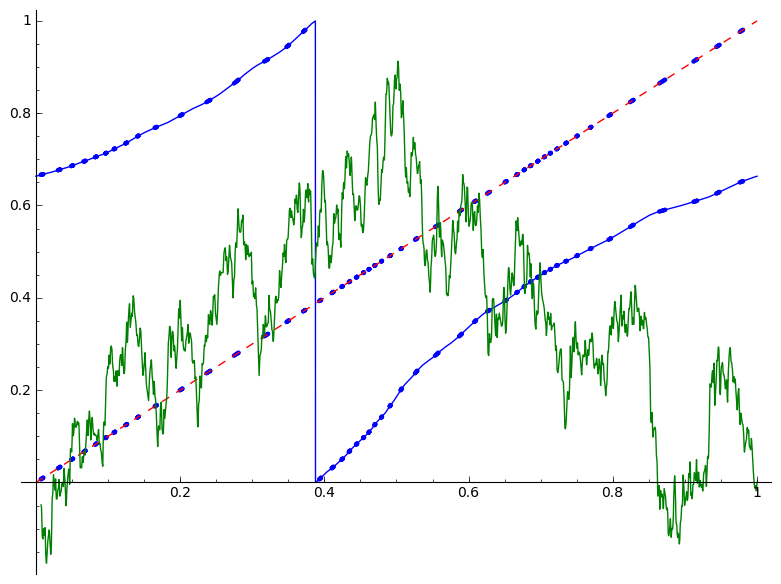}\qquad \includegraphics[scale=.3]{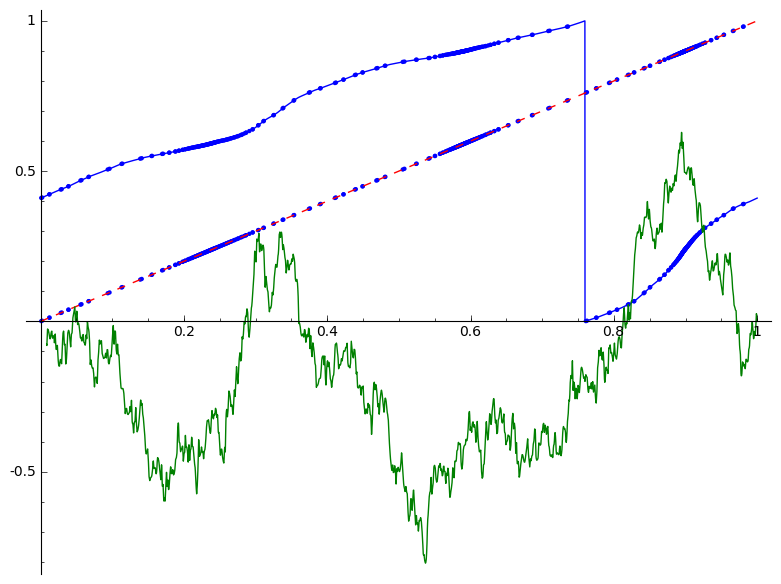}
\]
\caption{Deux exemples de difféomorphismes aléatoires, les ponts browniens qui les définissent et la trace d'une orbite.}
\end{figure}

Nos résultats donnent une bonne description une fois que l'on a conditionné à qu'il y ait des orbites périodiques :

\begin{thm}\label{thm:periodiques_intro}
Les difféomorphismes du cercle qui possèdent un nombre fini d'orbites périodiques forment une partie de mesure MS totale parmi ceux qui ont des points périodiques.
\end{thm}

La preuve s'appuie fortement sur les propriétés du mouvement brownien et s'inspire de la démonstration du fait suivant: le mouvement brownien planaire ne passe presque sûrement par aucun point donné (que nous avons trouvée dans \cite{peres}\,).

Plus précisément, nous montrons que presque sûrement les orbites périodiques sont \emph{hyperboliques} : s'il existe une orbite périodique, alors on en trouve un nombre pair, une moitié est répulsive, l'autre est attractive. Les points sur le cercle qui n'appartiennent pas à l'une de ces orbites sont attirés vers deux d'entre elles, vers l'une dans le futur et vers l'autre dans le passé.
Ce genre de difféomorphismes prend le nom de \emph{Morse-Smale}. Une telle propriété permet d'étudier de manière satisfaisante la dynamique, même en faible régularité : en généralisant un résultat classique de Sternberg, les difféomorphismes Morse-Smale de classe $\C{1+\alpha}$ peuvent être localement linéarisés, on peut borner leur distorsion et, en suivant Kopell \cite{kopell}, on peut facilement décrire leur centralisateur (et ainsi leur classe de conjugaison).

\begin{thm}
Les difféomorphismes du cercle qui ont centralisateur $\C{1}$ trivial forment une partie de mesure MS totale parmi ceux qui ont des points périodiques.
\end{thm}

Cette description confirme l'intuition que l'on a au niveau topologique : dans la topologie $\C{1}$, les difféomorphismes de type Morse-Smale, avec centralisateur $\C{1}$ trivial forment une partie ouverte et dense de~$\Dc$ \cite{kopell,diviseurs}. Les théorèmes précédents montrent que le complémentaire est invisible d'un point de vue de la mesure, au moins parmi les difféomorphismes avec points périodiques.

\smallskip

La preuve du théorème \ref{thm:periodiques_intro} peut être adaptée pour démontrer un résultat qui est intéressant en soi :
\begin{thm}\label{thm:free_intro}
Soient $f_1,\ldots,f_n$ des difféomorphismes MS indépendants. Alors le sous-groupe qu'ils engendrent est presque-sûrement libre (de rang $n$).
\end{thm}
Ce théorème confirme l'observation topologique : Ghys démontre dans \cite{ghys-circle} que génériquement dans $\Hom$, deux homéomorphismes engendrent un sous-groupe libre. Il est aussi possible de voir le théorème~\ref{thm:free_intro} comme la généralisation du théorème de Epstein \cite{epstein} dans le contexte des groupes de Lie au groupe de Lie de dimension infinie $\Dc$.

\begin{rem}
L'espace des actions fidèles du groupe libre sur le cercle se sépare qualitativement en deux sous-parties : les actions peuvent être minimales ou laisser un ensemble de Cantor invariant. Par rapport à la topologie naturelle, ces deux sous-parties contiennent des ouverts non-vides. On en déduit que les deux différentes dynamiques sont observables avec probabilité strictement positive.
\end{rem}

\bigskip

Pour ce qui concerne la présence des difféomorphismes sans orbites périodiques, nous pensons qu'ils forment un ensemble négligeable : une partie du deuxième chapitre est dédiée à nos conjectures à ce propos.
Sans rentrer dans le détail, nous terminons cette introduction par une image qui montrent la statistique du nombre de rotation, sur un échantillonnage d'un millier de difféomorphismes « aléatoires ».

\begin{figure}[ht]
\[
\includegraphics[scale=.4]{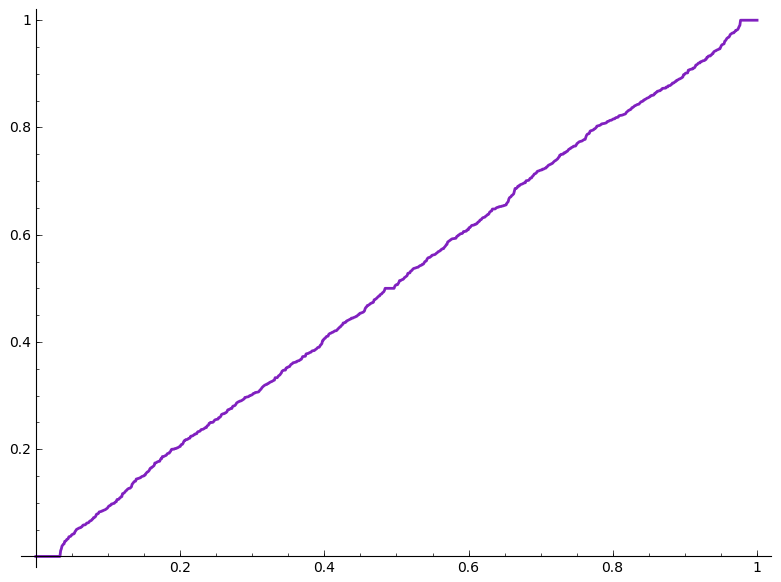}
\]
\caption{Cet escalier représente la fonction de répartition de la variable \emph{nombre de rotation}, selon la mesure MS.}
\end{figure}

\section*{Remerciements}

Ce livre doit beaucoup à Étienne Ghys et aux questions qu'il m'a posées tout au long de la thèse. Également l'ambiance conviviale et multidisciplinaire de l'Unité de Mathématiques Pures et Appliquées de l'École Normale Supérieure de Lyon a été fondamentale : je tiens à remercier en particulier Alessandro Carderi, Christophe Garban, Emmanuel Jacob, Cristóbal Rivas, Bruno Sévennec, Marielle Simon et Romain Tessera pour leur conseils, parfois fondamentaux. En outre, j'ai bénéficié énormément des discussions avec François Béguin, Victor Kleptsyn, Marco Martens et Andrés Navas.

Beaucoup d'arguments sont devenus plus clairs après le cours que j'ai eu le plaisir de donner à la Pontifícia Universidade Católica de Rio de Janeiro pendant l'automne 2014 : je remercie tous les participants, en particulier Sébastien Alvarez, Pablo Barrientos, Dominique Malicet, Carlos Meniño, Maurizio Monge, Isaia Nisoli et Artem Raibekas.
\clearpage
%\cleardoublepage
\tableofcontents

\clearpage
%\cleardoublepage
\pagenumbering{arabic}

\chapter{Mesures quasi-invariantes et groupes localement compacts}
\label{sc:localement}
Le fil conducteur sera l'interaction (ou la \emph{non}-interaction) entre groupes, mesures et topologie. En dimension finie, ceci se concentre principalement dans la notion de mesure de Haar, bijou mathématique que l'on perd lorsque la dimension est infinie. Ce livre part du problème suivant : comprendre quelles sont les mesures qui pourraient remplacer celle de Haar dans certains espaces trop grands.

Nous nous proposons ici de décrire les grandes lignes de ce paysage : des textes plus que complets sont disponibles dans la littérature comme références plus techniques. Nous  conseillons, à ce propos, le très beau livre de Halmos \cite{halmos} qui porte particulièrement bien ses 60 ans !

\section{Notions de base de théorie de la mesure}
\label{s:basemesure}
Soit $X$ un ensemble. Une famille non-vide $\cB\subset \mathbf{2}^{X}$ de parties de $X$ est une \emph{tribu} (ou \emph{$\sigma$-algèbre}) sur $X$ si $\cB$ est fermée par les opérations suivantes :
\begin{itemize}
\item si $E\in \cB$ alors $X-E\in \cB$,
\item si $(E_n)_{n\in\N}\subset\cB$ alors $\bigcup_{n\in \N}E_n\in\cB$.
\end{itemize}
Un couple $(X,\cB)$ est dit \emph{espace mesurable}. Une fonction $f$ entre deux espaces mesurables $(X,\cB)$ et $(X',\cB')$ est \emph{mesurable} si $f^{-1}(E')\in \cB$ pour tout $E'\in \cB'$.

\begin{ex}
La droite réelle $\R$ munie de sa tribu borélienne $\cB(\R)$(\emph{i.e.}~engendrée par les ouverts) est un espace mesurable. Plus généralement, on peut associer un espace mesurable $(X,\cB(X))$ à tout espace topologique $(X,\tau)$, en considérant la tribu des boréliens sur $X$.
\end{ex}

Un \emph{espace mesuré} $(X,\cB,\mu)$ est la donnée d'un espace mesurable $(X,\cB)$ et d'une mesure $\mu$ définie sur $\cB$ (\textit{i.e.}~$\mu(E)\in[0,+\infty]$ est défini pour tout $E\in\cB$). Lorsque la msure totale $\mu(X)$ est égale à $1$, on dit que $\mu$ est une mesure de probabilité.

Lorsque $f$ est une fonction mesurable de $(X,\cB)$ dans $(X',\cB')$, pour toute mesure $\mu$ définie sur $(X,\cB)$ on peut définir la \emph{mesure image} $f_*\mu$ sur $(X',\cB')$ : $f_*\mu(E'):=\mu(f^{-1}(E'))$ pour tout $E'\in\cB'$.

Ne voulant pas rentrer dans toutes les subtilités de la théorie de la mesure, nous supposerons, dorénavant, que tout espace mesuré $X$ est un espace mesurable \emph{standard} avec une mesure \emph{$\sigma$-finie}. Cela signifie que l'espace~$(X,\cB,\mu)$ est mesurablement isomorphe à l'espace $\left ([0,1],\cB\left ([0,1]\right ),\Leb\right )$. Cette restriction est plutôt fictive pour nous, car dans les exemples que nous traiterons, ces hypothèses seront toujours vérifiées.

\section{Actions sur espaces mesurés}
\label{ssc:actions}

\begin{dfn}
Soit $G$ un groupe qui agit (à gauche) sur un espace mesuré $(X,\cB,\mu)$, avec $\cB$ invariante par l'action de~$G$ et $\mu$ mesure $\sigma$-finie non-identiquement nulle. On dit que la mesure $\mu$ est \emph{quasi-invariante} pour l'action si pour tout~$g\in G$, l'image $gA$ d'un ensemble $A\in \cB$ est de mesure nulle si et seulement si la mesure de $A$ est nulle.

Lorsque $\mu(gA)=\mu(A)$ pour tout $A\in \cB$ et tout $g\in G$, la mesure $\mu$ est \emph{invariante}.
\end{dfn}

Si $G$ agit sur $(X,\cB,\mu)$ et $\mu$ est quasi-invariante, il existe une unique fonction
\[\rho:G\to L^1_{loc}(X,\mu)\]
telle que pour tout $g\in G$, $\int_Xf(x)\,d\left (g_*\mu\right )(x)=\int_X f(x)\,\rho_g(x)\,d\mu(x)$ pour tout $f\in L^1(X,g_*\mu)$. Ce résultat est le bien connu \emph{théorème de Radon-Nikodym} et la fonction $\rho_g$ est plus souvent notée comme la dérivée (de Radon-Nikodym) $\frac{dg_*\mu}{d\mu}$ ; une telle notation est surtout appropriée si l'on veut souligner sa nature de \emph{cocycle} (de Radon-Nikodym) : pour tous $g,h\in G$,
\[\frac{d(gh)_*\mu}{d\mu}(x)=\frac{dg_*\mu}{d\mu}(hx)\cdot \frac{dh_*\mu}{d\mu}(x)\]
en dehors d'un ensemble $\mu$-négligeable de $X$ (\emph{i.e.}~de mesure nulle).

\begin{rem}
Lorsque la mesure totale $\mu(X)$ est finie, alors pour tout élément du groupe $g\in G$, la dérivée de Radon-Nikodym $\rho_g$ est intégrable, de norme $\|\rho_g\|_{L^1}=\mu(X)$.
\end{rem}

\begin{dfn}
Si $\mu$ et $\nu$ sont deux mesures sur $(X,\cB)$ avec la propriété que $A$ est $\mu$-négligeable si et seulement si~$A$ est $\nu$-négligeable, on dit que $\mu$ et $\nu$ définissent la même \emph{classe} de mesure $[\mu]=[\nu]$, ou, également, que $\mu$ est \emph{équivalente} à $\nu$, et on écrit $\mu\sim \nu$.

Lorsqu'une seule des implications est vérifiée, par exemple si $A$ est $\nu$-négligeable seulement lorsque $A$ est~$\mu$-négligeable, on dit que $\nu$ est \emph{absolument continue} par rapport à $\mu$ et on écrit $\nu\preceq\mu$.
\end{dfn}

Une mesure $\mu$ est quasi-invariante sous l'action de $G$ si et seulement si les éléments de $G$ préservent la classe de $\mu$.

\section{Actions continues sur espaces topologiques mesurés}
Lorsque l'on suppose que le groupe $G$ agit de façon continue sur un espace topologique $X$, il est préférable (et plus intéressant) de considérer des mesures sur $X$ qui s'adaptent à la topologie. Dans la suite, on supposera toujours que les espaces topologiques en question sont séparables.

\begin{dfn}
Soit $(X,\tau,\cB)$ un \emph{espace topologique mesurable} (ici $\cB$ est la tribu borélienne\footnote{Normalement on demande que $\cB$ \emph{contienne} la tribu borélienne.}). Une mesure $\mu$ sur~$(X,\tau,\cB)$ est une \emph{mesure de Borel} si $\mu$ est non-identiquement nulle et finie sur les parties compactes de $X$.\footnote{Dans la littérature, la définition de mesure de Borel peut changer selon les différents auteurs : la propriété qu'elle soit finie sur les compacts n'est pas toujours demandée.}
\end{dfn}

Rappelons qu'une mesure de Borel est \emph{régulière à l'extérieur} si pour tout $E\in \cB$,
\[\mu(E)=\inf\{\mu(U)\,:\,U\trm{ ouvert qui contient }E\}.\]
De manière similaire, $\mu$ est dite \emph{régulière à l'intérieur} si pour tout $E\in \cB$,
\[\mu(E)=\sup\{\mu(K)\,:\,K\trm{ compact contenu dans  }E\}.\]
Une mesure régulière dans le sens précédent est essentiellement déterminée par ses valeurs sur les ouverts ou les compacts de $X$ et une telle mesure est donc appropriée dans un contexte topologique. Ces deux notions sont presque équivalentes, mais il s'avère que demander la régularité \emph{intérieure} est préférable.

\begin{dfn}
Soit $(X,\tau,\cB)$ un espace topologique mesurable. Une mesure de Borel $\mu$ est une \emph{mesure de Radon} si elle est finie sur les compacts et régulière à l'intérieur (par rapport aux compacts).
\end{dfn}

Les mesures que nous utiliserons seront de Radon.

\begin{ex}
La mesure de Lebesgue définie sur un ouvert de $\R^n$ est évidemment une mesure de Radon et donne une mesure positive à tout ouvert.
\end{ex}
\begin{ex}
La mesure de Wiener sur l'espace des fonctions continues $C_0([0,1])$ (voir la partie \ref{sc:wiener}\,) est une mesure de Radon et donne une mesure positive à tout ouvert (d'après le théorème de Wiener \ref{thm:wiener} il suffit de le vérifier sur les cylindres).
\end{ex}

Lorsque $G$ agit de manière continue sur $X$, on peut se demander s'il existe une mesure de Radon \mbox{(quasi-)invariante}. Sans hypothèses supplémentaires sur $X$, la réponse à cette question est, en général, négative. En contrepartie, dès que l'on suppose que $X$ est \emph{localement compact}, on trouve des résultats surprenants !

\begin{thm}
\label{thm:steinlage}
Soit $X$ un espace topologique séparable localement compact et soit $G$ un groupe qui agit sur $X$ de façon continue. On fait les deux hypothèses suivantes :
\begin{enumerate}
\item toute orbite est dense,
\item lorsque $K$ et $L$ sont deux compacts disjoints dans $X$, il existe un ouvert $U$ non vide dans $X$ tel que, pour tout $g\in G$, le translaté $gU$ n'intersecte pas $K$ et $L$ en même temps.
\end{enumerate}
Alors il existe une mesure de Radon $\mu$ sur $X$, \emph{invariante} par l'action de $G$.
\end{thm}

Ce théorème est dû à Steinlage \cite{steinlage75} et nous l'avons appris sous cette forme dans l'œuvre (monumentale) de Fremlin \cite[441C]{fremlin}. En fait, cet énoncé généralise le très célèbre théorème d'existence de la \emph{mesure de Haar} pour les groupes localement compacts, que nous allons discuter dans la prochaine partie.

\smallskip

Pour conclure, terminons par un résultat technique qui s'avérera utile pour étendre par approximation les propriétés de quasi-invariance en dimension finie (voir~\cite[Lemma~6.1.8]{bogachev-gaussian}).

\begin{prop}\label{prop:approx}
Soit $X$ un espace topologique métrisable\footnote{Cette hypothèse peut être affaiblie.} avec une mesure de probabilité  de Radon $\mu$, soit $T:X\to X$ une application $\mu$-mesurable telle que $T_*\mu$ est de Radon, et soit $T_n:X\to X$ une suite d'applications mesurables qui convergent $\mu$-presque partout vers $T$. Supposons que les mesures $(T_n)_*\mu$ sont absolument continues par rapport à $\mu$ et que les dérivées de Radon-Nikodym correspondantes forment une suite uniformément intégrable (\textit{i.e.}~relativement compacte dans $L^1$).

Alors la mesure $T_*\mu$ est absolument continue par rapport à $\mu$ et la dérivée de Radon-Nikodym est la limite des dérivées dans la topologie faible de $L^1(X,\mu)$.
\end{prop}

\section{Mesures quasi-invariantes sur les groupes topologiques}
Les actions qui nous intéresserons le plus seront celles des groupes topologiques sur eux-mêmes. Tout groupe agit naturellement sur lui-même par translations, de deux cotés différents : il peut agir \emph{à gauche} par l'action
\[\dfcn{L}{G\times G}{G}{(g,x)}{g^{-1}x}\]
et \emph{à droite} par l'action
\[\dfcn{R}{G\times G}{G}{(g,x)}{xg}\]
et ces actions sont dites \emph{régulières}.

\begin{rem}
Les actions $L$ et $R$ commutent toujours. On remarquera aussi que ces actions sont \emph{libres} et \emph{transitives} : si $g^{-1}x=h^{-1}x$ alors $g=h$, et si $x$, $y\in G$ alors, en posant $g=xy^{-1}$, on a $g^{-1}x=y$.
\end{rem}

Une autre action dynamiquement intéressante est l'\emph{action par conjugaison} :
\[\dfcn{\g}{G\times G}{G}{(g,x)}{gxg^{-1}}\]

\begin{rem}
Contrairement aux actions régulières, l'action par conjugaison n'est ni libre, ni transitive (si le groupe n'est pas trivial !). Les orbites forment les \emph{classes de conjugaison} dans $G$ et le stabilisateur d'un élément~$x\in G$ coïncide avec le \emph{centralisateur} de $x$ dans $G$.
\end{rem}

Lorsque le groupe $G$ est un groupe topologique, ces actions sont continues.

\begin{dfn}
Une \emph{mesure de Haar} invariante à  gauche (resp. à droite) sur $G$ est une mesure de Radon sur $G$, invariante par l'action régulière gauche (resp. droite).
\end{dfn}

\begin{ex}
Lorsque $G$ est le groupe abélien $\R^n$ muni de sa topologie euclidienne,  il est classique que toute mesure de Haar doit être proportionnelle à la mesure de Lebesgue.
\end{ex}

\begin{ex}
Soit $\G$ un groupe dénombrable discret. La mesure de comptage définit une mesure de Haar sur $\G$ qui est invariante à gauche et à droite en même temps. Il est facile d'observer que toute autre mesure de Haar sur $\G$ doit être proportionnelle à la mesure de comptage. En outre, la mesure de Haar est finie si et seulement si le groupe $\G$ est fini.
\end{ex}

Ce dernier exemple, même un peu naïf, représente pleinement la situation pour les groupes \emph{localement compacts}. Nous avons déjà annoncé que le théorème de Steinlage \ref{thm:steinlage} est une version « décorée » du théorème de Haar que nous énonçons ici.

\begin{thm}
\label{thm:haar}
Sur un groupe localement compact $G$, il existe toujours une mesure de Haar invariante à gauche, et deux telles mesures sont proportionnelles.

De plus, la mesure est finie si et seulement si le groupe $G$ est compact.
\end{thm}

\begin{rem}
Toute mesure de Haar $\mu$ invariante à gauche définit une mesure de Haar $\nu$ invariante à droite : si~$E\in \cB$ est un borélien dans $G$, on pose $\nu(E)=\mu(E^{-1})$.
%\mu^{-1}(E)$. 
Alors pour tout $g\in G$,
\[\nu(Eg)=\mu(g^{-1}E^{-1})=\mu(E^{-1})=\nu(E).\]
En général, les mesures de Haar invariantes à gauche ne sont pas invariantes à droite.
\end{rem}

De manière encore plus étonnante et remarquable, l'existence d'une mesure de Haar caractérise essentiellement les groupes localement compacts.

\begin{dfn}
Un groupe \emph{mesurable (au sens de Weil)} est un espace mesuré $(G,\cB,\mu)$ tel que
\begin{enumerate}
\item $\mu$ est $\sigma$-finie et non-identiquement nulle,
\item $G$ est un groupe,
\item la tribu $\cB$ et la mesure $\mu$ sont invariantes par translations à gauche,
\item la multiplication dans $G$ est mesurable. Plus précisément, la fonction qui à $(x,y)\in G\times G$ \mbox{associe~$x^{-1}y\in G$} est mesurable.
\end{enumerate}
\end{dfn}

\begin{dfn}
Un groupe mesuré $(G,\cB,\mu)$ est dit \emph{séparé (au sens de Weil)} si pour tout élément $g\in G$ différent de l'identité, il existe $E\in \cB$ de mesure finie strictement positive tel que la mesure de la différence symétrique est positive :
\[\mu(gE\bigtriangleup E)>0.\]
\end{dfn}

Lorsque $(G,\cB,\mu)$ est un groupe mesuré et séparé, il est possible de définir une topologie de groupe topologique sur $G$, la \emph{topologie de Weil} $\tau_{\trm{Weil}}$. En effet, sous ces hypothèses, la famille
$\cN$ des ensembles
\[\left \{ g\in G\,:\,\mu(gE\bigtriangleup E)<\ve \right \}\]
tels que $E\in \cB$ soit de mesure finie strictement positive, $\ve\in ]0,2\,\mu(E)[$,
définit un système de voisinages de l'identité dans $G$.

Le groupe topologique mesuré $(G,\tau_{\trm{Weil}},\cB, \mu)$ est un groupe \emph{localement borné}, où l'on dit que $E\in\cB$ est \emph{borné} si pour tout voisinage $I$ de l'identité, $E$ est recouvert par un nombre fini de translatés de $I$. 

Lorsque l'intérieur d'une partie $E\in \cB$ n'est pas vide, la mesure de $E$ est positive ; lorsque $E\in\cB$ est borné, la mesure de $E$ est finie. Donc $G$ est « presque » un groupe localement compact et $\mu$ sa mesure de Haar. Le « presque » est précisé par le théorème de Mackey et Weil \cite{mackey,weil} que nous allons énoncer.

\begin{dfn}
Un sous-groupe $G$ dans un groupe mesuré $(\hat{G},\hat{\cB},\hat{\mu})$ est dit \emph{épais} si pour tout $\hat{E}\in \hat{\cB}$,
\[\sup \left \{ \hat{\mu}(\hat{K})\,:\,\hat{K}\trm{ est un compact dans }\hat{G}
\trm{ contenu dans }\hat{E}-G\right \}=0.\]
\end{dfn}

\begin{thm}[Mackey~--~Weil]
\label{thm:mackey-weil}
Soit $G$ un groupe mesuré, séparé. Alors il existe un groupe localement compact $\hat{G}$ avec une mesure de Haar $\hat{\mu}$ sur la tribu de Borel
$\hat{\cB}$ de $\hat{G}$, tel que 
\begin{itemize}
\item $G$ soit un sous-groupe épais de $\hat{G}$, 
\item $\cB\supset \hat{\cB}\cap G$,
\item $\mu(E)=\hat{\mu}(\hat{E})$ lorsque $\hat{E}\in \hat{\cB}$ et $E=\hat{E}\cap G$.
\end{itemize}
\end{thm}

Pour une preuve de ce théorème, le lecteur pourra consulter \cite{weil}, aussi bien que le déjà cité \cite{halmos}.

\begin{rem}
Dans \cite{weil}, une version plus fine de ce théorème est montrée. En effet, on peut affaiblir les hypothèses du théorème \ref{thm:mackey-weil} pour montrer que le même résultat est valable lorsque l'on suppose seulement que la mesure sur $(G,\cB)$ est \emph{quasi-invariante}.
\end{rem}
\clearpage

%\cleardoublepage
\chapter{Le mouvement brownien : la fonction continue générique}
\label{sc:wiener}
Nos premiers vers dans l'étude de la dynamique générique selon le point de vue probabiliste transite naturellement par le \emph{mouvement brownien}, qui représente, dans les mathématiques modernes, l'idée de fonction continue choisie au hasard. Le problème de la définition mathématique de cet objet a nécessité beaucoup d'efforts, achevés principalement grâce aux contributions de Wiener \cite{wiener23}.

Nous présentons les outils probabilistes nécessaires à la compréhension de ce travail. Le lecteur pourra consulter le texte \cite{peres} (voir aussi \cite{revyor, karatzas}\,) comme référence plus que complète sur ce sujet.\footnote{Le lecteur francophone pourra également se faire du plaisir en consultant aussi \cite{levy}.}

\section{Notions de base de probabilité}
Un \emph{espace de probabilité} $(\Omega,\cB,\P)$ est un espace mesuré tel que la mesure totale $\P(\Omega)$ est égale à $1$ (voir \S\ref{s:basemesure} dans le chapitre précédent). Une \emph{variable aléatoire} $X$ sur $(\Omega,\cB,\P)$ est une
  fonction mesurable $X:(\Omega,\cB)\to (\R,\cB(\R))$. Nous noterons $\E[X]$ son espérance, $\E[X]=\int_\Omega X(\omega)\,d\P(\omega)$, et $\mrm{Var}(X)$ sa variance,
  $\mrm{Var}(X)=\E[X^2]-\E[X]^2$. Une variable aléatoire
   $X$ définit une mesure de probabilité $X_*\P$ sur $(\R,\cB(\R))$, que l'on appelle \emph{distribution} de la variable :~$X_*\P(A)=\P(X\in A)$, $A\in\cB(\R)$.

\begin{ex}
Une variable $X$ sur $(\Omega,\cB,\P)$ est dite \emph{gaussienne} si sa distribution est une loi gaussienne : autrement dit, il existe $\sigma>0$, $m\in \R$ tels que
\[\P(X\in A)=\tfrac{1}{\sqrt{2\pi \sigma}}\int_A\exp\left (-\tfrac{(t-m)^2}{2\sigma}\right )\,dt.\]
Les paramètres $m$ est $\sigma$ sont, respectivement, l'\emph{espérance} et la \emph{variance} de la variable $X$, et sa distribution est notée $\mathcal{N}(0,\sigma)$.

L'analyse de Fourier caractérise les variables de distribution gaussienne : une variable $X$ est une variable gaussienne d'espérance $m$ et variance $\sigma$ si et seulement si
\beqn{fourier}{\E\left [ e^{i\xi X}\right ]=\exp\left (i\xi m-\tfrac{\sigma}{2}\xi^2\right )\quad\trm{pour tout }\xi\in \R.}
\end{ex}

Une \emph{filtration} $\cF$ (d'une tribu $\cB$) est une suite monotone croissante de sous-tribus $(\cF_s)_{s\ge 0}$ : $\cF_s\subset\cF_t$ si $s<t$. Une filtration est dite \emph{régulière} si pour tout $s\ge 0$,
\[\bigcap_{t>s}\cF_t = \cF_s.\]
On dit qu'une filtration sur $(\Omega,\cB,\P)$ définit un \emph{espace de probabilité filtré}.

Un \emph{processus} $X=(X_s)_{s\ge 0}$ sur un espace de probabilité filtré $(\Omega,\cB,\P,\cF)$ est la donnée, pour chaque $s\ge 0$, d'une variable aléatoire $X_s:(\Omega,\cF_s,\P)\to (\R,\cB(\R))$. On dit aussi que le processus $X$ est \emph{adapté} à la filtration $\cF$. 

\begin{rem}
Tout processus $X$ est adapté par rapport à sa filtration \emph{naturelle} $(\sigma(X_s,s\le t))_{t\ge 0}$, où $\sigma(X_s, s\le t)$ est la plus petite tribu pour laquelle les variables $(X_s,s\le t)$ sont mesurables, à savoir engendrée par les pré-images~$X_s^{-1}(A)$ des ouverts $A\subset \R$.
\end{rem}

\begin{ex}
Un processus $X=(X_s)$ est dit \emph{gaussien}, si toute combinaison linéaire des $X_s$ est une variable gaussienne par rapport à la loi conjointe. Les processus gaussiens sont déterminés par les espérances et les fonctions de \emph{covariance} $\G(s,t)=\mrm{Cov}(X_s,X_t)=\E[X_sX_t]-\E[X_s]\,\E[X_t]$.
\end{ex}

\section{Mouvement brownien et régularité}
\label{ssc:bm}

\begin{dfn}
\label{dfn:bm}
Soit $(\Omega,\mathcal{B},\P)$ un espace de probabilité. Un \emph{mouvement brownien (standard)} sur l'intervalle $I=[0,1]$ est un processus $B=(B_t)_{t\in I}$ défini sur $\Omega$ qui satisfait les propriétés suivantes :
\begin{enumerate}

\item pour presque tout $\omega\in \Omega$, l'application $t\in I\mto B_t(\omega)$ est une fonction continue,
\item pour tout choix de $t_0=0<t_1<\ldots<t_{n-1}<t_{n}=1$, les variables aléatoires $B_{t_i}-B_{t_{i-1}}$ sont deux à deux indépendantes,
\item pour tout choix $s<t$ de temps dans $I$, la variable $B_t-B_s$ suit une loi gaussienne $\mathcal{N}(0,t-s)$, à savoir
\[\P(a<B_t-B_s<b)=\tfrac{1}{\sqrt{2\pi(t-s)}}\int_a^b\exp\left (-\tfrac{1}{2(t-s)}x^2\right )\,dx,\]
\item pour presque tout $\omega\in \Omega$, on a $B_0(\omega)=0$.
\end{enumerate}
\end{dfn}

Tout mouvement brownien $B$ induit une mesure de probabilité $\W:=B_*\P$, sur l'espace $C_0(I)$ des fonctions continues sur l'intervalle $I$ qui valent $0$ en $0$, muni de la tribu cylindrique $\mathcal{C}$ : elle est engendrée par les \emph{cylindres}~$C\in \mathcal{C}$, parties de la forme
\beqn{cylindre}{C=\{x\in C_0(I)\,|\, x(t_i)-x(t_{i-1})\in A_i,\,i=1,\ldots n\},}
avec $0=t_0<t_1<\ldots<t_{n-1}<t_{n}=1\trm{ et }A_i\in \cB(\R)$.
La mesure $\W$ est dite \emph{mesure de Wiener}. La mesure d'un cylindre est déterminée par les propriétés du mouvement brownien (définition \ref{dfn:bm}\,) : pour $C$ comme à la ligne \eqref{eq:cylindre},
\beqn{cylindre2}{\W(C)=\tfrac{1}{\sqrt{(2\pi)^n(t_{n}-t_{n-1})\cdots(t_1-t_0)}}\int_{A_1}\cdots\int_{A_n}\exp\left (-\tfrac{1}{2}\sum_{i=1}^n\frac{y_i^2}{t_i-t_{i-1}}\right )\,dy_1\cdots dy_n.}

\begin{thm}[Wiener]
\label{thm:wiener}
Les tribus cylindrique $\cC$ et borélienne $\cB(C_0(I))$ sur $C_0(I)$ coïncident. 
\end{thm}

Par conséquent, l'\emph{espace de Wiener (classique)} $(C_0(I),\cB(C_0(I)),\W)$ représente le choix \emph{intrinsèque} d'espace de probabilité sur lequel définir concrètement le mouvement brownien : pour $x\in C_0(I)$, $B_t$ est la variable aléatoire qui correspond à la valeur de la fonction $x$ au temps $t$,
\[B_t(x):=x(t).\]

\begin{cor}
Le mouvement brownien existe et il est uniquement déterminé par ses propriétés.
\end{cor}

Nous choisirons comme réalisation des trajectoires browniennes, celle du mouvement brownien donné par la mesure de Wiener. 

Nous verrons que la mesure de Wiener est un candidat naturel pour remplacer la mesure de Lebesgue sur le groupe abélien de dimension infinie.

\begin{rem}\label{rem:covariance}
Rappelons que le mouvement brownien est un processus gaussien centré et par conséquent caractérisé par la famille des covariances $\G(s,t)$.
Dans le cas du mouvement brownien, on a $\G(s,t)=\min \{s,t\}$, pour tous \mbox{$s$, $t$}.
\end{rem}

\section{Propriétés génériques de régularité}

Il s'avère que les trajectoires du mouvement brownien ne sont pas différentiables, cependant selle sont un peu plus que continues. Plus précisément nous avons les résultats suivants (voir \cite[chapitres 1 et 10]{peres}\,) :

\begin{thm}[Différentiabilité en aucun point]
Presque sûrement, le mouvement brownien n'est différentiable en aucun point $t\in [0,1]$.
\end{thm}

\begin{thm}[Variation infinie]
\label{thm:variationbrownien}
Presque sûrement, la variation (voir \eqref{eq:variation}) du mouvement brownien est infinie.
\end{thm}

\begin{thm}[Exposant d'Hölder $1/2$]
\label{thm:bmholder}
~

\begin{enumerate}
\item Presque sûrement, le mouvement brownien est une fonction $\alpha$-höldérienne, pour tout $\a<1/2$.
\item Presque sûrement, le mouvement brownien est une fonction qui n'est pas $\alpha$-höldérienne en aucun point, pour tout $\a>1/2$.
\end{enumerate}
\end{thm}

\begin{rem}
Pour compléter l'énoncé du théorème \ref{thm:bmholder}, il est intéressant de savoir qu'il existe presque sûrement des points \emph{lents}, auxquels le mouvement brownien est exactement $1/2$-höldérien, mais ils sont rares (au sens de la dimension de Hausdorff).
\end{rem}

\begin{thm}[Module de continuité de Lévy]\label{thm:modcontinuite}
Soit $w(t)=\sqrt{2t|\log t|}$. Soit $B$ un mouvement brownien standard. On a
\beqn{levy0}{\P\left ( \lim_{\ve\rightarrow 0} \sup_{\substack{
s<t \le 1,\\ t-s\le \ve}
} \frac{|B_t-B_s|}{w(\ve)}=1 \right )=1.}
\end{thm}

L'expression \eqref{eq:levy0} dit en particulier que le mouvement brownien est une fonction avec \emph{module de continuité} $w$. Pour des rappels sur la notion de module de continuité, voir la partie \ref{sssc:modulecontinuite} du chapitre \ref{chapter:1}.

\smallskip

Essayons d'expliquer pourquoi l'exposant critique est exactement $1/2$ : cela provient de la \emph{nature gaussienne} du mouvement brownien. Pour cela nous devons rappeler le \emph{critère de Kolmogorov-\v{C}entsov} (voir \cite[Chapter 1, (1.8)]{rev-yor}\,) :
\begin{thm}
Soit $X=(X_s)_{s\in [0,1]}$ un processus sur un espace de probabilité filtré $(\Omega,\cB,\P,\cF)$ tel qu'il existe des constantes $\alpha$, $\beta$ et $C>0$ telles que pour tout $0\le s<t\le 1$ on ait
\beqn{kc}{\E|X_t-X_s|^\alpha\le C\,|t-s|^{1+\beta}.}
Alors il existe une \emph{modification} $\tilde X=(\tilde X_s)_{s\in [0,1]}$ de $X$ (à savoir un processus $\tilde X$ tel que pour tout $s\in [0,1]$, $\tilde X_s=X_s$ presque sûrement) qui est $\gamma$-hölderienne pour tout $\gamma<\beta/\alpha$.
\end{thm}

En effet, si $X$ est une variable gaussienne centrée de variance $t-s>0$, alors $Y=\tfrac{1}{\sqrt{t-s}}X$ est une variable gaussienne de variance $1$, donc $C=\E|Y|^{\alpha}$ ne dépend pas de $t-s$ et on peut écrire
\[\E|X|^{\a}=C\,(t-s)^{\alpha/2}.\]
Donc pour tout $\alpha>2$, la condition \eqref{eq:kc} est vérifiée pour $B=(B_s)_{s\in[0,1]}$, avec  $\beta=\tfrac\a2-1$ et le critère de Kolmogorov-\v{C}entsov implique qu'il existe une modification de $B$ qui est $\gamma$-hölderienne pour tout $\gamma<\tfrac12 -\tfrac1\a$. En prenant $\a$ arbitrairement grand, on retrouve l'exposant critique $1/2$.

\section{L'intégrale de Paley-Wiener}
Il est souvent pratique, et il le sera en effet pour nous, de pouvoir définir en quelque sorte l'intégrale d'une fonction par rapport au mouvement brownien. Si $f$ est une fonction dans $L^2(I)=L^2(I,\Leb)$, ceci se fait sans trop de complications au sens de Stieltjes. Si $f=\mathbf{1}_{[s,t[}$, on pose $\int_0^1\mathbf{1}_{[s,t[}(u)\,dB_u:=B(t)-B(s)$. On étend ensuite la fonctionnelle $\int_0^1\,\cdot\,dB_u$ par linéarité à toutes les fonctions qui s'écrivent sous la forme
\beqn{simple}{f=\sum_{i=1}^ka_i\,\mathbf{1}_{[s_i,t_i[},\quad a_i\in\R.}
On voit alors que pour presque tout $\omega\in \Omega$, $\int_0^1\,\cdot\,dB_u(\omega)$ définit une fonctionnelle linéaire continue (bornée) sur une partie dense dans $L^2(I)$, et qu'elle possède donc une extension à $L^2(I)$ entier.

De plus, pour tout $f$ comme dans \eqref{eq:simple}, $\int_0^1f\,dB$ est une somme de variables aléatoires gaussiennes, donc variable gaussienne elle même. Il est classique qu'une limite de variables gaussiennes est encore gaussienne, donc pour tout $f\in L^2(I)$,
\beqn{integral_wiener}{\int_0^1f\,dB}
est une variable gaussienne centrée, de variance $\int_0^1f^2$ (ceci peut se voir grâce à la caractérisation \eqref{eq:fourier}\,). La variable~$\int_0^1f\,dB$ est connue sous le nom d'\emph{intégrale de Wiener}.

\section{L'espace de Wiener et le théorème de Cameron-Martin}
\label{ssc:espace-wiener}
En suivant la présentation de \cite[\S 1.4]{peres}, nous définissons l'\emph{espace de Dirichlet} $\bD_0(I)$ des fonctions absolument continues $x\in C_0(I)$ avec dérivée dans $L^2(I)$ : il existe $\xi\in L^2(I)$ tel que pour presque tout $t$, $x(t)=\int_0^t\xi(s)\,ds$.

Cet espace s'insère naturellement dans le contexte de la mesure de Wiener, puisqu'il représente exactement le sous-espace de $C_0(I)$ des fonctions $x$ qui préservent la classe de la mesure de Wiener sous l'action de translation~$T_x(y):=y+x$. Plus précisément :

\begin{thm}[Cameron~--~Martin \cite{cameron-martin}]
\label{thm:cameron-martin0}
Soit $x$ une fonction dans $C_0(I)$. Alors la classe de la mesure de Wiener $\W$ est préservée par $T_x$ si et seulement si $x\in\bD_0(I)$ et, sous cette hypothèse, le cocycle de Radon-Nikodym\footnote{voir la partie \ref{ssc:actions} pour la définition.} est
\beqn{cameron-martin0}{\frac{d(T_x)_*\W}{d\W}(y)=\exp\left \{-\frac{1}{2}\int_0^1x'(t)^2\,dt+\int_0^1x'(s)\,dy(s)\right \},\quad y\in C_0(I).}
\end{thm}

L'expression $\int_0^1x'(s)\,dy(s)$ dans \eqref{eq:cameron-martin0} dénote l'intégrale de Wiener. Pour montrer ce théorème, nous allons d'abord donner une construction alternative de la mesure de Wiener.

Considérons l'espace de Hilbert réel séparable $L^2(I)$ et fixons $(g_n)_{n\in \N}$ une base orthonormée.

\begin{ex}
En suivant Wiener \cite{wiener}, on pourra prendre $g_0=1$ et, pour $n>0$,
\beqn{base-fourier}{g_n(t)=\sqrt{2}\cos (\pi nt).}
Ce choix, naturel d'un point de vue de l'analyse de Fourier, est cependant moins adapté quand on veut faire des calculs. Un choix bien meilleur, comme Lévy le fit \cite{levy}, est la \emph{base de Haar} : pour $m\ge 1$ et $1\le k \le 2^{m-1}$ on définit
\[h_{m,k}:=2^\frac{m-1}{2}\left (\mathbf{1}_{\left [\frac{2k-2}{2^m},\frac{2k-1}{2^m}\right ]}-\mathbf{1}_{\left [\frac{2k-1}{2^m},\frac{2k}{2^m}\right ]}\right )\]
et on pose $g_0=1$ et $g_n=h_{m,k}$, pour $n=2^{m-1}-1+k$. En fait, nous n'utiliserons pas la forme explicite des $g_n$ dans ce qui va suivre.
\end{ex}

L'espace $\bD_0(I)$ est un espace de Hilbert réel lorsque l'on définit le produit scalaire
\[\langle x,y\rangle_{\bD_0(I)}:=\langle x',y'\rangle_{L^2}.\]
Une base orthonormée $\Phi_n$ est définie par intégration des $g_n$ :
\[\Phi_n(t):=\int_0^tg_n(s)\,ds.\]

Observons que pour toutes fonctions $x$ et $y$ dans $L^2(I)$, l'inégalité de Cauchy-Schwarz donne
\[\left \vert \int_0^t\left (x(s)-y(s)\right )\,ds\right \vert \le \|x-y\|_{L^1} \le \|x-y\|_{L^2},\]
donc si $\lim_{n\rightarrow\infty}\|x_n-x\|_{\bD_0(I)}=0$ alors la suite $(x_n)$ tend uniformément vers $x$. Par conséquent, pour tout $x\in \bD_0(I)$, la série
\[x=\sum_{n\in\N} \langle x',g_n\rangle_{L^2} \,\Phi_n\]
converge dans $\bD_0(I)$ et donc uniformément.

Soit $(Z_n)$ une suite de variables aléatoires gaussiennes centrées de variance~$1$, indépendantes, identiquement distribuées. Soit $x$ la somme de la série aléatoire de fonctions
\beqn{gaussian_coefficients}{x=\sum_{n\in\N}Z_n\,\Phi_n.}
Il est possible de montrer que $x$ est presque sûrement une fonction continue (et il s'agit du point où le choix de la base $(g_n)$ peut simplifier la preuve, voir \cite{peres}\,).

En utilisant la remarque \ref{rem:covariance}, il est alors clair que $x$ défini par \eqref{eq:gaussian_coefficients} est un mouvement brownien : il s'agit d'un processus gaussien centré, de covariance
\begin{align*}
\G(s,t):=\,&\mrm{Cov}\left (x(s),x(t)\right )=\E \left [\sum_{n\in \N} Z_n^2 \Phi_n(t)\,\Phi_n(s)\right ]\\
=\,&\sum_{n\in\N}\langle  \mathbf{1}_{[0,t]},g_n\rangle^2_{L^2}\,\langle  \mathbf{1}_{[0,s]},g_n\rangle^2_{L^2}=\langle  \mathbf{1}_{[0,s]},\mathbf{1}_{[0,t]}\rangle^2_{L^2}=\min\left \{s,t\right \}.
\end{align*}
Si l'on note $\g$ la mesure gaussienne sur $\R$, $d\gamma(x)=\tfrac{1}{\sqrt{2\pi}}e^{-\frac{1}{2}x^2}\,dx$, on peut ainsi résumer les arguments précédents :
\begin{thm}\label{thm:rep_gaus}
Soit $(g_n)_{n\in\N}$ une base de l'espace $L^2(I)$. Étant donné un mouvement brownien $x$ sur l'intervalle $I$ et une fonction $g\in L^2(I)$, notons $\langle x',g\rangle_{L^2}$ l'intégrale de Wiener $\int_0^1 g\,dx$ définie à la ligne \eqref{eq:integral_wiener}. L'application
\beqn{rep_gaus}{\rdfcn{C_0(I)}{\R^{\N}}{x}{\left (\langle x',g_n \rangle_{L^2}\right )_{n\in\N}}}
définit un isomorphisme (linéaire) mesurable entre l'espace de Wiener classique $(C_0(I),\cB(C_0(I)),\W)$ et l'espace vectoriel gaussien de dimension infinie $(\R^{\N},\cB(\R^{\N}),\g^{\otimes\N})$.
\end{thm}

\begin{proof}[Preuve du théorème de Cameron-Martin~\ref{thm:cameron-martin0}] Soit $x\in \bD_0(I)$, on va montrer que la mesure $\W$ est absolument continue par rapport à la mesure $(T_x)_*\W$. L'image de l'espace de Dirichlet $\bD_0(I)$ est l'espace $\ell^2(\N)\subset\R^\N$ et comme l'isomorphisme \eqref{eq:rep_gaus} est linéaire, l'action induite par $T_x$ sur $(\R^{\N},\cB(\R^{\N}),\g^{\otimes\N})$ est également une action de translation $T_{\mathbf{a}}$, $\mathbf{a}=(a_n)\in\ell^2(\N)$. Donc il nous faut montrer que la mesure gaussienne $\g^{\otimes\N}$ est absolument continue par rapport à $(T_{\mathbf{a}})_*\g^{\otimes\N}$.

Pour calculer la dérivée de Radon-Nikodym de $(T_{\mathbf{a}})_*\g^{\otimes\N}$ par rapport à $\g^{\otimes\N}$ on procède par approximation, en vertu de la proposition~\ref{prop:approx} du chapitre~\ref{sc:localement}. Soit $N\in \N$ assez grand, on pose
\[\mathbf{a}_N:=(a_0,\ldots,a_N,0,\ldots)\]
la projection de $\mathbf{a}$ sur le sous-espace engendré par les premiers $N+1$ vecteurs de la base canonique de $\ell^2(\N)$ (la suite $\mathbf{a}_N$ est égale à $\mathbf{a}$ dans les premières $N+1$ cordonnées et identiquement nulle après). La transformation $T_{\mathbf{a}_N}$ agit uniquement sur les premiers $N+1$ facteurs de l'espace produit $\R^\N$ et par conséquent on peut calculer aisément la densité de  $(T_{\mathbf{a}_N})_*\g^{\otimes\N}$ par rapport à $\g^{\otimes\N}$ :
\[\frac{d(T_{\mathbf{a}_N})_*\g^{\otimes\N}}{d\g^{\otimes\N}}
((z_n))=
\prod_{n=0}^{N}
\tfrac{\exp\left (-(z_n-a_n)^2/2\right )}{\exp\left (-z_n^2/2\right )}=
\exp\left (\sum_{n=0}^Na_n\,z_n-\tfrac{1}{2}\sum_{n=0}^Na_n^2\right ).\]

Puisque l'on a fait l'hypothèse que la suite $(a_n)_{n\in \N}$ appartient à l'espace $\ell^2(\N)$, la limite sur $N$ de la suite~$\frac{d(T_{\mathbf{a}_N})_*\g^{\otimes\N}}{d\g^{\otimes\N}}$ existe presque sûrement et est égale à
\[\exp\left (\langle \mathbf{a},(z_n)\rangle_{\ell^2}-\tfrac{1}{2}\|\mathbf{a}\|^2_{\ell^2}\right )\]
(la convergence presque sûre de $\sum_{n\in\N}a_nz_n$ vers $\langle \mathbf{a},(z_n)\rangle_{\ell^2}$ est, par dualité de Fourier, l'existence presque sûre de l'intégrale de Wiener). D'après la proposition~\ref{prop:approx} du chapitre~\ref{sc:localement}, la limite de dérivées de Radon-Nikodym est la dérivée de Radon-Nikodym de la transformation $T_\mathbf{a}$.

En utilisant l'inverse de l'isomorphisme \eqref{eq:rep_gaus} et l'identité de Parseval, on en conclut (cf. avec \eqref{eq:cameron-martin0} dans l'énoncé)
\[\frac{d(T_x)_*\W}{d\W}(y)=\exp\left (\langle x',y'\rangle_{L^2}-\tfrac{1}{2}\|x'\|^2_{L^2}\right )=\exp\left (\langle x,y\rangle_{\bD_0(I)}-\tfrac{1}{2}\,\|x\|^2_{\bD_0(I)}\right ).\]

Nous avons aussi la relation symétrique $(T_x)_*\W\preceq\W$ puisque la mesure $(T_x)_*\left ((T_{-x})_*\W\right )$ est absolument continue par rapport à la mesure $(T_{-x})_*\W$.
\end{proof}

\section{Le pont brownien}
Il est aussi possible de définir le mouvement brownien sur le cercle $\T$ (i.e.~comme fonction $1$-périodique aléatoire). Dans ce cas on parle de \emph{pont brownien}. Après avoir identifié le cercle au tore $\R/\Z$, l'espace $C_0(\T)$ des fonctions continues sur le cercle qui prennent la valeur $0$ en $0$ peut se voir comme l'espace des fonctions continues sur $\R$ de période $1$ qui valent $0$ en $0$. Puisque l'intervalle $[0,1]$ est un domaine fondamental pour l'action de translation de $\Z$, on pourra écrire
\[C_0(\T)=\left \{ x\in C_0(I)\,:\, x(0)=x(1)=0\right \}\]
et on ne fera pas de distinction entre les domaines de définition $I$ ou $\T$.

On remarquera que l'application $\pi$ définie par
\beqn{proj_pont}{y\in C_0(I)\mto\left \{ x:\,t\mto y(t)-t\,y(1)\right \}\in C_0(\T)}
est une projection (elle est linéaire et vérifie $\pi^2=\pi$) continue et on peut définir la \emph{mesure de Wiener} sur $C_0(\T)$ par
\[\W_0:=\pi_*\W.\]

Le processus canonique $x=\left (x(t)\right )_{t\in \T}$ sur $\left (C_0(\T),\cB\left (C_0(\T)\right ),\W_0\right )$, définit un processus gaussien que l'on appelle \emph{pont brownien} et porte comme covariances
\beqn{cov_pont}{\G(s,t)=\min\{s,t\}-s\,t.}

Puisque la projection $\pi$ modifie $y$ d'une fonction linéaire, les propriétés de régularité presque sûres du pont brownien sont les mêmes que celles du mouvement brownien.

\begin{rem}
Dans la littérature, le pont brownien est aussi présenté en conditionnant le mouvement brownien à revenir en $0$ au temps $1$. On retrouve un processus gaussien, dont les fonctions de covariance sont données par~\eqref{eq:cov_pont}.
\end{rem}

\begin{rem}
Une représentation gaussienne comme \eqref{eq:rep_gaus} est aussi possible pour définir le pont brownien : il suffit de modifier les fonctions $\Phi_n\in\mathbf{D}_0(I)$ en posant (confronter avec \eqref{eq:proj_pont}\,)
\[\tilde{\Phi}_n(t):=\Phi_n(t)-t\,\Phi_n(1)\in\mathbf{D}_0(\T).\]
On remarquera que $\tilde{\Phi}_0\equiv 0$, et on doit donc retirer $\tilde{\Phi}_0$ de la base. 
En posant $\tilde{x}(t)=\sum_{n\in \N^*}Z_n\,\tilde{\Phi}_n(t)$, on a alors
\[ \tilde{\G}(s,t):=\mrm{Cov}\left (\tilde{x}(s),\tilde{x}(t)\right )=\E\sum_{n\in \N^*}Z^2_n\,\tilde{\Phi}_n(t)\,\tilde{\Phi}_n(s)=\min\{s,t\}-s\,t.\]

Le théorème de Cameron-Martin pour la mesure de Wiener sur $C_0(\T)$ s'énonce ainsi :

\begin{thm}
Soit $x$ une fonction dans $C_0(\T)$. Alors la classe de la mesure de Wiener $\W_0$ est préservée par $T_x$ si et seulement si $x\in\bD_0(\T)$ et, sous cette hypothèse, le cocycle de Radon-Nikodym est
\[\frac{d(T_x)_*\W_0}{d\W_0}(y)=\exp\left ( \langle x,y\rangle_{\bD_0(\T)}-\tfrac{1}{2}\,\|x\|^2_{\bD_0(\T)}\right ),\quad y\in C_0(\T).\]
\end{thm}
\end{rem}
Le phénomène que l'on observe avec le théorème de Cameron-Martin \ref{thm:cameron-martin0} est par conséquent presque optimal : puisque $C_0(I)$ est un groupe qui n'est pas localement compact, il n'y a aucun espoir de définir une mesure quasi-invariante, non seulement sur $C_0(I)$, mais sur un sous-espace quelconque de $C_0(I)$ de mesure strictement positive. Cependant, la mesure de Wiener est quasi-invariante par l'action d'un sous-groupe dense, qui n'est pas localement compact, et en même temps de mesure nulle !

\section{Un deuxième théorème de Cameron-Martin}
\label{sc:nonlineaire}

Terminons cette introduction au mouvement brownien par une partie, plus technique, qui est à la base de la quasi-invariance des mesures sur les groupes de difféomorphismes : il s'agit d'étendre le théorème de Cameron-Martin~\ref{thm:cameron-martin0} à une classe de transformations non-linéaires de l'espace de Wiener $C_0(I)$. 

Pour cela, reprenons l'isomorphisme~\eqref{eq:rep_gaus} entre l'espace de Wiener classique $(C_0(I),\W)$ et l'espace vectoriel gaussien de dimension infinie $(\R^\N,\gamma^{\otimes\N})$. Supposons qu'au lieu d'une translation $T_\mathbf{a}:(z_n)_{n\in\N}\mto (z_n+a_n)_{n\in\N}$, avec $\mathbf{a}=(a_n)\in \ell^2(\N)$, nous considérons maintenant une \emph{translation non-linéaire}
\[T_A=I+A:\mathbf{z}=(z_n)\in\R^\N\mto (z_n+a_n(\mathbf{z}))\in\R^\N,\]
qu'on supposera par simplicité inversible presque partout.

Rappelons que l'hypothèse que $\mathbf{a}$ soit de carré intégrable est fondamentale pour garantir la quasi-invariance de la mesure gaussienne $\g^{\otimes\N}$. On supposera ainsi que la translation non-linéaire $T_A=I+A$ est dans la direction de l'espace de Cameron-Martin, ce qui signifie simplement que l'image de $A$ est contenue dans le sous-espace $\ell^2(\N)$.

Essayons de deviner comment cette translation opère-t-elle sur la mesure gaussienne $\g^{\otimes\N}$ en procédant par approximations finies, comme dans le théorème de Cameron-Martin~\ref{thm:cameron-martin0}.

Pour $N\in\N$ fixé, on définit l'opérateur de projection $\pi_N:\R^\N\to\R^N$ dans le sous-espace vectoriel engendré par les $N+1$ premiers vecteurs de la base standard de $\R^\N$. Notons $A_N=\pi_NA\pi_N$ et par suite $T_{A_N}=I+A_N$, dont la restriction à $\ell^2(\N)$ est une perturbation de rang fini de l'identité.

La densité de la mesure $(T_{A_N} )_*\g^{\otimes\N}$ par rapport à $\g^{\otimes\N}$ est obtenue alors par simple changement de variable :
\beqn{changevar}{
\dfrac{d(T_{A_N})_*\g^{\otimes\N}}{d\g^{\otimes\N}}(T_{A_N}(\mathbf{z}))=\frac{1}{|D_{\mathbf{z}}(I+A_N)|}\,\exp\left (\langle A_N(\mathbf{z}),\mathbf{z}\rangle_{\ell^2}+\tfrac12\|A_N(\mathbf{z})\|_{\ell^2}^2\right ),
}
où $D_{\mathbf{z}}(I+A_N)=I+D_{\mathbf{z}}A_N$ dénote la dérivée jacobienne de $I+A_N$ au point $z\in\R^\N$. Bien évidemment, pour que cela ait du sens, on requiert une hypothèse de différentiabilité pour $A$, que l'on explicitera un peu plus loin.

Observons d'abord que, puisque l'image de l'opérateur $A$ est incluse dans le sous-espace $\ell^2(\N)$, on peut facilement en déduire le candidat pour la limite de l'expression dans l'exponentielle dans \eqref{eq:changevar} (en négligeant les problèmes dans la définition d'une << intégrale stochastique >> $\langle A(\mathbf{z}),\mathbf{z}\rangle_{\ell^2}$).

Regardons la suite des déterminants jacobiens $|D_{\mathbf{z}}(I+A_N)|$. Pour avoir que la limite d'une telle suite existe, il faut imposer des conditions sur l'opérateur $A$. Rappelons qu'un \emph{opérateur compact} $K$ de $\ell^2(\N)$ dans $\ell^2(\N)$ est diagonalisable, son spectre $\sigma(K)=(\lambda_k)_{k\in\N}$ est dénombrable et la multiplicité $\nu(\lambda)$ de toute valeur propre $\la$ non-nulle est finie. Un opérateur compact $K$ est de \emph{classe trace}, si la suite des valeurs propres, considérée avec multiplicité, est sommable. On peut alors définir, pour tout opérateur de classe trace, le \emph{déterminant de Fredholm}
\beqn{fredholm}{\det(I+K)=\prod_{\lambda\in\sigma(K)}(1+\lambda)^{\nu(\lambda)},}
qui naturellement généralise le déterminant usuel (le nom est justifié par le fait que $I+K$ est un opérateur de Fredholm).

\begin{rem}
D'après Ramer \cite{ramer}, le bon cadre pour définir un déterminant en dimension infinie est constitué des \emph{opérateurs de Hilbert-Schmidt}, à savoir les opérateurs compacts dont le spectre est de carré intégrable. Dans ce cas il est possible de définir le \emph{déterminant régularisé de Fredholm-Carleman}
\[\mrm{det}_2(I+K)=\prod_{\lambda\in\sigma(K)}(1+\lambda)^{\nu(\lambda)}e^{-\nu(\lambda)\lambda}.\]
\end{rem}

En revenant à l'opérateur $A$, on souhaite que la limite des jacobiens $|I+D_{\mathbf{z}}A_N|$ soit de la forme \eqref{eq:fredholm}. On fait donc l'hypothèse que pour tout $\mathbf{z}\in\R^\N$, la restriction de $D_{\mathbf{z}}A_N$ à $\ell^2(\N)$ converge vers un opérateur compact de classe trace dans la topologie de la norme, que l'on note $D_{\mathbf{z}}A$ (rappelons qu'une limite d'opérateurs de rang fini est nécessairement compact).

En fait, on peut vérifier que l'opérateur $D_{\mathbf{z}}A:\ell^2(\N)\to\ell^2(\N)$ est défini comme une différentielle directionnelle usuelle : pour tout $\mathbf{v}\in \ell^2(\N)$, on a
\[
D_{\mathbf{z}}A(\mathbf v)=\lim_{\ve\rightarrow 0}\dfrac{A(\mathbf{z}+\ve\,\mathbf{v})-A(\mathbf{z})}{\ve}.
\]

On espère avoir motivé ainsi la formule de changement de la mesure de Wiener sous des transformations non-linéaires, due originairement á Cameron-Martin \cite{cameron-martin2}, puis largement améliorée par des travaux successifs, en commençant par Ramer \cite{ramer} (voir aussi les livres \cite[\S6.6]{bogachev-gaussian} ou \cite{zakai}). Le résultat s'énonce ainsi :

\begin{thm}\label{thm:ramer}
Soit $K:C_0(I)\to \mathbf{D}_0(I)$ un opérateur non-linéaire différentiable dans la direction de $\mathbf{D}_0(I)$ et tel que en tout point $x\in C_0(I)$, la différentielle $D_xK$ définie sur $\mathbf{D}_0(I)$ est un opérateur de Hilbert-Schmidt. Supposons que la transformation $T=I+K$ soit presque partout inversible. Alors la mesure $T_*\W$ est absolument continue par rapport à la mesure $\W$ et la dérivée de Radon-Nikodym est
\[
\frac{dT_*\W}{\W}(T(x))=\frac{1}{\mrm{det}_2(I+D_xK)}\exp\left (\langle K(x),x\rangle_{\mathbf{D}_0(I)}+\tfrac12\|K(x)\|^2_{\mathbf{D}_0(I)}\right ).
\]
\end{thm}

Bien évidemment, étendre ce résultat dans le cas de transformations de la mesure de Wiener $\W_0$ sur $C_0(\T)$ ne présente aucune difficulté.
\clearpage

\chapter{Difféomorphismes du cercle}
\label{sc:diffeos}
%!TEX root = master.tex

Il ne sera pas possible de résumer en quelques pages la théorie des difféomorphismes du cercle : leur étude, inaugurée par Poincaré dans son célèbre article \cite{poincare}, constitue une partie remarquable dans la théorie des systèmes dynamiques, et des mathématiques en général. Nous renvoyons le lecteur aux textes généraux
\mbox{\cite[chapitres 11 et 12]{katok-hasselblatt}}, \cite[chapitre I]{demelo-strien}, et plus spécifiques : \cite{diviseurs,herman} -- pour les aspects analytiques~-- et \cite{ghys-circle,navasEM,booknavas} -- pour une approche plus « géométrique ».

Nous nous limiterons ici à introduire les notions et résultats fondamentaux qui seront utilisés dans la suite de ce travail.

\section{Poincaré et le nombre de rotation}
\label{ssc:poincare}

Soit $f$ un \emph{homéomorphisme} du cercle $\T$, que l'on voit comme le quotient de $\R$ par la translation $x\mto x+1$. Il existe alors un homéomorphisme $F$ de $\R$ dans $\R$ qui vérifie $F(x+1)=F(x)+1$ pour tout $x\in \R$ et $F(x)=f(x)\pmod{1}$. On dira que $F$ \emph{relève} l'homéomorphisme $f$.\footnote{On supposera toujours que $f$ est un homéomorphisme qui préserve l'orientation de $\T$.}

De manière réciproque, tout homéomorphisme $F:\R\to \R$ qui commute avec la translation par $1$, définit par réduction$\pmod{1}$ un homéomorphisme du cercle $\bar{F}$.

En termes plus savants, le morphisme $\pi$ de revêtement universel
\[\xymatrix{\mathbf{0}\ar[r] & \Z\ar[r] & \R\ar[r]^\pi & \T\ar[r]& \mathbf{1}}\]
induit le morphisme de revêtement universel
\[\xymatrix{\mathbf{0}\ar[r] & \Z\ar[r] & \Homt \ar[r]^\pi& \Hom\ar[r]& \mathbf{1}}\]
et $\Homt $ s'identifie aux fonctions $\{F\in \mrm{Hom\acute{e}o}_+(\R)\,:\, F(x+1)=F(x)+1\}$, qui est un groupe topologique lorsqu'on le munit de la topologie induite par la norme $\C{0}$ : un homéomorphisme $F$ est $\ve$-proche de l'identité $id$ si
\[\|F-id\|_{\hspace{.2pt}0}:=\sup_{x\in[0,1[}|F(x)-x|<\ve.\]

\medskip

Parfois nous ne faisons pas de distinction entre un homéomorphisme du cercle et ses relevés. Pour limiter cette ambiguïté, nous supposerons souvent que l'homéomorphisme $F$ qui relève $f\in \Hom$ sera déterminé par la condition $F(0)\in [0,1[$.

\begin{rem}
\label{rem:nombrerotation}
Si $F_1$ et $F_2$ relèvent $f$, leur différence $F_2-F_1$ se projette $\pmod{1}$ sur la fonction identiquement nulle. Deux relevés diffèrent ainsi d'un nombre entier.
\end{rem}

\bigskip

\begin{prop}[Poincaré]
\label{prop:nombrerotation}
Pour tout $f\in \Hom$ donné et pour tout choix du relevé $F$, il existe un unique $\alpha\in \R$ tel que pour tout $x\in \R$
\[|F^{\hspace{0.3pt}n}(x)-x-n\alpha|<1.\]

\end{prop}

En d'autres termes, il existe une unique translation $R_\alpha$ sur $\R$ telle que la fonction $F$ reste à distance $\C{0}$ bornée de $R_\alpha$, ainsi que toutes ses itérées $F^{\hspace{0.3pt}n}$ de $R_{n\alpha}$, avec une borne ne dépendant pas de $n$. Cette remarque justifie la terminologie : le nombre $\alpha$ donné par la proposition \ref{prop:nombrerotation} est appelé le \emph{nombre de translation} de $F$.

Il suit de la remarque \ref{rem:nombrerotation} que le nombre de translation d'un homéomorphisme $f$ est bien défini $\pmod{1}$ : l'application $\rho$ qui à $F\in \Homt $ associe son nombre de translation $\alpha=\rho(F)$ induit l'application
\[
\setlength{\arraycolsep}{1.5pt}\begin{array}{ccccc}
\bar{\rho}:& \Hom & \to  & \T      &\\
           & f    & \mto & \rho(F) & \pmod{1}
\end{array}
\]
qu'on appelle \emph{application de nombre de rotation} et qu'on notera, avec un abus de notation, plus souvent par $\rho$.

Lorsqu'on munit le groupe $\Hom$ de sa topologie naturelle $\C{0}$, l'application $\rho$ est une fonction continue : elle est la limite uniforme de la suite de fonctions continues 
\[\frac{1}{n}(F^{\hspace{0.4pt}n}(x)-x)\pmod{1}.\]

\medskip

L'application nombre de rotation donne beaucoup d'informations sur la dynamique induite par un homéomorphisme $f$ : son nombre de rotation $\rho(f)$ est rationnel si et seulement si $f$ possède une orbite périodique. Plus précisément, si $\rho(f)=\frac{p}{q}$, alors toute orbite périodique de $f$ est de période $q$.

\begin{prop}
Soient $f$ et $g$ deux homéomorphismes du cercle tels qu'il existe $h$, fonction sur $\T$ continue et croissante\footnote{On dit qu'une fonction $h$ sur le cercle est croissante si les fonctions $H$ qui relèvent $h$ à $\R$ sont croissantes.} qui \emph{semi-conjugue} $f$ à $g$,
\[h\circ f=g\circ h.\]
Alors, les nombres de rotations de $f$ et $g$ coïncident :
\[\rho(f)=\rho(g).\]
\end{prop}

Quand l'homéomorphisme $f$ ne possède pas d'orbite périodique, deux cas peuvent se produire. La première possibilité est que $f$ soit un homéomorphisme \emph{minimal}, à savoir que toute orbite est dense : dans ce cas il existe un homéomorphisme $h$, unique à rotation près, qui \emph{conjugue} $f$ à la rotation d'angle $\alpha=\rho(f)$,
\[h\circ f=h+\alpha.\]

Sinon, les orbites s'accumulent sur un unique ensemble minimal invariant $K=K_f$, qui est homéomorphe à un Cantor. Dans ce cas, il existe une fonction continue et croissante $h$ qui \emph{semi-conjugue} $f$ à la rotation d'angle~$\alpha=\rho(f)$,
\[h\circ f=h+\alpha.\]
Ce cas est dit \emph{exceptionnel} et l'ensemble minimal $K$ est également dit exceptionnel. Ce terme dérive du fait que, selon Poincaré, de tels exemples devaient être bien rares... Hélas, Poincaré n'a pas vécu assez longtemps pour que Denjoy \cite{denjoy} lui confirme et contredise à la fois cette affirmation : il construit des exemples, que l'on appelle aujourd'hui \emph{de Denjoy}, d'homéomorphismes et difféomorphisme $\C{1}$ qui sont exceptionnels\footnote{La communauté hésite entre l'emploi du terme \emph{exemple} ou \emph{contre-exemple}. Peut-être, parler de \emph{Cantor-exemples} surmonterait ce problème...}. D'un autre coté, il montre que la vision de Poincaré n'était pas entièrement erronée :

\begin{thm}[Denjoy]
\label{thm:denjoy}
Tout difféomorphisme du cercle de classe $\C{1+vb}$ et de nombre de rotation irrationnel est minimal.
\end{thm}

Rappelons qu'une fonction $f$ est dite de classe $\C{1+vb}$ si elle est de classe $\C{1}$ et si la dérivée $f'$ est une fonction à \emph{variation bornée} :
\beqn{variation}{\mrm{Var}(f')=\sup_{\Pi}\sum_{\alpha=1}^{\#\Pi}\sup_{x_\a,y_\a\in I_\a}|f'(x_\a)-f'(y_\a)|<\infty,}
où $\Pi$ est une partition de $\T$ en un nombre fini d'intervalles disjoints $I_\a$, \mbox{$\a=1,\ldots,\#\Pi$.}

\section{Questions de régularité : de Denjoy à nos jours}
\label{ssc:regularite}
Le théorème de Denjoy \ref{thm:denjoy} signe le début de l'étude systématique des \emph{difféomorphismes} du cercle. Quelle est la régularité maximale de l'homéomorphisme $h$, donné par le théorème de Denjoy, qui conjugue $f$ à $R_\a$ ? Cette question a été au cœur de travaux modernes sur les difféomorphismes du cercle parmi les plus fascinants : les résultats à ce propos de Siegel, Arnol'd, Moser, Herman, Yoccoz, Katznelson et Ornstein, Khanin, Sinai et beaucoup d'autres \cite{siegel,arnold-small,herman,moser,yoccoz,katz-ornstein,khanin-sinai,stark,khanin-sinai2,herman-revisited} ont initié des développements majeurs de la théorie. Nous nous occuperons de ce problème un peu plus loin (voir \ref{sss:conjugaison-differentiable}\,), après quelques rappels pour fixer notations et définitions.

\subsection{Espaces de difféomorphismes}

Soit $r\in[0,\infty]$ un nombre réel, éventuellement infini, on notera $\Diff{r}(\T)$ le groupe des difféomorphismes du cercle $f$ qui préservent l'orientation et sont $\C{r}$ réguliers\footnote{Par abus de langage, un difféomorphisme de classe $\C{0}$ est simplement un homéomorphisme.}. Quand $r$ est un entier, cela signifie que $f$ est dérivable~$r$~fois et que sa dérivée $r$-ième est continue. Quand $r$ n'est pas un nombre entier, cela signifie que $f$ est dérivable~$\lfloor r\rfloor$~fois et que sa dérivée $\lfloor r\rfloor$-ième est une fonction $(r-\lfloor r\rfloor)$-höldérienne.

Par analogie au cas continu, on définit le groupe $\wt{\Diff{r}}(\T)=\{F\in \Diff{r}(\R)\,:\,F(x+1)=F(x)+1\}$, revêtement universel du groupe $\Diff{r}(\T)$ :
\[\xymatrix{\mathbf{0}\ar[r] & \Z\ar[r] & \wt{\Diff{r}}(\T) \ar[r]^\pi& \Diff{r}(\T)\ar[r]& \mathbf{1}.}\]

Lorsque $r\in [0,\infty]$, on munit l'espace $\Diff{r}(\T)$ de la topologie induite par la norme $\C{r}$ : si $r$ est un nombre entier ou bien l'infini,
\[\|f\|_{\hspace{.2pt}r}:=\max_{s\le r}\sup_{x\in \T}|D^{\hspace{0.4pt}s}f(x)|\]
et plus généralement, en posant $\tau:=r-\lfloor r\rfloor$,
\[\|f\|_r:=\|f\|_{\lfloor r\rfloor}+\sup_{x,y\in \T}\frac{|D^{\lfloor r\rfloor}f(x)-D^{\lfloor r\rfloor}f(y)|}{|x-y|^\tau}.\]

\begin{rem}
Soit $f$ un difféomorphisme de classe $\C{r}$. Alors $f$ est un difféomorphisme de classe $\C{s}$ pour tout~$s<r$. Notamment, par définition des normes, $\|f\|_{\hspace{.2pt}s}\le 2\,\|f\|_{\hspace{.2pt}r}$ : l'inclusion de $\Diff{r}(\T)$ dans $\Diff{s}(\T)$ est continue.
\end{rem}

Malheureusement, cette topologie n'est pas toujours compatible avec la structure de groupe.

\begin{prop}
\label{prop:topologique}
Le groupe $\Diff{r}(\T)$ des difféomorphismes $\C{r}$ du cercle, muni de la topologie de la norme $\C{r}$, est un groupe topologique si est seulement si $r$ est un entier ou $r=\infty$.
\end{prop}

\begin{rem}
Dorénavant, lorsque l'on souhaitera voir $\Diff{r}(\T)$ comme un groupe, la topologie  sur $\Diff{r}(\T)$ sera toujours la topologie $\C{\lfloor r\rfloor}$.
\end{rem}

Nous donnerons une preuve de la proposition \ref{prop:topologique} en appendice à ce chapitre.

\medskip

Un autre espace, très important, est l'espace des difféomorphismes analytiques. Dans ce contexte, le cercle est identifié à l'ensemble des nombres complexes de module $1$. Rappelons qu'une fonction du cercle $f$ est dite \emph{analytique} s'il existe un nombre $\rho>1$ tel que $f$ possède un prolongement analytique sur la couronne circulaire~$A_r:=\{z\in \CC\,:\, 1/\rho<|z|<\rho\}$. L'espace des difféomorphismes analytiques sur le cercle est noté $\Diff{\om}(\T)$ et forme un groupe, qui est de plus topologique, lorsque l'on le munit de la topologie qui est la limite inductive des topologies de la convergence uniforme $\C{0}$ sur les fonctions analytiques des couronnes $A_r$.

\subsection{Arithmétique du nombre de rotation}
\label{ssc:arith}

Soit $\alpha\in \R-\Q$ un nombre irrationnel. On notera par $\left (\frac{p_n}{q_n}\right )_{n\in \N}$ la suite des \emph{approximations rationnelles} de $\alpha$ : pour tout $n$,
\beqn{approximation-rationnelle1}{0<\left \vert \alpha -\frac{p_n}{q_n}\right \vert<\frac{1}{q_n^2},}
et pour tout $p$ et tout $0<q<q_n$
\beqn{approximation-rationnelle2}{\left \vert \alpha -\frac{p}{q}\right \vert>\left \vert \alpha -\frac{p_n}{q_n}\right \vert.}

Si $[a_0,a_1,a_2,\ldots]$ est le développement en fraction continue de $\alpha$, à savoir
\beqn{continued}{\alpha=a_0+\frac{1}{a_1+\displaystyle\frac{1}{a_2+\frac{1}{a_3+ \cdots}}},}
la fraction $\frac{p_n}{q_n}$ coïncide avec la troncature au niveau $n$ de \eqref{eq:continued} :
\[\frac{p_n}{q_n}=[a_0,a_1,\ldots,a_n]=a_0+\frac{1}{a_1+\displaystyle\frac{1}{\ddots+\frac{1}{a_n}}}.\]

Les conditions \eqref{eq:approximation-rationnelle1} et \eqref{eq:approximation-rationnelle2} se traduisent dynamiquement : le point $q_n\alpha\pmod{1}$ est le point de l'orbite de $0$ par $R_\alpha$ qui se rapproche le mieux de $0$.

Rappelons qu'une mesure $\mu$ sur $\T$ est \emph{invariante} par un homéomorphisme $f$ si elle est invariante par l'action du groupe cyclique $\Z\cong \langle f\rangle$ : $f_*\mu=\mu$. La théorie des espaces de Banach assure toujours l'existence d'une mesure de probabilité invariante, que l'on peut obtenir comme valeur d'adhérence de la suite $\left (\frac{1}{n}\,\sum_{k=0}^{n-1}\delta_{f^k(x)}\right )_{n\in \N}$, pour un point $x\in \T$ fixé. Toute mesure invariante a pour support une partie compacte invariante de $\T$ : une réunion d'orbites périodiques, lorsque $f$ en possède, ou le compact minimal invariant, dans tout autre cas.

Les mesures invariantes donnent ainsi beaucoup d'informations sur la dynamique. Par exemple, si l'on pose~\mbox{$h(x)=\mu\left ([x_0,x[\right )$}, pour $\mu$ mesure $f$-invariante et $x_0\in \T$ point arbitraire, on a
\[h\circ f=h+\mu \left ([x_0,f(x_0)[\right ).\]
Autrement dit, la fonction $h$ semi-conjugue $f$ à la rotation d'angle $\mu \left ([x_0,f(x_0)[\right )$ et, en particulier, le nombre de rotation de $f$ est égal à $\mu \left ([x_0,f(x_0)[\right )$, pour tout choix de $\mu$ et $x_0$.

Le lien entre mesures invariantes et dynamique est bien plus profond.

\begin{thm}[Inégalité de Denjoy-Koksma]
\label{thm:denjoy-koksma}
Soit $f$ un difféomorphisme du cercle de classe $\C{1}$ et de nombre de rotation irrationnel $\a$. Soit $q=q_n$ le dénominateur d'une approximation rationnelle de $\alpha$. Alors, pour toute fonction~$u\in \C{1+vb}(\T)$ et $\mu$  mesure de probabilité $f$-invariante sur $\T$, on a l'inégalité
\beqn{denjoy-koksma}{\left \Vert \sum_{k=0}^{q-1}u\circ f^k-q\int_{\T}u\,d\mu \right \Vert_0\le \mrm{Var}(u).}
\end{thm}

\begin{rem}
L'inégalité \eqref{eq:denjoy-koksma} permet de voir facilement qu'il y a une unique mesure $f$-invariante, qu'on note~$\mu_f$.
\end{rem}

\begin{lem}[Denjoy]
Soit $f$ un difféomorphisme du cercle de classe $\C{1}$, $\mu$ une mesure $f$-invariante sur $\T$. On a
\[\int_{\T}\log Df\,d\mu=0.\]
\end{lem}

\begin{cor}[Inégalité de Denjoy]
Soit $f$ un difféomorphisme du cercle de classe $\C{1+vb}$ et de nombre de rotation irrationnel $\a$. Soit $q=q_n$ le dénominateur d'une approximation rationnelle de $\alpha$. Alors,
\beqn{inegalite-denjoy}{\left \Vert \log Df^q \right \Vert_0\le \mrm{Var}(\log Df).}
\end{cor}

Ces inégalités montrent que si l'on connaît la vitesse de divergence de $(q_n)$, ou en d'autres termes la qualité d'approximation de $\alpha$ par des rationnels, on peut envisager de bonnes estimations dans la topologie $\C{1}$. Pour cela, on rappelle des notions arithmétiques.

Un nombre irrationnel $\alpha$ vérifie une condition \emph{diophantienne} d'ordre $\delta$ s'il existe une constante $C>0$ telle que pour tout $p/q\in \Q$,
\[0<\left \vert \alpha -\frac{p}{q}\right \vert\ge \frac{C}{q^{2+\delta}}.\]

Lorsque $\alpha$ vérifie une condition diophantienne d'ordre $0$, on dit que $\a$ est de \emph{type constant}. Un nombre $\a$ est de \emph{type Roth} s'il satisfait une condition diophantienne d'ordre $\ve$ pour tout $\ve>0$. Un nombre $\a$ qui vérifie une condition diophantienne quelconque est dit \emph{diophantien}. Les nombres irrationnels qui  ne sont pas diophantiens sont de \emph{type Liouville} et forment une partie résiduelle (au sens de Baire) dans $\R$. En contrepartie, presque tout nombre (au sens de la mesure de Lebesgue) est de type Roth.

\subsection{Le théorème de conjugaison différentiable}
\label{sss:conjugaison-differentiable}
Après 50 années écoulées depuis les premiers résultats, la compréhension de la régularité du difféomorphisme conjugant à la rotation est presque optimale. Le résultat le plus avancé s'énonce ainsi \cite{katz-ornstein} (voir aussi \cite{demelo-strien}\,) :

\begin{thm}[Conjugaison différentiable]
\label{thm:conjugaison-differentiable}
Soit $f$ un difféomorphisme du cercle de classe $\C{r}$, soit $\alpha=\rho(f)$ son nombre de rotation. On suppose que $\alpha$ vérifie une condition diophantienne d'ordre $\delta$. Si $\delta+2<r$, alors l'homéomorphisme $h$ qui conjugue $f$ à $R_\a$ est un difféomorphisme de classe $\C{r-1-\delta-\ve}$ pour tout $\ve>0$.
\end{thm}

\begin{rem}
Lorsque $\delta+2<r<4$, Khanin et Teplins'kyi \cite{herman-revisited} et Teplins'kyi \cite{teplinsky} ont montré que le difféomorphisme $h$ est plus précisément de classe $\C{r-1-\delta}$.
\end{rem}

Essayons de comprendre, de manière un peu naïve, pourquoi une condition diophantienne sur $\a$ influence autant la régularité. Résoudre l'équation $h\circ f = h+\a$ dans l'espace $\C{1}$, revient à s'intéresser à l'\emph{équation cohomologique}
\beqn{eq_cohomologique}{u\circ f-u = -\log Df,}
dans l'espace $\C{0}$. Le théorème suivant, dû à Gottschalk et Hedlund \cite{gottschalk-hedlund} donne une condition nécessaire et suffisante pour que \eqref{eq:eq_cohomologique} ait des solutions.

\begin{thm}[Gottschalk~--~Hedlund]
Soit $(X,d)$ un espace métrique compact et $f$ un homéomorphisme minimal de~$X$ sur $X$. Soit $\vf$ une fonction continue sur $X$. Les deux affirmations suivantes sont équivalentes :
\begin{enumerate}
\item il existe $u\in \C{0}(X)$ tel que $u\circ f-u=\vf$,
\item il existe un point $x_0$ dans $X$ tel que $\sup_{n\in \N}\left \vert \sum_{k=0}^{n-1} f\left (f^k(x_0)\right )\right \vert <+\infty$.
\end{enumerate}
\end{thm}

Puisque pour tout $n\in \N$, $\log Df^n=\sum_{k=0}^{n-1}\log Df\circ f^k$, une condition nécessaire et suffisante pour qu'il existe une solution continue pour \eqref{eq:eq_cohomologique} est
\beqn{eq_cohomologique2}{\sup_{n\in \N}\|\log Df^n\|_0<+\infty.}
La fonction qui à $f\in\Dc$ associe $H_1(f)= \sup_{n\in \N}\|\log Df^n\|_0$ est appelée \emph{invariant de conjugaison $\C{1}$} par Herman. On confrontera avec \eqref{eq:inegalite-denjoy} que si les dénominateurs $q_n$ sont trop écartés l'un de l'autre, l'inégalité de Denjoy n'influence pas trop l'invariant $H_1(f)$. Au contraire, les inégalités du type  \eqref{eq:inegalite-denjoy} sont fondamentales pour la preuve du théorème de conjugaison différentiable. Pour expliquer de manière limpide ce dernier commentaire, nous citons les résultats que l'on peut obtenir grâce à la \emph{méthode de conjugaison rapide} due à Anosov et Katok \cite[\S12.5]{katok-hasselblatt}.

\begin{thm}
Pour tout entier $r\ge 1$, il existe un difféomorphisme minimal $\C{\infty}$ du cercle qui est $\C{r-1}$-conjugué à la rotation, mais qui n'est pas conjugué de manière $\C{r}$.
\end{thm}

\subsection*{Démonstration de la proposition \ref{prop:topologique}}

Montrons d'abord un lemme intermédiaire.

\begin{lem}
\label{lem:topologique}
Il existe un plongement du groupe des difféomorphismes $\C{r}$ de l'intervalle ouvert $]0,1[$ dans le groupe des difféomorphismes $\C{r}$ du cercle.
\end{lem}

\begin{proof}
Soit $\vf:[0,1]\to[0,1]$ un difféomorphisme de classe $\C{\infty}$, infiniment tangent à l'identité en $0$ et~$1$ et vérifiant la condition suivante : pour tout $f$, difféomorphisme de classe $\C{r}$ de $]0,1[$, le difféomorphisme conjugué
\[\tilde{f}:=\vf\circ f\circ \vf^{-1}\]
est un difféomorphisme de l'intervalle fermé $[0,1]$, $\C{r}$-tangent à l'identité en $0$ et $1$.

L'application qui à $f$ associe $\tilde{f}$ définit un homomorphisme continu de $\Diff{r}(]0,1[)$ dans le groupe
\[H:=\left \{g\in\Diff{r}([0,1])\,:\,D^{\hspace{0.4pt}s}g(0)=D^{\hspace{0.4pt}s}g(1)=1\trm{ pour tout }0<s\le r\right \}\]
et ce groupe s'identifie naturellement au sous-groupe $\bar{H}$ de $\Diff{r}(\T)$ composé des difféomorphismes du cercle $\C{r}$-tangents à l'identité en $0\in \T$.
\end{proof}

La démonstration qui suit est une adaptation de la preuve du lemme 3.4 dans \cite{brin}.

\begin{proof}[Démonstration de la proposition \ref{prop:topologique}]
D'après le lemme \ref{lem:topologique}, il suffit de montrer que le groupe des difféomorphismes de classe $\C{r}$ de l'intervalle ouvert $]-1,1[$ n'est pas un groupe topologique, lorsque $r$ n'est pas un nombre entier.

On écrit $r=k+\tau$, avec $k$ entier et $\tau\in ]0,1[$ et on définit le difféomorphisme de classe $\C{r}$
\[g(x)=\begin{cases}
\dfrac{x+x^{\hspace{0.4pt}k}|x|^{\hspace{0.2pt}\tau}}{2} & \text{si }k\text{ impair,} \\

\vspace{-10pt} \\

\dfrac{x+x^{\hspace{0.4pt}k-1}|x|^{\hspace{0.2pt}1+\tau}}{2}& \text{si }k\text{ pair.}
\end{cases}\]
On fixe $q>k$ un nombre pair et pour tout $\ve>0$, on définit le difféomorphisme analytique 
\[f_{\ve}(x)=x-\ve +\ve\hspace{0.2pt} x^{\hspace{0.4pt}q}.\]

Lorsque $\ve\rightarrow 0$, la suite $f_{\ve}$ tend vers $id$ dans la topologie $\C{r} $. On se propose de montrer que la norme $\C{r}$ de~$g\circ f_{\ve} -g$ ne tend pas vers $0$, ce qui impliquera que le groupe $\Diff{r}(]-1,1[)$ n'est pas un groupe topologique par rapport à la topologie $\C{r}$. Pour cela, nous allons donner une minoration uniforme en $\ve$ de la norme $\tau$-hölderienne de la fonction 
\[\phi_\ve=D^{\hspace{0.4pt}k}(g\circ f_{\ve})-D^{\hspace{0.4pt}k}g\]
en regardant dans l'intervalle $[0,\ve]$. À ce propos, nous aurons besoin de connaître les valeurs de $f_{\ve}$ en $0$ et $\ve$ :
\[f_{\ve}(0)=-\ve,\quad f_{\ve}(\ve)=\ve^{q+1}.\]

\smallskip

Comme échauffement, considérons d'abord le cas $k=1$. Nous avons
\[Dg(x)=\frac{1+(1+\tau)\,|x|^{\tau}}{2},\quad Df_{\ve}(x)=1+\ve\,qx^{q-1},\]
d'où
\begin{align*}
\phi_\ve(\ve)-\phi_\ve(0)=\, & Dg(f_{\ve}(\ve))\,Df_{\ve}(\ve)-Dg(\ve)-Dg(f_{\ve}(0))\,Df_{\ve}(0)+Dg(0) \\
=\,& Dg(\ve^{q+1})\,Df_{\ve}(\ve)-Dg(\ve)-Dg(-\ve)\,Df_{\ve}(0)+Dg(0) \\
=\,& Dg(\ve^{q+1})\,Df_{\ve}(\ve)-2Dg(\ve)+\tfrac12 \\
=\,& \tfrac{1+(1+\tau)\ve^{\tau(q+1)}}{2}\,\left (1+q\ve^{q}\right )-\left( 1+(1+\tau)\ve^{\tau}\right )+\tfrac12 \\
=\,&-(1+\tau)\,\ve^{\tau}+o(\ve^{\tau q}).
\end{align*}
On en déduit la minoration
\beqn{minoration_holder}{\|g\circ f_{\ve}-g\|_{1+\tau}\ge \frac{|\phi_\ve(\ve)-\phi_\ve(0)|}{\ve^{\tau}}=(1+\tau)+o(\ve^{\tau(q-1)}),}
donc la suite $(g\circ f_{\ve}-g)_{\ve>0}$ ne tend pas vers $0$ dans la topologie $\C{r}$ lorsque $\ve$ tend vers $0$.

\medskip
Lorsque $k>1$, les estimations sont moins directes, ceci étant dû à la complexité de prendre une dérivée d'ordre supérieure d'une composition. En laissant les vérifications au lecteur, en se restreignant aux termes d'ordre inférieur à $\tau$ en $\ve$, on trouve
\begin{align*}
\phi_{\ve}(\ve)-\phi_{\ve}(0)=\,&-D^kg(\ve)-D^kg(f_\ve(0))+o(\ve^{\tau}) \\
=\,&C_r\,\ve^{\tau}+o(\ve^{\tau}),
\end{align*}
avec $C_r:=-(k+\tau)\cdots(1+\tau)$. On peut donc en déduire une estimation de la norme $\C{r}$ comme à la ligne \eqref{eq:minoration_holder}.

\medskip

Quand $r$ est un nombre entier, la continuité de la composition à gauche et à droite suit de la \emph{formule de Faà di Bruno}
\[D^k (g\circ f)=\sum_{\nu=1}^k\frac{D^\nu g\circ f}{\nu!}\sum_{i_1+\ldots+i_\nu=k}{k \choose {i_1,\ldots,i_\nu}}\,D^{i_1}f\cdots D^{i_\nu}f,\]
où ${k \choose {i_1,\ldots,i_\nu}}$ est le coefficient multinomial $\frac{k!}{i_1!\cdots i_\nu!}$.
\end{proof}
\clearpage
%\cleardoublepage
\chapter{Les mesures de Malliavin-Shavgulidze : difféomorphismes génériques en dimension~$1$}
\label{chapter:1}

\section{Un panorama sur les mesures MS}
%!TEX root = master.tex

\subsection{Petite histoire}
\label{ssc:histoire}

Dans l'esprit d'une certaine partie des mathématiques des années 1980, un intérêt particulier apparût dans la recherche de \emph{mesures quasi-invariantes} sur des groupes de Lie de dimension infinie, comme $\Di$ et $\Dc$, notamment avec les finalités illustrées dans l'introduction. Le théorème de Mackey et Weil \ref{thm:mackey-weil} (chapitre~\ref{sc:localement}) caractérise les groupes localement compacts comme étant les groupes qui possèdent une mesure quasi-invariante par l'action régulière gauche \cite{mackey,weil} : nous n'avons donc aucun espoir de trouver une telle mesure sur $\Di$ (et $\Dc$). Pourtant, Shavgulidze donne des exemples de mesures quasi-invariantes pour les actions \emph{restreintes} à des sous-groupes (notamment par les difféomorphismes de classe $\C{2}$).

Les premiers exemples dans \cite{shav78} ne sont guère utilisables, mais on y voit déjà les outils que l'auteur utilisera dans ses futurs travaux. C'est une dizaine d'années plus tard que Shavgulidze reprendra son travail en définissant des mesures analogues pour les difféomorphismes de variétés riemanniennes arbitraires \cite{shav88}.

Probablement vers la fin des années 1980, les mathématiques de Shavgulidze rencontrent celles de P.~Malliavin qui s'intéresse alors aux processus de diffusion sur les variétés riemanniennes. Ils parviennent à la construction des mesures que nous appelons aujourd'hui mesures de \ms (\emph{mesures MS}) sur les difféomorphismes de l'intervalle et du cercle. Même si chacun est capable de construire des mesures boréliennes avec des propriétés d'invariance dans des contextes plus généraux, seules ces mesures possèdent des propriétés vraiment remarquables, en particulier pour leur simplicité \cite{mall90,mall91,shav97}.

En 1990, P.~Malliavin et M.~P.~Malliavin montrent une quasi-invariance \emph{infinitésimale} sous l'action des difféomorphismes $\C{3}$ (avec des techniques d'analyse stochastique). D'un autre coté, la propriété de quasi-invariance est claire dans l'environnement autour de Shavgulidze (par exemple un esquisse de preuve apparaît  dans l'article de Kosyak \cite{kosyak}\,). Une preuve plus détaillée apparaît en 1997 \cite{shav97} (en russe) et \cite{shav00} (en anglais) (voir aussi la thèse de l'auteur \cite{these}). Une exposition plus moderne se trouve dans le livre de Bogachev \cite[\S11.5]{bogachev-differentiable}, que nous allons suivre ici : il s'agit simplement de vérifier que l'on peut appliquer le théorème \ref{thm:ramer} présenté dans le chapitre \ref{sc:wiener}.
On espère que cette méthode plus conceptuelle puisse aider à comprendre géométriquement les cocycles de Radon-Nikodym associés aux mesures MS.

\medskip

Plus récemment Kuzmin  \cite{kuzm07} reprend le résultat de Shavgulidze et l'utilise pour affaiblir le plus possible la régularité des fonctions qui quasi-préservent la mesure.

D'autres mesures quasi-invariantes sur $\Dc$ ont été introduites par Neretin \cite{neretin} ; celui-ci cite les travaux de Malliavin et Shavgulidze, mais ne parvient pas à décrire le lien avec le sien.

Une liste très complète de constructions de mesures sur les groupes de difféomorphismes se trouve dans \cite[\S11.5]{bogachev-differentiable}.

\begin{rem}
Jones a défini une mesure sur les homéomorphismes du cercle et Astala, Jones, Kupiainen, Saksman \cite{homeoGFF} (voir aussi Sheffield \cite{zipper}\,) ont montré que, presque sûrement cette mesure donne un homéomorphisme höldérien. La construction étant assez proche de celle de Malliavin et Shavgulidze, nous croyons fortement que la mesure de Jones est quasi-invariante par l'action des difféomoprhismes $\C{2}$.
\end{rem}

\subsection{Définition des mesures}

Construire une mesure quasi-invariante sur un groupe topologique est un problème assez compliqué en général. Rappelons que, d'après ce que nous avons vu dans le chapitre~\ref{sc:localement}, lorsque $G$ est un groupe non localement compact, il ne peut pas exister de mesure de Radon sur $G$, quasi-invariante par rapport à l'action de translation d'un sous-groupe de mesure strictement positive.  Une mesure de Radon $\mu$ sur $G$ sera dite \emph{quasi-invariante} si elle est quasi-invariante par rapport à l'action par translation d'un sous-groupe $G_0\le G$, qui est de mesure nulle mais pourtant dense. Une telle définition amène la propriété suivante :

\begin{prop}
Soit $\mu$ une mesure de Radon non identiquement nulle sur le groupe topologique $G$. On suppose qu'il existe un sous-groupe dense $G_0\le G$ qui préserve la classe de la mesure $\mu$. Alors la mesure de tout ouvert $U\subs G$ est strictement positive : $\mu(U)>0$.
\end{prop}

\begin{proof}
En effet, le support de la mesure $\mu$ (défini comme étant le complémentaire du plus grand ouvert de mesure nulle) est un fermé $G_0$-invariant et coïncide donc avec $G$ entier.
\end{proof}

D'après le théorème de Cameron-Martin \ref{thm:cameron-martin0}, la mesure de Wiener sur $C_0(I)$ ou $C_0(\T)$ est quasi-invariante. En dehors de cet exemple, on n'en connait pas beaucoup d'autres (et ils sont certainement moins célèbres).

La construction des mesures de Malliavin-Shavgulidze s'appuie si fortement sur la mesure de Wiener, que d'un point de vue mesurable elles sont presque les mêmes.

\subsubsection{Les mesures sur le groupe des difféomorphismes de l'intervalle}
\label{ssc:MSi}

La première étape consiste à identifier topologiquement le groupe de difféomorphismes de l'intervalle de classe $\C{1}$ qui préservent l'orientation, à l'espace de Wiener $C_0(I)$. Il y a certainement beaucoup de manières de le faire, puisque si $h:\Di\to C_0(I)$ est un homéomorphisme, alors en composant $h$ à la source par~$\phi\in \mrm{Hom\acute{e}o}_+(I)$, on obtient un autre homéomorphisme entre les deux espaces. Pourtant, il existe un choix particulièrement naturel.

En effet, pour connaître un difféomorphisme de l'intervalle $f$, il suffit de connaître sa dérivée $f'$, qui est une fonction continue strictement positive, de moyenne égale à $1$. Réciproquement, soit $h$ une fonction continue strictement positive sur l'intervalle $I$, on peut alors lui associer un difféomorphisme $f$ défini par
\[f(t)=\dfrac{\int_0^th(s)\,ds}{\int_0^1h(s)\,ds}.\]

On remarquera que dans l'expression précédente un multiple constant de la fonction $h$ définit le même $f$. On résume ceci.

\begin{prop}\label{prop:homeo}
L'espace $\Di$ des difféomorphismes $\C{1}$ de l'intervalle qui préservent l'orientation s'identifie à l'espace $C_1(I,\R_+)$ des fonctions continues positives sur l'intervalle vérifiant $h(0)=1$. Cette application induit un homéomorphisme entre les deux espaces, munis respectivement des topologies $\C{1}$ et $\C{0}$.
\end{prop}

\begin{proof}
Nous avons déjà démontré la première partie de l'énoncé.

Soit $h$ une fonction fixée et soit $\ve>0$. Nous voulons montrer qu'il existe $\delta>0$ tel que si $h$ et $k$ se trouvent à une  distance $\C{0}$ inférieure à $\delta$, alors les difféomorphismes associés sont à distance $\C{1}$ inférieure à $\ve$. Nous avons
\begin{multline*}
\sup_{t\in I}\left\vert \tfrac{\int_0^t h(s)\,ds}{\int_0^1 h(s)\,ds}-\tfrac{\int_0^t k(s)\,ds}{\int_0^1 k(s)\,ds}\right \vert \\
= \sup_{t\in I}\left\vert \tfrac{\int_0^t[ h(s)-k(s)]\,ds}{\int_0^1 h(s)\,ds}+\int_0^t k(s)\,ds\left (\tfrac{1}{\int_0^1 h(s)\,ds}-\tfrac{1}{\int_0^1 k(s)\,ds}\right )\right \vert \\
\le \sup_{t\in I}\tfrac{t}{\int_0^1h(s)\,ds}\delta+\sup_{t\in I}\tfrac{\int_0^tk(s)\,ds}{\int_0^1h(s)\,ds\cdot\int_0^1k(s)\,ds}\delta\le \tfrac{2}{\int_0^1h(s)\,ds}\delta,
\end{multline*}
et pour l'estimation sur la dérivée
\begin{multline*}
\sup_{t\in I}\left\vert \tfrac{h(t)}{\int_0^1 h(s)\,ds}-\tfrac{k(t)}{\int_0^1 k(s)\,ds}\right \vert\\
\le \sup_{t\in I}\left\vert \tfrac{h(t)-k(t)}{\int_0^1 k(s)\,ds}+k(t)\left (\tfrac{1}{\int_0^1 h(s)\,ds}-\tfrac{1}{\int_0^1 k(s)\,ds}\right)\right \vert\\ \le \left (1+\sup_{t\in I}\tfrac{k(t)}{\int_0^1 k(s)\,ds}\right )\tfrac{\delta}{\int_0^1h(s)\,ds}\\
\le \left (1+\tfrac{\sup_{t\in I}h(t)+\delta}{\left \vert\int_0^1 h(s)\,ds-\delta\right \vert}\right )\tfrac{\delta}{\int_0^1h(s)\,ds}=\left( 1+\sup_{t\in I} \tfrac{h(t)}{\int_0^1h(s)\,ds}\right )\tfrac{\delta}{\left\vert\int_0^1h(s)\,ds-\delta\right \vert}.
\end{multline*}
En posant
\[\delta = \frac{\ve \int_0^1h(s)\,ds}{1+\mathrm{sup}_I\frac{h}{\int_0^1h(s)\,ds}+\ve},\]
nous déduisons des estimations précédentes que la norme $\C{1}$ de la différence des difféomorphismes associés à $h$ et $k$ est inférieure à $2\ve$.

\medskip

Réciproquement, soit $F$ un difféomorphisme fixé et $\ve>0$. Nous cherchons $\delta>0$ tel que, pour tout $G$ à distance $\C{1}$ de $F$ inférieure à $\delta$, les fonctions $F'/F'(0)$ et $G'/G'(0)$ soient à distance $\C{0}$ inférieure à $\ve$.

Pour tout $\delta\in ]0,F'(0)[$ nous avons
\begin{multline*}
\sup_I\left \vert \tfrac{F'}{F'(0)}-\tfrac{G'}{G'(0)}\right\vert \\
\le \tfrac{1}{F'(0)}\delta+\sup_I G' \left \vert\tfrac{1}{F'(0)}-\tfrac{1}{G'(0)}\right \vert\le \left (1+\tfrac{\sup_I G'}{G'(0)}\right )\tfrac{1}{F'(0)}\delta\\
\le \left (1+\tfrac{\sup_I F'+\delta}{F'(0)-\delta}\right )\tfrac{1}{F'(0)}\delta=\left (1+\tfrac{\sup_I F'}{F'(0)}\right )\tfrac{\delta}{F'(0)-\delta}
\end{multline*}
et par conséquent on peut choisir
\[\delta = \frac{\ve F'(0)}{1+\frac{\sup_I F'}{F'(0)}+\ve}\]
afin d'avoir la propriété de continuité cherchée.
\end{proof}

Cette proposition termine « presque » notre première étape, puisque $C_1(I,\R_+)$ n'est pas très différent de l'espace de Wiener $C_0(I)$ : en effet,
par composition à gauche par la fonction logarithme, les deux espaces s'identifient topologiquement.

Dans la suite, on notera 
\beqn{A}{A\,:\,\Di\to C_0(I)}
l'application qui à $f\in \Di$ associe la fonction continue $\log f'-\log f'(0)$. D'après la proposition~\ref{prop:homeo}, la fonction $A$ est un homéomorphisme lorsque l'on munit les espaces de leurs topologies naturelles.

Observons que l'application inverse $A^{-1}$ est définie par
\[A^{-1}x(t)=\frac{\int_0^te^{x(s)}\,ds}{\int_0^1e^{x(s)}\,ds},\]
et grâce à cette application, nous arrivons à la deuxième étape de la construction : la définition des mesures de Malliavin-Shavgulidze.

\medskip

Il est peut-être temps de rappeler que la mesure de Wiener définie dans la partie \ref{ssc:bm} du chapitre \ref{sc:wiener} fait partie d'une famille à un paramètre de mesures : en notant $\sigma>0$ ce paramètre, nous retrouvons avec $\sigma=1$ la mesure de Wiener classique $\W$. Les autres mesures $\W_\sigma$ diffèrent par la \emph{variance} que l'on considère pour les accroissements du mouvement brownien : si $x\in C_0(I)$ est distribué selon la mesure $\W$, la fonction $\sigma\hspace{0.3pt}x\in C_0(I)$ est distribuée selon la mesure $\W_{\sigma}$. Plus précisément, la mesure $\W_{\sigma}$ d'un cylindre 
\[C=\{x\in C_0(I)\,|\, x(t_i)-x(t_{i-1})\in A_i,\,i=1,\ldots n\},\]
avec $0=t_0<t_1<\ldots<t_{n}<t_{n+1}=1$  et $A_i\in \cB(\R)$,
est donnée par (voir \eqref{eq:cylindre2})
\[\W_\sigma(C)=\tfrac{1}{\sqrt{(2\pi\sigma)^n(t_{n}-t_{n-1})\cdots(t_1-t_0)}}\int_{A_1}\cdots\int_{A_n}\exp\left (-\tfrac{1}{2\sigma}\sum_{i=1}^n\tfrac{y_i^2}{t_i-t_{i-1}}\right )\,dy_1\cdots dy_n.\]

\begin{rem}
Le paramètre $\sigma$ mesure entre autres la distance de la fonction aléatoire à la fonction identiquement nulle. Par exemple, lorsque $\sigma$ est proche de zéro, les fonctions aléatoires se trouvent assez proches de la fonction nulle. Pour rendre cette affirmation plus formelle rappelons le \emph{principe de réflexion} \cite[Theorem 2.19]{peres} : pour tout $r>0$, on a
\[\W_\sigma\left (\left \{ \|x\|_0\le r\right \}\right )\le \tfrac{4}{\sqrt{2\pi}}\int_{r/\sigma}^\infty e^{-s^2/2}\,ds\le \tfrac{8}{\sqrt{2\pi}}\,\tfrac{\sigma}{r}\exp\left (-\tfrac12 \,\tfrac{r^2}{\sigma^2}\right ).\]

Mais $\sigma$ mesure aussi la distance locale : en effet, pour $\sigma$ proche de $0$, les oscillations locales d'une fonction typique s'affaiblissent, comme l'explique la \emph{loi des logarithmes itérés}
\beqn{iteratedlog}{
\W_{\sigma}\left (\left \{
	\varlimsup_{\ve\rightarrow 0} \frac{x(t+\ve)-x(t)}{\sqrt{2\ve\,\log\log (1/\ve)}}=\sigma
\right \}\right )=1\quad \trm{pour tout }t\in I.
}
\end{rem}

\begin{dfn}
Soit $\sigma>0$. On définit la mesure de \ms de paramètre $\sigma$ sur les difféomorphismes de l'intervalle, comme l'image de la mesure de Wiener par l'application $A^{-1}$ :
\[\nu_\sigma= (A^{-1})_*\W_\sigma.\]

Quand $\sigma=1$, on parlera de \emph{la mesure de Malliavin-Shavgulidze}, que l'on notera $\nu_{MS}$ ou bien simplement $\nu$.
\end{dfn}

\subsubsection{Les mesures sur le groupe des difféomorphismes du cercle}

En modifiant la construction de la partie \ref{ssc:MSi} on trouve des mesures définies sur le groupe des difféomorphismes du cercle.
Dans ce contexte, le \emph{cercle} $\T$ est paramétré par le tore euclidien $\R/\Z$.

\medskip

Munis de leurs topologies naturelles, l'espace $\Dc$ des difféomorphismes $\C{1}$ du cercle s'identifie topologiquement à l'espace $C_0(\T)\times \T$ : à tout difféomorphisme $f$ on associe le couple 
\beqn{B}{B(f)=(B_0(f),f(0)),}
avec
\[B_0(f)(t)=\log f'(t)-\log f'(0).\]

\begin{rem}
L'application inverse est donnée par
\[B^{-1}(x,\alpha)(t)=c(x)\int_0^te^{x(s)}\,ds+\alpha,\]
où $c(x)=\left (\int_0^1e^{x(s)}\,ds\right )^{-1}$.
\end{rem}

Le lecteur étant passé par la preuve de la proposition~\ref{prop:homeo} n'aura pas de difficulté à montrer que $B$ définit en effet un homéomorphisme.

\medskip

Nous avons introduit à la ligne \eqref{eq:proj_pont} une application $\pi$ qui projette l'espace $C_0(I)$ sur $C_0(\T)$, grâce à laquelle nous avons pu définir la mesure de Wiener $\W_0$ sur $C_0(\T)$. Par la même méthode, on dispose des mesures $\W_{0,\sigma}$, définies par
\[\W_{0,\sigma}:=\pi_*\W_{\sigma}.\]

\begin{dfn}
Soit $\sigma>0$. On définit la mesure de \ms de paramètre $\sigma$ sur les difféomorphismes du cercle, comme l'image du produit de la mesure de Wiener par la mesure de Lebesgue sur le cercle, par l'application $B^{-1}$ : 
\beqn{MSc}{\mu_\sigma= B^{-1}_*\left (\W_{0,\sigma}\otimes \Leb\right ).}

Quand $\sigma=1$, on parlera de \emph{la mesure de Malliavin-Shavgulidze}, que l'on notera $\mu_{MS}$ ou bien simplement $\mu$.
\end{dfn}

\subsection{Modules de continuité}
\label{sssc:modulecontinuite}

Nous l'avons vu dans la partie \ref{ssc:regularite} du chapitre~\ref{sc:wiener}, la classe de régularité d'un difféomorphisme du cercle  influence les phénomènes dynamiques. Il est donc important, si l'on souhaite étudier les mesures MS d'un point de vue dynamique, de connaître la régularité d'un difféomorphisme aléatoire. Les propriétés de la mesure de Wiener que l'on a rencontrées dans la partie \ref{ssc:bm} nous permettent une description assez exhaustive. Nous avons de nouveau besoin de quelques rappels.

\medskip

Soit $w\,:\,[0,1]\to [0,+\infty[$ une fonction continue en $0$ et vérifiant $w(s)=0$ si et seulement si $s=0$. On supposera aussi que $w$ est croissante. 

On définit l'espace $\C{w}(I)$ des fonctions avec \emph{module de continuité} $w$ :
\[\C{w}(I)=\left\{ x\in C(I)\,\left |\,\sup_{s,t\in I,\,s\ne t}\frac{|x(s)-x(t)|}{w(|s-t|)}<\infty\right. \right\}.\]

Si $x$ est une fonction de module continuité $w$, on définit sa \emph{norme} $\C{w}$ :
\[\|x\|_w=\sup_{s,t\in I,\,s\ne t}\frac{|x(s)-x(t)|}{w(|s-t|)}.\]

Une fonction de module de continuité est automatiquement continue. Réciproquement, toute fonction sur $I$ continue possède un module de continuité : si $x$ est une fonction continue, on pose
\[w_x(\ve)=\sup_{|s-t|\le\ve}|x(s)-x(t)|.\]

\begin{ex}
Quand $w(s)=s^\alpha$, une fonction de module de continuité $w$ est une fonction $\alpha$-höldérienne.
\end{ex}

Dans la suite, on notera $\Diff{1+w}(I)$ l'espace des difféomorphismes de l'intervalle dont la dérivée est une fonction de module de continuité $w$. Si $f$ est un difféomorphisme, on définit sa \emph{norme} $\C{1+w}$ :
\[\|f\|_{1+w}=\|f\|_1+\|f'\|_w.\]

\medskip

La proposition suivante est la clef de ce qui va suivre.

\begin{prop}\label{prop:modcont}
Soit $w$ un module de continuité. L'application $A$ définie dans \eqref{eq:A} induit un homéomorphisme entre les difféomorphismes de l'intervalle de classe $\C{1+w}$ et les fonctions de classe $\C{w}$, lorsque l'on munit ces espaces des topologies induites par leurs normes.
\end{prop}

\begin{proof}
Remarquons que si $h$ est une fonction strictement positive définie sur l'intervalle, alors on a
\[\sup_I h\,\|h\|_w\le\|\log h\|_w\le \frac{1}{\inf_I h}\|h\|_w,\]
puisque la fonction $\log$ est lipschitzienne sur l'intervalle compact $h(I)\subs ]0,+\infty[$. Le lecteur pourra alors facilement terminer la preuve.
\end{proof}

\begin{cor}[$\C{2}$ en aucun point]
L'ensemble des difféomorphismes $f\in \Di$ tels que $f'$ n'est continue en aucun point est de mesure $\nu_\sigma$ totale.
\end{cor}

\begin{cor}[Presque jamais $\C{1+vb}$]
Soit $\mrm{Diff}^{\hspace{0.8pt} 1+vb}_+(I)$ l'ensemble des difféomorphismes $f\in \Di$ tels que $f'$ est une fonction de variation bornée sur $I$. Alors $\nu_\sigma\left (\mrm{Diff}^{\hspace{0.8pt} 1+vb}_+(I)\right )=0$.
\end{cor}

\begin{cor}[Dérivée $1/2$-höldérienne]
~

\begin{enumerate}
\item Pour tout $\tau<1/2$, $\nu_\sigma\left (\mrm{Diff}^{\hspace{0.8pt} 1+\tau}_+(I)\right )=1$.
\item Pour tout $\tau>1/2$, $\nu_\sigma\left (\mrm{Diff}^{\hspace{0.8pt} 1+\tau}_+(I)\right )=0$.
\end{enumerate}
\end{cor}
Et pour être encore plus précis,

\begin{cor}[Module de continuité de Lévy]
\label{cor:levy}
Soit $w(t)=\sqrt{2t|\log t|}$ le module de continuité de Lévy, alors~$\nu_\sigma\left (\mrm{Diff}^{\hspace{0.8pt} 1+w}_+(I)\right )=1$.
\end{cor}

\begin{rem}
L'expression dans \eqref{eq:levy0} exprime une propriété locale. Cependant, comme la fonction $\frac{x(t)-x(s)}{w(|t-s|)}$ est uniformément continue sur le carré $I\times I$ privé d'un voisinage de la diagonale, on déduit de \eqref{eq:levy0} que la norme~$\|x\|_w$ est finie.\footnote{Pour être plus précis, il faudrait modifier la fonction $w$ pour la rendre croissante.}
\end{rem}

Remarquons que rien ne change dans les énoncés précédents si l'on s'intéresse aux mesures MS sur les difféomorphismes du cercle. 

\begin{prop}
Soit $w$ un module de continuité. L'application $B$ définie dans \eqref{eq:B} induit un homéomorphisme entre les difféomorphismes du cercle de classe $\C{1+w}$ et les fonctions de classe $\C{w}$, quand on munit ces espaces des topologies induites par leurs normes.
\end{prop}

\begin{cor}
Soit $w(t)=\sqrt{2t|\log t|}$ le module de continuité de Lévy. Pour tout $\sigma>0$, les difféomorphismes du cercle de classe $\C{1+w}$ ont mesure $\mu_\sigma$ totale. Il en résulte qu'un difféomorphisme MS est $\C{1+\alpha}$-régulier pour tout $\alpha<1/2$.
\end{cor}

\newpage
\section{Quasi-invariance des mesures MS}

\subsection{Le théorème de Cameron-Martin-Shavgulidze}

Nous en venons à la propriété des mesures MS qui les rend remarquables vis-à-vis des autres mesures sur les groupes de difféomorphismes : leur \emph{quasi-invariance}.

\begin{thm}[Shavgulidze]\label{thm:quasi-invariance}
Soit $G$ l'un des groupes $\Di$ ou $\Dc$, soit $\mu$ une mesure MS sur $G$ et soit $H$ le sous-groupe de $G$ constitué des difféomorphismes avec deuxième dérivée définie presque partout et dans $L^{2}$.
La mesure $\mu$ est quasi-invariante sous l'action du groupe $H$ par translations à gauche.

En outre, le cocycle de Radon-Nikodym est
\beqn{shavRD}{\frac{d(L_\vf)_*\mu}{d\mu}(f)=\exp\left \{-\frac{1}{\sigma}\left [\int\frac{\vf''(f(t))}{\vf'(f(t))}\hspace{0.3pt}df'(t)+\frac{1}{2}\int\left [\frac{\vf''(f(t))}{\vf'(f(t))}\hspace{0.3pt}f'(t)\right ]^2dt\right ]\right \},}
où $\sigma$ correspond au paramètre de la mesure MS et les intégrales sont faites sur l'espace total $I$, ou bien sur $\T$, en accord avec le cas considéré.
\end{thm}

En effet, lorsque $\vf$ est de classe $\C{3}$, on peut faire une intégration par parties dans l'expressions donnant le cocycle :
dans le cas des mesures sur $\Di$ on a
\[
\frac{d(L_\vf)_*\nu_{\sigma}}{d\nu_{\sigma}}(f)
=\exp\left \{
	-\tfrac{1}{\sigma}\left [
		\tfrac{\vf''(f(1))}{\vf'(f(1))}\hspace{0.3pt}f'(1)
		-\tfrac{\vf''(f(0))}{\vf'(f(0))}\hspace{0.3pt}f'(0)-
		\int_0^1\cS_\vf(f(t))
		f'(t)^2 dt
	\right ]
\right \},
\]
alors que les termes de bord disparaissent dans le cas des mesures sur $\Dc$
\[\frac{d(L_\vf)_*\mu_{\sigma}}{d\mu_{\sigma}}(f)=\exp\left \{\tfrac{1}{\sigma}\hspace{0.3pt}\int_{\T}\cS_\vf(f(t))
		f'(t)^2 dt\right \}.\]

\begin{rem}
Observons la présence de la dérivée de Schwarz dans les cocycles de Radon-Nikodym. Nous rappelons que la dérivée schwarzienne d'un difféomorphisme $\vf$ se définit ainsi:
\[\cS_\vf(t)=\frac{\vf'''(t)}{\vf'(t)}-\frac{3}{2}\left (\frac{\vf''(t)}{\vf'(t)}\right )^2\]
ou bien encore
\[\cS_\vf(t)= D^2\log D\vf-\frac{1}{2}(D\log D\vf)^2,\]
et vérifie la relation de cocycle
$\cS_{\vf\psi}=\cS_\vf\circ \psi\cdot (\psi')^2+\cS_\psi$.
D'après cette relation, elle est naturellement définie sur l'espace des différentielles quadratiques sur le cercle.
De plus, la dérivée $\cS_\vf$ est identiquement nulle si et seulement si $\vf$ est une transformation de Möbius. Elle pourrait suggérer une certaine invariance projective.
\end{rem}

Pour compléter l'énoncé du théorème~\ref{thm:quasi-invariance}, nous observons que lorsque le sous-groupe $H$ agit \emph{à droite}, la classe de la mesure de $\mu$ n'est pas préservée, sauf dans un cas très particulier :

\begin{thm}[\cite{kosyak}\,]
Soit $G$ l'un des groupes $\Di$ ou $\Dc$, soit $\mu$ une mesure MS sur $G$ et soient $\vf$, $\psi\in G$. Alors les mesures $(R_\vf)_*\mu$ et $(R_\psi)_*\mu$ sont équivalentes si et seulement si $\vf$ et $\psi$ diffèrent d'une rotation.
\end{thm}

\begin{rem}
En particulier, sur le groupe $\Di$, aucune mesure MS n'est quasi-invariante par la composition à droite par aucun difféomorphisme. On en déduit qu'il existe un très grand nombre de classes de mesures différentes sur $\Di$ et $\Dc$ qui sont invariantes par l'action à gauche de $H$.
\end{rem}

La preuve de ce dernier théorème utilise la propriété des logarithmes itérés \eqref{eq:iteratedlog} du mouvement brownien : si l'on considère la composition à droite par un difféomorphisme $\vf$, en chaque point $t$ on trouvera que l'événement
\beqn{modkosyak}{
\varlimsup_{\ve\rightarrow 0}\dfrac{|\log f'(t+\ve)-\log f'(t)|}{\sqrt{2\ve\,\log \log (1/\ve)}}=\sigma\sqrt{ \vf'(t)}
}
est de mesure totale pour la nouvelle mesure de probabilité $(R_\vf)_*\mu$.

L'une des différences majeures entre les actions à droite et à gauche est que la première induit un \emph{changement de temps} pour le brownien correspondant à $f$, alors que l'action à gauche est simplement une « translation non-linéaire » : soit $\vf$ un difféomorphisme assez régulier et $f$ un difféomorphisme aléatoire défini par un mouvement brownien $B$. En formules, l'action à droite donne pour la mesure de Wiener
\[ \log (f\circ \vf)'-\log (f\circ \vf)'(0)=B\circ \vf+\log \vf'-\log\vf',\]
alors que celle à gauche donne
\[\log (\vf\circ f)'-\log(\vf\circ f)'(0)=B+\log\vf'\circ f-\log\vf'(f(0)).\] 

\bigskip

Essayons de comprendre, grâce à ce qu'on a déjà essayé d'illustrer dans la partie \ref{sc:nonlineaire} du chapitre~\ref{sc:wiener}. Nous allons suivre l'exposition de Bogachev \cite{bogachev-differentiable}.

Notre but consiste maintenant à traiter la transformation non-linéaire que $L_\vf$ induit sur l'espace de Wiener. Pour des raisons de simplicité, nous considérerons seulement les mesures $\nu_{MS}$ et $\mu_{MS}$ (la variance $\sigma$ ne représente pas vraiment un obstacle, mais elle alourdit les calculs).

Soit $\vf\in \mrm{Diff}^{\hspace{0.8pt} 2}_+(I)$ et posons $h:=A(\vf^{-1})$, on écrira alors
\beqn{actionconj}{T_h(x)=(I+S_h)(x)=A\circ L_\vf\circ A^{-1}(x),}
avec 
\[S_h(x)(t)=h\left (c(x)\int_0^te^{x(s)} d s\right ),\quad c(x)=\frac{1}{\int_0^1e^x}.\]
Le groupe des difféomorphismes $\C{2}$ agit sur $C_0(I)$ par  $T$ défini en \eqref{eq:actionconj}.

\begin{rem}
De manière similaire, lorsque $\vf\in \mrm{Diff}^{\hspace{0.8pt} 1}_+(\T)$, on pose $h=B_0(\vf^{-1})$ et
\[B\circ L_\vf\circ B^{-1}(x,\alpha)=(T_h(x,\alpha),\vf^{-1}(\alpha)),\]
avec
\[T_h(x,\alpha)(t)=x(t)+h\left (\alpha+c(x)\int_0^te^{x(s)} d s\right )-h(\alpha),\quad c(x)=\frac{1}{\int_0^1e^x},\]
qui définit une action sur $C_0(\T)\times \T$.
\end{rem}

En fait, pour la preuve on se restreindra à l'action de $\mrm{Diff}^{\hspace{0.8pt} 2}_+(I)$ sur $\Di$.

L'image de l'opérateur non-linéaire $S_h$ est incluse dans $\mathbf{D}_0(I)$ puisque $h\in \mathbf{D}_0(I)$ et $c(x)\int_0^te^x\in \C{1}$. Sa dérivée jacobienne au point $x\in C_0(I)$, dans la direction de $v\in \mathbf{D}_0(I)$ est
\begin{align*}
D_xS_h(v)(t)=\,&\frac{d}{d\ve}h\left (
	c(x+\ve\cdot v)\int_0^te^{x+\ve\cdot v}\right )\,
\Bigg\vert_{\ve=0}\\
=\,&
h'\left (
	c(x)\int_0^te^x\right )\left [ c(x)\int_0^t(e^xv)-c(x)^2\int_0^te^x\int_0^1(e^xv)\right ].
\end{align*}
Cette expression peut sembler un peu trop complexe, donc pour aider le lecteur, rappelons qu'il faut substituer mentalement $f(t)=c(x)\int_0^te^x$ pour écrire plus simplement
\beqn{differentielleSh}{
D_xS_h(v)(t)=h'(f(t))\left[\int_0^t(f'v)-f(t)\int_0^1(f'v)\right].
}
On remarque que l'opérateur $D_xS_h$ est de type intégral avec son noyau défini par l'expression entre crochets :
\[K_{h,x}(t,s):=h'\left (f(t)\right )\left (\mathbf{1}_{[0,t]}(s)-f(t)\right )f'(s),\]
appartenant à l'espace $L^2([0,1]^2)$.
En particulier $D_xS_h$ est un opérateur compact.

\begin{lem}
Soit $h\in \mathbf{D}_0(I)$ et $S_h$ défini par \eqref{eq:actionconj}. Alors pour tout $x\in C_0(I)$, le spectre de la restriction de $D_xS_h$ à $\mathbf{D}_0(I)$ est réduit à la seule valeur propre $0$.
\end{lem}

\begin{proof}
Soit $\lambda\in \mathbf{C}-\{0\}$ et soit $v\in \mathbf{D}_0(I)$ une solution de l'équation
\[D_xS_h(v)=\lambda\,v.\]
On écrit, en utilisant un peu improprement $f(t)=c(x)\int_0^te^x$, comme dans \eqref{eq:differentielleSh},
\beqn{valeurpropre}{h'\left (
	f(t)\right )\left [ \int_0^t(f'v)-f(t)\int_0^1(f'v)\right ]=\lambda\,v(t),}
Substituons la fonction inconnue $v$ par une nouvelle inconnue $u$, liée par la relation
\[v(t)=h'\left (
	f(t)\right )\,u(t).\]
Alors l'équation \eqref{eq:valeurpropre} devient
\beqn{valeurpropre2}{\int_0^t f'(s)h'\left (
	f(s)\right )\,u(s)\,ds+E\,f(t)=\lambda\,u(t),}
avec
\[E=-\int_0^1(f'v)\]
ne dépendant pas de $t$ et que par homogeneité de l'équation, on suppose égal à $0$ ou $1$.
Remarquons qu'en posant $t=1$ dans l'équation \eqref{eq:valeurpropre2}, on trouve $u(1)=0$ (il s'agit en quelque manière de la clef de la preuve !).

Avec une intégration par parties, on trouve
\beqn{valeurpropre3}{
h\left (
	f(t)\right )\,u(t)-\int_0^th\left (
	f(s)\right )\,u'(s)\,ds+E\,f(t)=\lambda\,u(t).}
En dérivant l'expression \eqref{eq:valeurpropre3} on trouve l'égalité presque partout :
\[h'\left (
	f(t)\right )f'(t)\,u(t)+E\,f'(t)=\lambda\,u'(t).\]
Les solutions de cette équation linéaire du premier ordre sont de la forme
\[u(t)=\exp\left (\tfrac{1}{\lambda}h\left (
	f(t)\right )\right )\left (\tfrac1\lambda E\int_0^tf'(s)\exp\left (-\tfrac1\lambda h(f(s))\right )\,ds+b\right ),\quad b\in\R.\]
En fait $b$ doit forcement être égal à $0$, puisque d'après l'équation \eqref{eq:valeurpropre2} on a $u(0)=0$.
Donc
\[u(t)=E\tfrac1\lambda\exp\left (\tfrac{1}{\lambda}h\left (
	f(t)\right )\right )\int_0^tf'(s)\exp\left (-\tfrac1\lambda h(f(s))\right )\,ds\]
et en particulier, si $E\ne 0$, la fonction $u$ est non nulle pour tout $t>0$. Comme $u(1)=0$ alors $E=0$. Donc $\lambda$ n'appartient pas au spectre de $D_xS_h$.
\end{proof}

En donnant référence à la partie~\ref{sc:nonlineaire} du chapitre~\ref{sc:wiener},
on en déduit que le déterminant de Fredholm-Carleman de $T_h=I+S_h$ est égal à $1$. Donc, par le théorème~\ref{thm:ramer} sur les transformations linéaires de la mesure de Wiener, on a
\begin{multline*}
\frac{d(T_h)_*\W}{d\W}(T_h(x))=\exp\left (\langle S_h(x),x\rangle_{\mathbf{D}_0(I)}+\tfrac12\|S_h(x)\|_{\mathbf{D}_0(I)}^2\right ) \\
=
\exp\left (\scalebox{0.9}{\text{$\int_0^1h'\left (c(x)\int_0^te^x\right )c(x)e^{x(t)}dx(t)+\tfrac12\int_0^1\left (h'\left (c(x)\int_0^te^x\right )c(x)e^{x(t)}\right )^2dt$}}\right ).
\end{multline*}
En revenant à l'espace des difféomorphismes, on peut écrire
\beqn{deriveediff}{
\frac{d(L_\vf)_*\nu_{MS}}{d\nu_{MS}}(L_\vf(f))=
\exp\left (\scalebox{0.9}{\text{$\int_0^1(\log (\vf^{-1})')'\circ f\,df'+\tfrac12\int_0^1\left ((\log (\vf^{-1})')'\circ f\cdot f'\right )^2$}}\right ).
}
Nous aiderons le lecteur à transformer l'expression précédente en celle \eqref{eq:shavRD} de l'énoncé du théorème \ref{thm:quasi-invariance}. Nous notons $g$ la variable $L_\vf(f)=\vf^{-1}\circ f$. Grâce à ce changement de variable, on a
\begin{enumerate}
\item $(\log (\vf^{-1})')'\circ f=-(\log \vf')'\circ g\cdot (\vf^{-1})'\circ f=-(\log \vf')'\circ g\cdot\frac{1}{\vf'\circ g}$ ;
\item $f'=(\vf\circ g)'=\vf'\circ g\cdot g'$ ;
\item $df'=d(\vf\circ g)'=\vf''\circ g\cdot(g')^2+\vf'\circ g\,dg'$.
\end{enumerate}
Les égalités 1.~et 3.~donnent pour le premier terme dans \eqref{eq:deriveediff}
\begin{align*}
\int_0^1(\log (\vf^{-1})')'\circ f\,df'=\,& -\int_0^1 (\log \vf')'\circ g\,dg'-\int_0^1(\log \vf')'\circ g\cdot \frac{\vf''\circ g}{\vf'\circ g}(g')^2 \\
=\,&-\int_0^1 (\log \vf')'\circ g\,dg'-\int_0^1\left [(\log \vf')'\circ g\cdot g'\right ]^2,
\end{align*}
alors que pour le deuxième, en utilisant 1.~et 3.~on trouve :
\[\tfrac12\int_0^1\left ((\log (\vf^{-1})')'\circ f\cdot f'\right )^2=\tfrac12\int_0^1\left ((\log \vf')'\circ g\cdot g'\right )^2.\]
Donc on a
\[
\frac{d(L_\vf)_*\nu_{MS}}{d\nu_{MS}}(g)=
\exp\left (\scalebox{1}{\text{$-\int_0^1(\log \vf')'\circ g\,dg'-\tfrac12\int_0^1\left ((\log \vf')'\circ g\cdot g'\right )^2$}}\right ).
\]

\subsection{Ergodicité}

Soit $G$ l'un des groupes $\Di$ ou $\Dc$, soit $\mu$ une mesure MS sur $G$, le couple~$(G,\mu)$ définit un espace mesuré, sur lequel le groupe $H$ agit, non en préservant la mesure, mais en préservant les ensembles négligeables (l'action de $H$ \emph{quasi-préserve} la mesure $\mu$).

\begin{dfn} Soit $(X,\mu)$ un espace mesuré standard sur lequel un groupe $H$ agit en quasi-préservant la mesure $\mu$.

On dit que l'action de $H$ est \emph{ergodique} si tout ensemble mesurable invariant est négligeable ou bien conégligeable\footnote{Un ensemble mesurable $A\subset X$ est \emph{conégligeable} si le complémentaire dans $X$ est négligeable.}.
\end{dfn}

Dans ce contexte, pour des raisons liées à la théorie des représentations, Kosyak a étudié l'ergodicité de l'action de $H$ sur $(G,\mu)$ \cite{kosyak}. 

\begin{thm}[Kosyak]
\label{thm:kosyak}
Soit $G$ l'un des groupes $\Di$ ou $\Dc$, soit $\mu$ une mesure MS sur $G$ et soit $H$ un sous-groupe de $G$ qui contient tous les difféomorphismes de classe $\C{2+\alpha}$ ($\alpha<1/2$).

L'action de $H$ sur $(G,\mu)$ est ergodique.
\end{thm}

\begin{proof}
La preuve du théorème~\ref{thm:kosyak} consiste à étudier l'action non-linéaire donnée par les transformations~$S_h$ (voir \eqref{eq:actionconj}\,) sur l'espace de Wiener, mais vu de manière « hilbertienne », comme nous l'avons fait dans la partie \ref{ssc:espace-wiener} du chapitre \ref{sc:wiener}. Pour commencer, on s'intéressera à l'action de $H$ sur $(C_0(I),\W)$.

\medskip

Avec le choix de la base orthonormée de Fourier $(g_n)_{n\in \N}$ de $L^2(I)$ (voir \eqref{eq:base-fourier}\,), nous avons vu que la variable
\[x(t)=\sum_{n\in \N}Z_n\hspace{0.3pt}\Phi_n(t)\]
définit un mouvement brownien, lorsque les $Z_n$ forment une suite de variables aléatoires gaussiennes centrées de variance $1$ i.i.d., et les $\Phi_n$ sont définies par intégration des $g_n$ (voir encore la partie \ref{ssc:espace-wiener} du chapitre \ref{sc:wiener}). Nous allons nous restreindre aux fonctions $x$ appartenant à un ensemble de mesure pleine $\Omega\subs C_0(I)$ de sorte que tout $x\in \Omega$ est de classe~$\C{1/2-\ve}$ pour tout $\ve>0$.

L'expression définissant $S_h$ dit que la fonction $x$ est transformée en
\begin{multline*}
S_h(x)(t)=x(t)+h\left (c(x)\int_0^t e^{x(s)}\,ds\right )\\=
\sum_{n\in \N}\left [Z_n+\left \langle h\left (c(x)\int_0^t e^{x(s)}\,ds\right ),\Phi_n\right \rangle_{L^2}\right ]\hspace{0.3pt}\Phi_n.\end{multline*}

Fixons $x\in \Omega$ et soit $(a_n)_{n\in \N}$ une suite à support dans $\{0,\ldots,N\}$, soit $h$ tel que $h\left (c(x)\int_0^t e^{x(s)}\,ds\right )=\sum_{n\le N}a_n\hspace{0.3pt}\Phi_n(t)$, alors $h$ est une fonction de classe $\C{3/2-\ve}$ (les éléments de la base de Fourier sont de classe $\C{\infty}$), et si l'on définit~\mbox{$\vf^{-1}:=A^{-1}(h)$}, alors $\vf\in H$.

Pourquoi ce choix ? En fait, nous venons de montrer que la suite $(Z_n)_{n\in \N}$ peut être envoyée par un élément de $H$ sur la suite $(Z_n+a_n)_{n\in \N}$, qui peut être arbitrairement choisie sur toute fenêtre $\{0,\ldots,N\}$. 

En regardant les coefficients de Fourier, nous identifions $(C_0(I),\W)$ avec l'espace gaussien $(\R^{\N},\g^{\otimes\N})$, où $\gamma$ est la mesure  gaussienne standard sur $\R$. Nous avons que toute orbite de l'action sur $(\R^{\N},\gamma^{\otimes\N})$ du groupe~$\R^{fin}$ des translations à support fini, est incluse dans une orbite de l'action non-linéaire induite par $H$. Or, il est classique que l'action de $\R^{fin}$ est ergodique (pour les probabilistes, il s'agit fondamentalement de la loi du $0$--$1$ de Hewitt-Savage \cite[Lemma 1.25]{peres}\,) et on peut finalement conclure que l'action de $H$ sur $(\Di,\nu_{MS})$ est ergodique.

\medskip

La preuve pour l'action sur $(\Dc,\mu_{MS})$ suit les mêmes lignes : en notant $\tilde{\Phi}_n$ la base de $\bD_0(\T)$, pour toute suite $(a_n)_{n\in \N}$ à support dans $\{0,\ldots,N\}$ et $\theta\in \T$, on peut définir $\vf\in H$ tel que
si $h=B_0(\vf^{-1})$, alors
\begin{enumerate}
\item $h\left (\a+c(x)\int_0^t e^{x(s)}\,ds\right )=\sum_{n\le N}a_n\hspace{0.3pt}\tilde{\Phi}_n(t)$,
\item $\vf^{-1}(\alpha)=\theta$,
\end{enumerate}
où $(x,\alpha)\in C_0(\T)\times \T$ est fixé.
\end{proof}

\clearpage
%\cleardoublepage
\chapter{Décrire la dynamique aléatoire}
\label{chapter:2}
\section{Points périodiques et nombre de rotation aléatoire}
%!TEX root = master.tex

Pour comprendre la dynamique décrite par un difféomorphisme, le premier pas consiste à obtenir le plus d'informations possibles sur les orbites. Dans le cas d'une dynamique en dimension $1$, nous avons pu voir au début du chapitre~\ref{sc:diffeos} que la connaissance des orbites périodiques permet une bonne description des phénomènes globaux. Il s'avère que si l'on considère des difféomorphismes aléatoires par rapport aux mesures MS, les orbites périodiques sont de nature assez contrôlée.

\subsection{Difféomorphismes de l'intervalle}
\label{ssec:points-fixes-intervalle}

Puisque nos difféomorphismes préservent l'orientation, tout point périodique sur l'intervalle est nécessairement un point fixe.
Nous voudrions donc  décrire l'ensemble des points fixes d'un difféomorphisme aléatoire de~$[0,1]$ : est-il fini ou bien infini ?
Pour nous donner une réponse exhaustive, nous étudions le \emph{caractère} des points fixes : un point $t$ tel que $f(t)=t$ est \emph{parabolique} si $f'(t)=1$, et \emph{hyperbolique} dans le cas contraire. Pour $t$ point fixe hyperbolique, on dira que $t$ est \emph{attractif} si $f'(t)<1$ et \emph{répulsif} autrement. Les difféomorphismes qui possèdent uniquement des points fixes hyperboliques ne sont pas trop compliqués. Le théorème suivant affirme que l'on peut négliger les phénomènes paraboliques pour les difféomorphismes MS (nous entendons par ceci qu'une certaine propriété est vérifiée par un ensemble de difféomorphismes dans~$(\Di,\nu_{MS})$ de mesure $1$).

\begin{thm}
\label{thm:nomultiplicateur}
Pour tout $\la>0$, la probabilité que le difféomorphisme $f$ possède un point fixe avec multiplicateur $\la$ est nulle : 
\[\nu_{MS}(\exists\,t\,:\,f(t)=t,\,f'(t)=\la)=0.\]
En particulier un difféomorphisme MS ne possède pas de points fixes paraboliques.
\end{thm}

Nous pouvons en déduire un premier corollaire :

\begin{cor} \label{prop:points-fixes}
Un difféomorphisme MS possède un nombre fini de points fixes, tous hyperboliques.
\end{cor}
	
\begin{proof}
Il suffit de montrer que tout point fixe est isolé.
Si $\{t_k\}$ est une suite de points fixes qui s'accumule sur un point fixe $t$, par continuité de la dérivée et d'après le théorème des valeurs intermédiaires, on doit avoir~$f'(t)=1$. Donc $t$ est parabolique.
\end{proof}

La preuve du théorème~\ref{thm:nomultiplicateur} utilise fortement les propriétés du mouvement brownien et la définition des difféomorphismes MS. Nous avons recherché une démonstration qui pourtant n'utilise pas les expressions explicites des densités des variables considérées : il est assez probable que les résultats présentés dans \cite{mayo05,mayo08,yor01} puissent permettre de déduire l'énoncé du théorème, mais une telle preuve serait extrêmement technique.

\subsubsection{Cadre et notations}

Pour cette partie, on peut ne pas préciser l'espace de probabilité $(\Omega,\cB,\P)$ qui sera donc arbitraire. Pour tout mouvement brownien standard $\left (B(t)\right )_{t\in [0,1]}$ sur $\left (\Omega,\cB, \P\right )$ on définit le difféomorphisme de $[0,1]$ dans lui-même
\beqn{diffeoB}{f\,:\,t\longmapsto \dfrac{\int_0^te^{B(s)}\,ds}{\int_0^1e^{B(s)}\,ds}}
qui a la même loi qu'un difféomorphisme MS.

La définition de $f$ nous oblige à considérer une filtration plus grande que $\cF_t=\sigma\left ((B(s))_{s\le t}\right )$, la filtration naturelle associée à $B$ : pour définir $f$ au temps $t$, il faut connaître l'évolution de $B$ entre $0$ et $t$ et la moyenne totale $\int_0^1e^{B(s)}\,ds$. Il s'ensuit que si l'on définit, par exemple, la filtration régulière
\beqn{filtration}{
\cG_t=\bigcap_{t'>t}\left (\cF_{t'}\vee \sigma\left (\int_0^1e^{B(s)}\,ds\right )\right ),}
alors $f$ définit un processus sur $\left (\Omega,\cB,\P,\cG\right )$.

\subsubsection{Invariance par inversion du temps}

Il est bien connu que si $B$ est un mouvement brownien sur $(\Omega,\cB,\P)$, alors le processus «~inversé »
\[\ovl{B}(s):=B(1-s)-B(1)\]
défini encore un mouvement brownien sur $[0,1]$, mais qui est un processus par rapport à la tribu inversée \mbox{$\ovl{\cF}_t=\sigma\left ((B(s))_{1-t \le s \le 1}\right )$}. On étend cette remarque aux difféomorphismes de Malliavin-Shavgulidze : 
\begin{lem}\label{lem:temps}
Si $f$ est un difféomorphisme MS défini par l'expression \eqref{eq:diffeoB}, le difféomorphisme aléatoire
\[\ovl{f}(t)=1-\frac{\int_0^{1-t}e^{B(s)}\,ds}{\int_0^1e^{B(s)}\,ds},\]
a la même loi que $f$.
\end{lem}
\begin{proof}
En effet soit $t\in [0,1]$, on a
\begin{align*}
\log \ovl{f}'(t)&\,=B(1-t)-\log \int_0^1e^{B(s)}ds\\
&\,=B(1-t)-B(1)-\log e^{-B(1)}\int_0^1e^{B(s)}ds\\
&\,=\ovl{B}(t)-\log \int_0^1e^{B(s)-B(1)}ds\\
&\,=\ovl{B}(t)-\log \left (-\int_1^0 e^{B(1-u)-B(1)}du\right )\\
&\,=\ovl{B}(t)-\log \int_0^1e^{\ovl{B}(u)}du.
\end{align*}
Comme $\ovl{B}$ est encore un mouvement brownien, la preuve du lemme est terminée.
\end{proof}

Si l'on définit la nouvelle filtration régulière
\[\ovl{\cG}_t=\bigcap_{t'>t}\left (\ovl{\cF}_{t'}\vee \sigma\left (\int_0^1e^{\ovl{B}(u)}\,du\right )\right ),\]
alors $\ovl{f}$ définit un processus sur $(\Omega,\cB,\P,\ovl{\cG})$.

\subsubsection{Preuve du théorème~\ref{thm:nomultiplicateur}}

La preuve suivante s'inspire de \cite[corollaire 2.26]{peres}, où l'on montre qu'un mouvement brownien planaire ne passe presque sûrement pas par un point donné (dans un intervalle de temps fini).

Notons  $E$ l'événement 
\[E=\{\exists\in [0,1]\,:\,f(t)=t,\,f'(t)=\la\}.\]
Il s'avère plus simple montrer que la probabilité de l'événement
\[\tilde{E}:=\left \{\exists\,t\in [0,1]\,:\,f'(t)=\lambda,f'(f(t))=\lambda\right \}\]
est nulle. Le résultat en suivra puisque $\tilde{E}\supset E$.

\smallskip

D'après \eqref{eq:diffeoB},
\[\log f'(t)=B(t)-\log\int_0^1e^{B(s)}\,ds,\]
et nous voulons démontrer que le processus planaire
\[Z(t):=\left (B(t)-\log\int_0^1e^{B(s)}\,ds,B(f(t))-\log\int_0^1e^{B(s)}\,ds\right )\]
ne passe pas par le point de la diagonale $(\log \la,\log \la)$: l'événement $\tilde{E}$ est égal à l'événement 
\[
\left \{\exists\,t\in [0,1]\,:\,Z(t)=(\log\la,\log\la)\right \}.
\]
Le processus planaire $\left (Z(t)\right )_{t\in[0,1]}$ part du point \emph{aléatoire} 
\[\left (-\log\int_0^1e^{B(s)}\,ds,-\log\int_0^1e^{B(s)}\,ds\right )\] et on veut connaître la probabilité qu'il atteigne le point \emph{déterministe} $(\log \lambda,\log\lambda)$. Bien évidemment la loi du point de départ possède une densité par rapport à la mesure de Lebesgue sur la diagonale $\Delta\subset \R^2$. Il nous suffira donc de montrer le résultat suivant :
\begin{lem}
Soit $x\in \R$ un point donné et $f$ un difféomorphisme de Malliavin-Shavgulidze défini par un mouvement brownien $B$ comme à la ligne~\eqref{eq:diffeoB}. Alors pour Lebesgue-presque tout $y\in\R$, on a
\[\P\left (\exists\,t\in [0,1]\,:\,(B(t)-y,B(f(t))-y)=(x,x)\right)=0.\]
\end{lem}
\begin{proof}
D'après l'invariance par translations du mouvement brownien, on peut rendre le point de départ \emph{déterministe} en rendant \emph{aléatoire} le point ciblé :
\begin{multline*}
\P\left (\exists\,t\in [0,1]\,:\,(B(t)-y,B(f(t))-y)=(x,x)\right)\\ = \P\left (\exists\,t\in [0,1]\,:\,(B(t),B(f(t)))=(x+y,x+y)\right).\end{multline*}
Ce petit changement nous permet de simplifier le problème en faisant un aller-retour par le théorème de Fubini. En effet, en intégrant par rapport à $z=x+y$, la question équivaut à se demander si l'intégrale
\[\int_{\R}\P\left (\exists\,t\in [0,1]\,:\,(B(t),B(f(t)))=(z,z)\right)\,dz\]
est nulle, et on peut alors échanger l'ordre d'intégration :
\begin{align*}
&\int_{\R}\P\left (\exists\,t\in [0,1]\,:\,(B(t),B(f(t)))=(z,z)\right)\,dz\\=\,&\E\left [\int_{\R}\mathbf{1}_{\{\exists\,t\in [0,1]\,:\,(B(t),B(f(t)))=(z,z)\}}\,dz \right ]\\
=\,&\E \left [
	\Leb_\Delta\left (
		Im\left (
			B(t),B(f(t))
		\right )_{t\in[0,1]} \cap \Delta
	\right )
\right ],\numberthis{}
\label{eq:esperance5}
\end{align*}
où $\Leb_\Delta$ dénote la mesure de Lebesgue sur la diagonale $\Delta\subs \R^2$ et $Im$ l'image du processus. L'espérance \eqref{eq:esperance5} est inférieure ou égale à
\[ \E\left [\Leb\left (\mrm{Z\acute{e}ros}\left (B(f(t))-B(t)\right )_{t\in [0,1]}\right )\right ],\]
en notant $\mrm{Z\acute{e}ros}\left (B(f(t))-B(t)\right )_{t\in [0,1]}$ l'ensemble des $t\in[0,1]$ tels que $B(f(t))=B(t)$. En ayant obtenu cette majoration, on peut revenir en arrière par le théorème de Fubini :
\[ \E\left [\Leb\left (\mrm{Z\acute{e}ros}\left (B(f(t))-B(t)\right )_{t\in [0,1]}\right )\right ]=\int_0^1\P\left (B(f(t))=B(t)\right )\,dt.\]
On va montrer que pour tout $t\in(0,1)$
\begin{equation}\label{eq:bftbt}
\P\left (B(f(t))=B(t)\right )=0.
\end{equation}
Or, la probabilité que l'on ait $f(t)=t$ étant nulle, on va montrer que pour $t\in(0,1)$ fixé
\beqn{claim}{\P\left (
\{B(f(t))=B(t)\}\cap \,\{f(t)>t\}
\right )=0,}
puis le fait que la probabilité
\[\P\left (
\{B(f(t))=B(t)\}\cap\{f(t)<t\}
\right )\]
est nulle suivra par méthode de couplage (en utilisant le lemme \ref{lem:temps}\,) : si~$f(t)<t$, en « renversant le temps »,  on obtient $\ovl{f}(1-t)>1-t$, ce qui nous permet de déduire le deuxième cas à partir du premier.

Les événements $\{f(t)>t\}$ et $\{f(t)<t\}$ appartiennent à $\cG_t$ pour tout $t\in [0,1]$. Donc on peut effectivement se restreindre au cas $f(t)>t$, avec le grand avantage que sur cet événement, la loi de $B(f(t))-B(t)$ n'a pas d'atome :
on peut le constater facilement, puisque pour calculer la valeur $f(t)$ on doit avoir connaissance de $B$ uniquement jusqu'à l'instant $t$ qui est inférieure à $f(t)$ (outre que $\int_0^1e^{B(s)}\,ds$). En utilisant par exemple la propriété de Markov du mouvement brownien, on en déduit l'assertion \eqref{eq:claim}, qui permet de conclure la preuve du lemme.
\end{proof}

\subsubsection{Sur le nombre de points fixes}

On peut améliorer notre description en étudiant la probabilité qu'un difféomorphisme possède exactement~$m$ points fixes. Si l'on définit la variable aléatoire $M=\#\mrm{Fix}(f)$, on voudrait décrire la distribution de $M$. On peut facilement montrer que la probabilité que $M$ soit égal à $2$ est positive ($0$ et $1$ sont toujours des points fixes) : en effet, supposons que l'intégrale totale du mouvement brownien $B$ sur $[0,1]$ soit presque $1$ et que $B$ croisse beaucoup près de $0$, puis reste presque constant ; on trouvera ainsi un difféomorphisme qui n'a pas d'autres points fixes. Plus généralement :

\begin{prop}
Pour chaque entier positif $m\ge 2$, la probabilité que $M$ soit égal à $m$ est strictement positive.
\end{prop}

\begin{proof}
En effet, l'ensemble des difféomorphismes avec $m$ points fixes hyperboliques est d'intérieur non vide dans la topologie $C^1$. L'affirmation suit du fait que la mesure MS est strictement positive sur tout ouvert.
\end{proof}

\begin{rem}
Des simulations numériques laissent penser que la loi de $M$ est proche d'une loi géométrique de paramètre $1/2$. Cependant les statistiques changent remarquablement lorsque le paramètre $\sigma$ de la variance du brownien change : par exemple, pour $\sigma=2$ la distribution de $M$ est encore proche d'une loi géométrique de paramètre $1/2$, mais elle ne l'est plus pour des valeurs de $\sigma$ plus grandes que $4$.
\end{rem}

Une façon possible d'obtenir une estimation sur la décroissance de la loi de la variable $M$ pourrait être un simple lemme de dynamique hyperbolique $\C{1+\alpha}$ en dimension $1$, qui dit que le nombre de points fixes d'un difféomorphisme $f$ est contrôlé par la distance de $\log f'$ à $0$ sur l'ensemble des points fixes de $f$.

\begin{lem}\label{lem:estimee_fixes}
Soit $f\in\mrm{Diff}^{\hspace{0.8pt} 1+\alpha}_+(I)$ et $t_0$ un point fixe de $f$. On note $t_1=\inf\{ t>t_0\text{ t.q. }f(t)=t\}$. Alors la distance entre $t_0$ et $t_1$ est minorée :
\beqn{distance_minoree}{t_+-t_-\ge \frac{|\log f'(t_0)|^{1/\alpha}+|\log f'(t_1)|^{1/\alpha}}{\|\log f'\|_\alpha^{1/\alpha}}.}
\end{lem}

\begin{proof}
Si $t_0$ et $t_1$ coïncident, le terme de droite dans \eqref{eq:distance_minoree} est nulle puisqu'alors $t_0$ est un point fixe parabolique (cf.~le corollaire~\ref{prop:points-fixes}). Nous pouvons donc supposer $t_0< t_1$.

Par continuité de la dérivée, il existe un point $t_*$ entre $t_0$ et $t_1$ auquel la dérivée de $f$ vaut $1$.
Puisque la fonction $\log f'$ est $\C{1+\alpha}$, on a pour tout $t$ (en particulier pour $t=t_0,t_1$) :
\[|\log f'(t)|=|\log f'(t_*)-\log f'(t)|\le \|\log f'\|_\alpha \,|t-t_*|^{\alpha}\]
et donc
\[
|t-t_*|\ge \left (\frac{|\log f'(t)|}{\|\log f'\|_{\alpha}}\right )^{1/\alpha}.
\]
En écrivant $t_1-t_0=(t_1-t_*)+(t_*-t_0)$, l'inégalité précédente nous donne~\eqref{eq:distance_minoree}.
\end{proof}

Des conséquences plus intéressantes du théorème~\ref{thm:nomultiplicateur} interviendront dans la suite.

\subsection{Difféomorphismes du cercle}

Soit $(\Omega,\cB,\P)$ un espace de probabilité, soit $B=\left (B_t\right )_{t\in \T}$ un pont brownien sur $\Omega$ et $\a\in \T$ une variable de loi uniforme et indépendante de $B$. Nous définissons le difféomorphisme du cercle aléatoire de Malliavin-Shavgulidze :
\[f\,:\,t\longmapsto \dfrac{\int_0^te^{B(s)}\,ds}{\int_0^1e^{B(s)}\,ds}+\a.\]

Il est facile de remarquer que la loi de $f$ est invariante lorsque l'on fait une conjugaison par une rotation.

\subsubsection{Nombre de rotation rationnel}

Rappelons (voir partie \ref{ssc:poincare} dans l'introduction) qu'un difféomorphisme du cercle possède une orbite périodique si et seulement si son nombre de rotation est rationnel. Nous avons

\begin{prop}
Soit $q$ un entier strictement positif. La probabilité qu'un difféomorphisme de Malliavin-Shavgulidze possède une orbite de période $q$ est positive.

Plus précisément, la probabilité que le nombre de rotation soit égal à $p/q$ est positive pour tout $p/q\in \Q$.
\end{prop}

\begin{proof}
On fixe $p/q$ rationnel. L'ensemble des difféomorphismes avec nombre de rotation $p/q$ est d'intérieur non vide, puisqu'il contient les difféomorphismes avec une orbite périodique hyperbolique, dont la combinatoire est de type $p/q$. On conclut encore grâce aux propriétés de la mesure.
\end{proof}

Comme analogue à la proposition \ref{prop:points-fixes}, on peut dire que les orbites périodiques sont presque toujours hyperboliques.

\begin{prop}
\label{prop:noparabolicorbit}
La probabilité qu'un difféomorphisme MS possède une orbite périodique \emph{parabolique} est nulle.

Il en résulte que les orbites périodiques sont en nombre fini, pair, et la probabilité d'en avoir exactement $m$ est positive pour tout $m$ pair.
\end{prop}

\begin{proof}[Esquisse de la preuve]
Rappelons qu'une orbite $\{t,f(t),\ldots,f^{q-1}(t)\}$ est parabolique si et seulement si
\[f^q(t) = t\textrm{ et }\log Df^q(t)=\sum_{k=0}^{q-1} B(f^k(t))-q\,\log \int_0^1 e^B=0.\]

On appelle \emph{multiplicateur} de l'orbite la quantité $Df^q(t)$ (qui ne dépend pas du point $t$ dans l'orbite). On montrera que pour tout $\lambda>0$, la probabilité qu'il existe une orbite périodique de multiplicateur égal à $\la$ est nulle.

Pour cela, on peut procéder comme dans la preuve du théorème \ref{thm:nomultiplicateur}, en regardant les événements
\begin{align*}
E=&\,\left \{\exists\,t\,:\,f^q(t)=t\trm{ et }Df^{q}(t)=\la\right \} \\
\tilde{E}=&\,\left \{\exists\, t\in [0,1]\,:\, Df^{q}(t)=\la\trm{ et }Df^q(f(t))=\la\right \}.
\end{align*}
La preuve procède de manière semblable jusqu'à la ligne  \eqref{eq:bftbt} : on veut montrer que pour tout~$t\in [0,1]$
\[\P\left (B(f^q(t))-B(t)\right )=0.\]
 Au lieu de considérer seulement deux cas différents ($\{f(t)>t\}$ et $\{f(t)<t\}$), nous devons faire face à une combinatoire plus complexe. En effet, si l'on souhaite appliquer la propriété de Markov pour $B$, il faut que la condition suivante soit vérifiée :
\[\max_{k<q}f^k(t)<f^q(t)<1.\]
Il est encore possible, par méthode de couplage, de se ramener toujours à ce cas : il suffit en fait d'opérer une conjugaison par une rotation appropriée (voir la figure \ref{fig:combi}\,).
\begin{figure}[ht]
\[
\includegraphics[scale=.5]{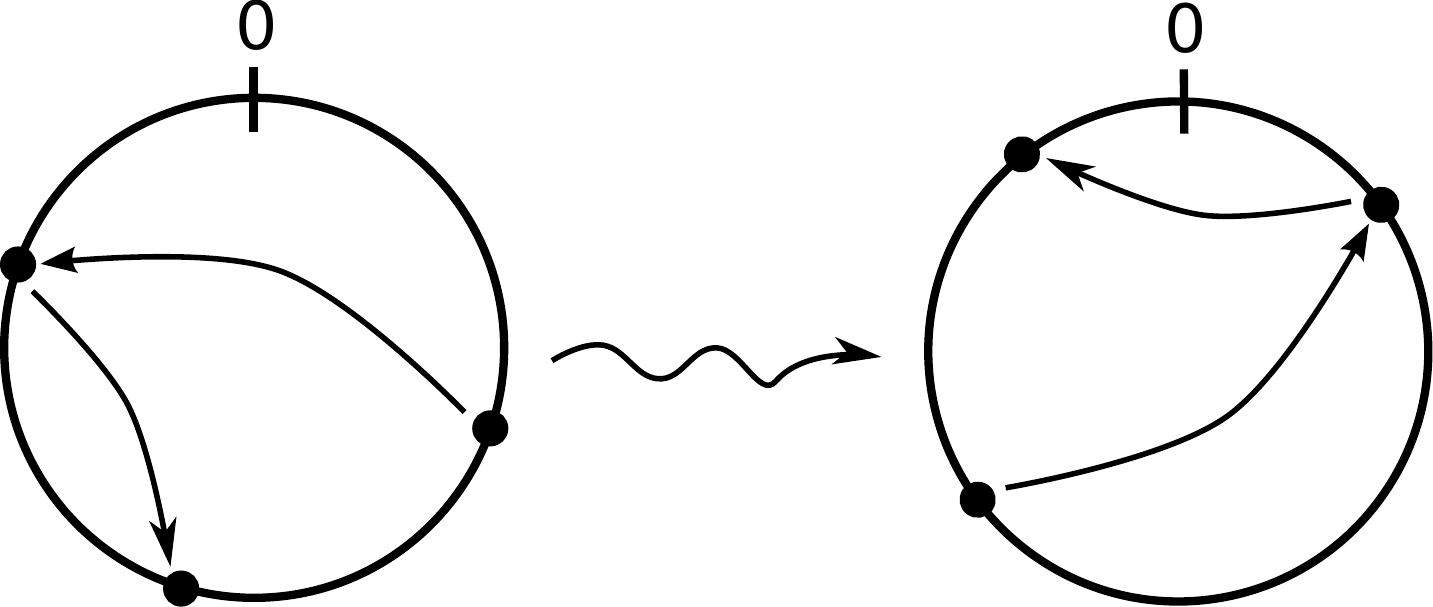}
\]
\caption{Preuve du lemme \ref{prop:noparabolicorbit} : combinatoire d'une orbite et choix convenable de la rotation}\label{fig:combi}
\end{figure}
\end{proof}

Par les mêmes arguments, on peut démontrer le résultat suivant qui est intéressant en soi :
\begin{thm}\label{thm:free}
Soient $f_1,\ldots,f_n$ des difféomorphismes de Malliavin-Shavgulidze indépendants. Alors le sous-groupe qu'ils engendrent est presque sûrement libre.
\end{thm}

\begin{proof}
Soit $w=w(f,g)$ un mot non-trivial en $f$ et $g$. On veut montrer que la probabilité que $w(f,g)=id$ est nulle, lorsque $f$ et $g$ sont deux difféomorphismes MS indépendants. Soit $\ell$ la longueur du mot $w$ et notons $w_j$ le suffixe de $w$ de longueur $j$ (avec $j\le \ell$).
Il nous suffit de démontrer que la probabilité~$\P(\log Dw\equiv 0)$ est égale à $0$.

Or, on peut écrire $\log Dw$ sous la forme
\[
\log Dw=\sum_{j=0}^{\ell-1}\epsilon_j \left [ B_{i_j}\circ \psi_{j} \circ w_j-\log \int_{\T}\exp(B_{i_j})\right ],
\]
avec $i_j\in\{1,2\}$, $\epsilon_j\in \{\pm 1\}$, $\psi_j\in\{id, f^{-1}, g^{-1}\}$ et $B_1$, $B_2$ sont deux ponts browniens indépendants qui définissent $f$ et~$g$. Pour tous $C_1$, $C_2\in\R$, on définit l'événement 
\[A_{C_1,C_2}=\left \{\int_{\T}\exp(B_i)=\exp(C_i),\,i=1,2 \right\}.\]
En moyennant sur les possibles valeurs des variables $\int_{\T}\exp(B_i)$, on peut écrire
\begin{align*} 
&\P\left (\forall\, t\in \T, \sum_{j=0}^{\ell-1}\epsilon_j \left [B_{i_j}(\psi_j(w_j(t)))-\log \int_{\T}\exp(B_{i_j})\right ]=0\right )\\
=\,& \E\left[\P
\left (
\forall t\in \T, 
\sum_{j=0}^{\ell-1}\epsilon_j B_{i_j}(\psi_j(w_{j}(t)))=K(w;C_1,C_2)
\,\middle\vert \,A_{C_1,C_2}
\,\right )\right ]
\end{align*}
avec 
\[K(w;C_1,C_2)=\left (\#_{f^{-1}}w-\#_fw\right )\cdot C_1+\left (\#_{g^{-1}}w-\#_g w\right )\cdot C_2\] 
et où $\#_\gamma w$ dénote  le nombre de $\gamma\in \{f^{\pm 1}, g^{\pm 1}\}$ dans le mot $w$.

Or, pour $t$ fixé et presque tous les $C_1$, $C_2$,
\[ 
\P\left (\sum_{j=0}^{\ell-1}\epsilon_j  B_{i_j}(\psi_j w_{j}(t)))=K(w;C_1,C_2)\,\middle\vert\,A_{C_1,C_2}\right )=0.
\]
Ceci est facile à observer : supposons que $\psi_{\ell-1}(w_{\ell-1}(t))$ soit le dernier point sur le cercle à la gauche de $0$, parmi les $\{\psi_j(w_j(t))\}$ (sinon, on fait un couplage avec le processus que l'on obtient par la conjugaison avec une rotation appropriée). Alors la loi de $B_{i_{\ell-1}}(\psi_{\ell-1}(w_{\ell-1}(t)))$ n'as pas d'atome.
\end{proof}

\subsubsection{Nombre de rotation irrationnel}

Une question plus délicate, qui certainement mériterait une réponse, est la suivante :
\begin{q}\label{q:irrational}
Les difféomorphismes avec nombre de rotation \emph{irrationnel}, forment-ils un ensemble de mesure strictement positive ?
\end{q}
La topologie dans ce cas ne peut pas nous venir en aide, car l'ensemble des difféomorphismes avec nombre de rotation $\alpha$ irrationnel fixé forment un fermé (contractile), d'intérieur vide dans $\Dc$ (voir \cite{herman}\,). En utilisant le théorème de conjugaison différentiable (théorème \ref{thm:conjugaison-differentiable} dans le chapitre introductif), Herman \cite{herman_lebesgue} démontre le résultat suivant :

\begin{thm}
\label{thm:herman_lebesgue}
Soit $f$ un difféomorphisme du cercle de classe $\C{3}$, on définit $K_f\subset \T$ comme étant la réunion des paramètres $\la\in \T$ pour lesquels le nombre de rotation de $R_\la\circ f$ est irrationnel. Alors, la mesure de Lebesgue de $K_f$ est strictement positive.
\end{thm}

\begin{figure}[ht!]
\centering
\includegraphics[scale=.4]{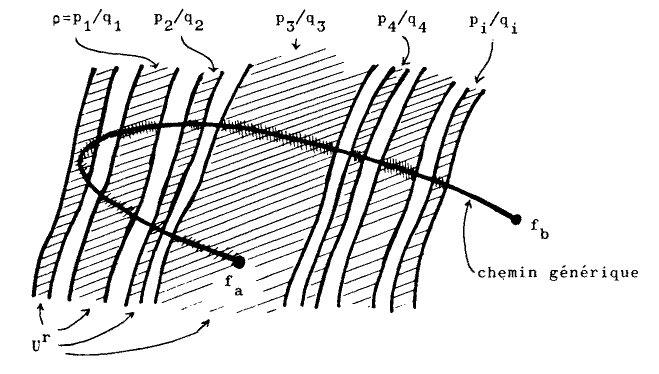}
\caption{Image prise de \cite{herman_lebesgue}. Elle représente un chemin générique dans $\Dc$ qui traverse des zones où le nombre de rotation est constant.}
\end{figure}

Pour montrer ce théorème, Herman explique que la fonction \emph{nombre de rotation} $\rho$ est différentiable (au sens de Gateaux) aux points constitués par le difféomorphismes $f$ avec nombre de rotation irrationnel $\alpha$ et $C^1$-conjugués à la rotation, et sa différentielle, définie sur l'algèbre de Lie des champs de vecteurs, est donnée par
\beqn{gateaux}{v\in \mrm{Vect}^1(\T)\mto \int_{\T}(Dh\circ f\circ h^{-1})\,v\circ h^{-1},}
où $h$ est le difféomorphisme conjuguant $f$ à $R_\alpha$ : $f=h^{-1}\circ R_\alpha \circ h$.

\begin{ex}
\label{ex:arnold}
Ce phénomène s'observe par exemple dans la célèbre \emph{famille d'Arnol'd}. Pour tout~$a\in ]-1/2\pi, 1/2\pi[$ et~$\la\in \T$, on définit la fonction $f_{a,\la}(x)=x+\la+a\,\sin(2\pi\,x)$. Les ensembles~$A_{p/q}$ des paramètres~$a$ et~$\la$ tels que~$\rho(f_{a,\la})=p/q$ ont la forme de petites langues. Pour des petites valeurs de $a$, les $t$ pour lesquels $\rho(f_{a,t})$ est irrationnel est presque $1$.

Cette dernière affirmation suit du fait que le bord des régions sur lesquelles le nombre de rotation est constant et rationnel, est Lipschitz au voisinage des pointes ($a\sim 0$, $\rho(f_{a,t})$ rationnel).
\end{ex}

Nous avons fait des simulations numériques (voir appendice \ref{app:simulations}\,), qui montrent des langues quand on regarde les mesures MS : pour tout couple $(\alpha,\sigma)$, on peut calculer $\mu_{\sigma}\left (\rho^{-1}(\alpha)\right )$. Nous pensons que, contrairement à ce que l'on constate dans l'exemple d'Arnol'd, la mesure des difféomorphismes avec nombre de rotation irrationnel est nulle pour tout $\sigma>0$. 
L'argument géométrique donné par Herman \cite{herman_lebesgue} ne pourra pas être exploité, car l'on s'attend à ce que la fonction $\rho$ ne soit pas Lipschitz aux difféomorphismes avec nombre de rotation diophantien, mais $(1/2-\ve)$-H\"older.

En revanche, il est facile de voir que les difféomorphismes avec nombre de rotation irrationnel fixé, forment une partie de mesure nulle (il est en soi évident que les nombres irrationnels pour lesquels la mesure est positive, doivent être en nombre au plus dénombrable). Pour cela, nous rappellerons un résultat dû à Arnol'd (voir \cite{herman}\,) :

\begin{prop}
\label{prop:arnold}
Soit $f$ un homéomorphisme du cercle qui fixe $0$ et soit $\alpha$ un nombre irrationnel. Il existe un seul~$\lambda=\lambda(\alpha)\in\T$ tel que le nombre de rotation de $R_\la\circ f$ soit égal à~$\alpha$. En outre, la fonction $\la(\alpha)$ est continue en la variable $f$.
\end{prop}

\begin{rem}
On peut aussi montrer que pour tout nombre rationnel $p/q$, il existe deux fonctions $\la_+$ et $\la_-$, telles que $\rho(R_{\la_+}\circ f)=\rho(R_{\la_-}\circ f)=p/q$ et pour tout $\la>\la_+$ (resp. $<\la_-$) le nombre de rotation de $R_{\la}\circ f$ est strictement supérieur (resp. inférieur). Les difféomorphismes en ces paramètres sont dits \emph{semi-stables}.

La proposition \ref{prop:noparabolicorbit} montre en particulier que la probabilité qu'un difféomorphisme de Malliavin-Shavgulidze avec nombre de rotation $p/q$ soit semi-stable est nulle. Nous en déduisons, par dénombrabilité des nombres rationnels, que les difféomorphismes semi-stables forment une partie négligeable dans $\Dc$.
\end{rem}

\medskip

D'après la proposition \ref{prop:arnold}, on peut penser à $\rho^{-1}(\alpha)$ comme à une hypersurface dans $\Dc$, non lisse en général, mais avec une régularité qui dépend de la régularité du difféomorphisme et, surtout, de la condition diophantienne que $\alpha$ vérifie : ces deux paramètres donnent un contrôle sur la différentielle du nombre de rotation (voir \eqref{eq:gateaux}\,).

Pour ce qui nous concerne, il suffit de remarquer que l'on peut voir $\rho^{-1}(\alpha)$ comme l'image d'une section continue $s_\alpha$ du fibré en cercles $\Dc$ sur $\T \backslash\Dc$,  où $\T\backslash\Dc$ s'identifie, via l'application $B$ donnée par \eqref{eq:B}, à $B^{-1}\left (C_0(\T)\times \{0\}\right )$. Par définition de la mesure MS (voir \eqref{eq:MSc}\,), on peut appliquer le théorème de Fubini à l'intégrale
\[\int_{C_0(\T)}\int_{\T}\mathbf{1}_{\{s_\alpha(x)\}}(t)\,dt\,d\W_{0,\sigma}(x)=\mu_{\sigma}\left (s_\alpha\left (C_0(\T)\right )\right ),\]
qui est pourtant nulle.

\begin{prop}
Pour tout nombre irrationnel $\alpha$, la mesure des difféomorphismes avec nombre de rotation $\alpha$ est nulle :~$\mu_\sigma\left (\rho^{-1}(\alpha)\right )=0$. \vspace{-1.5em}\cvd
\end{prop}

\subsubsection{Renormalisation et mesures}

L'une des clefs pour comprendre la nature des difféomorphismes du cercle est la \emph{renormalisation}. Cachée dans les travaux de Herman et Yoccoz \cite{herman,yoccoz}, elle permet de mieux expliquer pourquoi le résultat de conjugaison différentiable est vrai. Il s'agit de regarder de plus en plus localement la dynamique décrite par $f$. Il est possible de définir la renormalisation $\cR f$ de tout difféomorphisme $f$ sans point fixe : en prenant l'intervalle $I_1$ contenant $0$ et dont les extrémités sont $f(0)$ et $f^{q_1}(0)$ (rappelons de la partie \ref{ssc:arith} que l'entier $q_1$ est choisi comme le plus petit $k$ tel que $f^k(0)$ appartient à l'intervalle de gauche délimité par $0$ et $f(0)$ et tel que $f^{k+1}(0)$ appartient à celui de droite), on considère le difféomorphisme du cercle $I_1/_{f^{q_1}(0)\sim f(0)}$ qui est défini par $f$ sur l'intervalle d'extrémités $f^{q_1}(0)$ et $0$ et par $f^{q_1}$ sur l'autre (voir figure \ref{fig:reno}\,). On opère par la suite une conjugaison affine pour envoyer $I_1/_{f^{q_1}(0)\sim f(0)}$ sur le cercle $\R/\Z$ standard : nous avons défini le renormalisé $\cR f$. 

\begin{figure}[ht]
\[
\includegraphics[scale=.6]{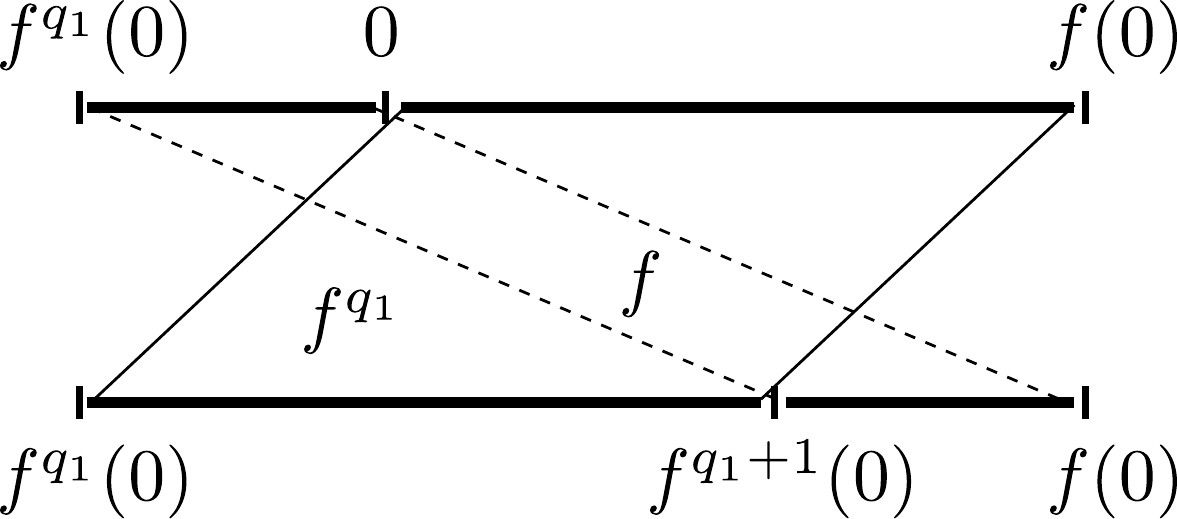}
\]
\caption{Construction de la renormalisation de $f$}\label{fig:reno}
\end{figure}

Notons $T$ l'homothétie $T(t)=\left (f(0)+1-f^q(0)\right )\,t$ qui envoie $[0,1]$ sur $[-f^{q}(0),f(0)]$ en fixant $0$. Avec une terminologie plus récente, $\cR f$ est défini par la \emph{paire d'applications qui commutent} $(T^{-1}fT, T^{-1}f^qT)$, qui est l'échange d'intervalles généralisé sur $[0,1]$ en figure \ref{fig:reno}.

La théorie de Herman-Yoccoz peut s'expliquer en disant que les renormalisations successives $\cR^n
f$ se rapprochent d'une rotation lorsque le nombre de rotation est diophantien (les difféomorphismes infiniment renormalisables sont exactement ceux avec nombre de rotation irrationnel).

\begin{q}
Quelle est l'image de la mesure $\mu_{MS}$ sous l'opérateur $\cR$ ?
\end{q}

Nous ne savons pas répondre à cette question. Nous pouvons cependant faire une remarque facile. Nous avons vu, dans le chapitre précédent, que lorsque l'on compose à droite par un difféomorphisme $\vf$, l'événement~\eqref{eq:modkosyak} est de mesure $1$ pour la probabilité image. Lorsque l'on renormalise $f$, on fait apparaitre une itération $f^q$ dans la définition de $\cR f$, qui modifie cette loi des logarithmes itérés : pour tout $t\in [0,1]$ on a
\begin{align*}
&\varlimsup_{\ve\rightarrow 0}\frac{\log D(\cR f)(t+\ve)-\log D(\cR f)(t)}{\sqrt{2\ve\log\log(1/\ve)}}\\
=\,&\text{\scalebox{0.8}{$\begin{cases}
\displaystyle \sqrt{T'(t)} \\ \trm{si }t\in\left  ]-\dfrac{f^q(0)}{f(0)+1-f^q(0)},0\right [ ,\\
&\\
\displaystyle \sum_{k=0}^{q-1}\varlimsup_{\ve\rightarrow 0}\frac{\log Df(f^k(T(t+\ve)))-\log Df(f(T(t)))}{\sqrt{2\ve \log\log (1/\ve)}}
\\ \trm{si }t\in \left] 0,\dfrac{f(0)}{f(0)+1-f^{q}(0)}\right [.\\
\end{cases}$}}\\
\\
=\,&
\text{\scalebox{0.8}{$\begin{cases}
\displaystyle \sqrt{f(0)+1-f^{q}(0)} \\ \trm{si }t\in\left  ]-\dfrac{f^q(0)}{f(0)+1-f^q(0)},0\right [ ,\\
&\\
\displaystyle \sqrt{f(0)+1-f^{q}(0)}\sum_{k=0}^{q-1}\sqrt{Df^{k}\left ((f(0)+1-f^{q}(0))\,t\right )} \\ \trm{si }t\in \left] 0,\dfrac{f(0)}{f(0)+1-f^{q}(0)}\right [.\\
\end{cases}$}}
\end{align*}

Nous en déduisons 

\begin{prop}
Les mesures $\cR_*\mu_{MS}$ et $\mu_{MS}$ ne sont pas équivalentes.\vspace{-1.5em}\cvd
\end{prop}

La proposition précédente donne seulement une réponse partielle au problème : si l'on s'attend à ce que la renormalisation « lisse » les difféomorphismes, il faut croire que ce module de continuité local donné par la loi de logarithmes itérés doit tendre vers $0$ lorsque l'on applique $\cR$ une infinité de fois. En d'autres mots, on doit penser que l'opérateur $\cR$ concentre le support de la mesure $\mu_{MS}$ de plus en plus dans un voisinage~$\C{1+1/2-\ve}$ des rotations : soit $n$, disons, un nombre impair, et regardons encore la loi des logarithmes itérés après $n$ renormalisations d'un difféomorphisme  renormalisable $n$ fois. On a
\begin{align*}
&\varlimsup_{\ve\rightarrow 0}\frac{\log D(\cR^n f)(t+\ve)-\log D(\cR^n f)(t)}{\sqrt{2\ve\log\log(1/\ve)}}\\
=\,&\text{\scalebox{0.8}{$
\begin{cases}
\displaystyle \sqrt{f^{q_{n-1}}(0)+1-f^{q_n}(0)}\sum_{k=0}^{q_{n-1}-1}\sqrt{Df^{k}\left ((f^{q_{n-1}}(0)+1-f^{q_{n}}(0))\,t\right )} 
 \\ \trm{si }t\in\left  ]-\dfrac{f^{q_n}(0)}{f^{q_{n-1}}(0)+1-f^{q_n}(0)},0\right [ ,\\
\\
\displaystyle \sqrt{f^{q_{n-1}}(0)+1-f^{q_n}(0)}\sum_{k=0}^{q_n-1}\sqrt{Df^{k}\left ((f^{q_{n-1}}(0)+1-f^{q_{n}}(0))\,t\right )} \\ \trm{si }t\in \left]0,\dfrac{f^{q_{n-1}}(0)}{f^{q_{n-1}}(0)+1-f^{q_n}(0)}\right [.\\
\end{cases}$}}
\end{align*}
Si le nombre de rotation de $f$ est irrationnel, on peut regarder le comportement lorsque $n$ tend vers l'infini. Si $f$ n'est pas minimal, alors le facteur 
\[\sqrt{f^{q_{n-1}}(0)+1-f^{q_n}(0)}\]
converge vers un nombre strictement positif, et on devrait penser que les sommes convergent (d'après le résultat de Norton expliqué dans la partie \ref{ssc:Denjoy} de ce chapitre). En revanche, si $f$ est minimal, il faut regarder finement la compétition entre ces deux limites : le facteur~$\sqrt{f^{q_{n-1}}(0)+1-f^{q_n}(0)}$ tend vers $0$, mais la somme diverge.

\medskip

Pourquoi sommes-nous convaincus que le nombre de rotation doit être rationnel presque sûrement ? Le fait est que pour construire $\cR f$ on se retrouve toujours avec beaucoup de choix. Nous allons être plus précis.

Considérons le processus $(f,\cG_t)$, où $\cG_t$ est la filtration introduite à la ligne  \eqref{eq:filtration} (avec les modifications appropriées en considérant que l'on travaille avec un \emph{pont} brownien). Pour définir $\cR f$, nous n'avons pas besoin de connaitre entièrement le pont brownien entre $0$ et $1$.
\begin{enumerate}
\item Le premier point de l'orbite est $f(0)=\a$ : il est choisi uniformément par rapport à Lebesgue et il est indépendant de $\cG_t$ pour tout $t$.
\item Ensuite la définition de $f(\a)$ nécessite la connaissance de $\a$ et de $B$ entre $0$ et $\a$. Supposons 
\begin{equation}
\label{eq:condition_direction}
\tag{$\ast$}
\a<f(\a)<1
\end{equation}
(dans le cas contraire on pourra faire un couplage avec $1-f(1-t)$). Au lieu de considérer $(B_t)_{t\in[0,\a]}$, on peut se restreindre à la variable
\[\dfrac{\int_0^\a e^B}{\int_0^1 e^B}+\a\,:\]
nous avons que $f^2(0)$ est mesurable par rapport à la tribu 
\[\sigma\left (B_1,\int_0^1e^B , \a, \int_0^\a e^B\right ).\]
\item Sous la condition précédente, $q_1$ définit un temps d'arrêt qui est mesurable par rapport à la tribu
\[\sigma\left (B_1,\int_0^1e^B , \a, \int_0^\a e^B, (B_{t})_{t\in [\a, 1[} \right ).\]
Le temps $q_1$ est fini si et seulement si $f$ n'a pas de point fixe. On en déduit que la variable  $f^{q_1+1}(0)$ est mesurable par rapport à la tribu
\[\mathcal{T}_1=\sigma\left (B_1,\int_0^1e^B , \a, \int_0^\a e^B, (B_{t})_{t\in [\a, f^{q_1}(0)]} \right ).\]
\item Par conséquent, l'intervalle $I_1$ sur lequel on définit la renormalisation $\cR f$, ainsi que le couple d'applications~$\left (f,f^{q_1}\right )$, sont mesurables par rapport à la tribu $\cT_1$. Ce qui est important est qu'ils sont presque indépendants de l'évolution de $B$ dans l'intervalle délimité par $f^{q_1}(0)$ et $f(0)$ qui contient $0$ : il faut seulement connaitre
\[\int_{f^{q_1}(0)}^{1} e^B \trm{ et }\int_0^{f(0)} e^B.\]
Le point $f^{q_1+1}(0)$ joue le rôle de $\a$ dans le premier point.
\item En répétant les étapes précédentes, on en déduit que l'intervalle $I_n$ d'extrémités $f^{q_{n-1}}(0)$ et $f^{q_n}(0)$ qui contient $0$ et le couple $\left (f^{q_{n-1}},f^{q_n}\right )$ sont \emph{presque indépendants} de l'évolution de $B$ dans l'intervalle $I_n$, modulo deux conditions de bords comme dans le point précédent.
\end{enumerate}

En conclusion, le fait que le difféomorphisme aléatoire $f$ soit renormalisable une infinité de fois est indéterminé si l'on connait l'évolution de $B$ partout, sauf en un petit intervalle autour de $0$. On aimerait alors disposer d'une espèce de propriété de Markov pour pouvoir appliquer une loi du $0$--$1$. Les images de la figure \ref{fig:qdistrib1} sont assez suggestives sous ce point de vue.
\newpage
\section{Centralisateurs}
%!TEX root = master.tex

\subsection{Quelques rappels sur les centralisateurs des difféomorphismes de l'intervalle}

Rappelons d'abord un résultat classique, dû originairement a Sternberg et dans cette version $\C{1+\alpha}$ par Chaperon \cite{chaperon}, mais présenté sous cette forme par Yoccoz (voir \cite{navasEM,booknavas,diviseurs}\,).

\begin{dfn}
On dit que deux difféomorphismes $f$ et $g$ de classe $\C{1+\alpha}$ définissent le même \emph{germe} s'ils coïncident dans un voisinage suffisamment petit de l'origine. On définit $\cG^{\hspace{0.8pt}1+\alpha}_+(\R,0)$ comme étant l'\emph{espace des germes des difféomorphismes~$\C{1+\alpha}$} de la droite qui fixent $0$ et préservent l'orientation.

On dit qu'un germe $[g]$ est \emph{hyperbolique} si la dérivée en $0$ de $g$ définissant $[g]$ est différente de $1$ (observer que cette notion est toujours bien définie).
\end{dfn}

\begin{thm}[Théorème de linéarisation de Sternberg]
\label{thm:sternberg}
Soit $[g]\in \cG^{\hspace{0.8pt}1+\alpha}_+(\R,0)$ un germe hyperbolique et $a:=g'(0)$. Il existe alors un germe $[h]\in \cG^{\hspace{0.8pt}1+\alpha}_+(\R,0)$ tel que 
\beqn{sternberg}{
\tag{$\ast$}
h'(0)=1  \trm{ et } h(g(x))=a\hspace{0.3pt}h(x) \trm{ dans un voisinage de l'origine.}}

En outre, si $[h_1]$ est un germe $\C{1}$ qui vérifie \eqref{eq:sternberg}, alors $[h_1]$ est de classe~$\C{1+\alpha}$ et définit le même germe que $h$.
\end{thm}

\begin{proof} On va transformer le problème de la recherche d'un tel $h$ en un problème de point fixe. Pour cela, linéarisons l'équation \eqref{eq:sternberg} : soit $\delta>0$ suffisamment petit tel que $g$ soit défini sur l'intervalle $[-\delta,\delta]$, alors on peut écrire $g(x)=ax+\phi(x)$ et $h(x)=x+\psi_0(x)$, avec les fonctions $\phi$ et $\psi_0$ de classe $\C{1+\alpha}$ et nulles en $0$ jusqu'au premier ordre. L'équation \eqref{eq:sternberg} s'écrit sur $[-\delta,\delta]$ :
\[ax+\phi(x)+\psi_0(g(x))=ax+a\,\psi_0(x)\]
ou bien
\beqn{sternberg_lin}{\psi_0=\tfrac1a \psi_0\circ g+\tfrac1a \phi.}
On définit l'espace vectoriel $E_\delta$ des fonctions $\psi$ de classe $\C{1+\alpha}$ sur $[-\delta,\delta]$, telles que $\psi(0)=\psi'(0)=0$. On munit $E_\delta$ de la norme
\[\|\psi\|=\|\psi'\|_{\alpha}.\]
L'espace $E_\delta$ est alors un espace de Banach et la fonction $\phi$ appartient à $E_\delta$.
L'équation \eqref{eq:sternberg_lin} nous suggère de définir l'endomorphisme $S_\delta$ de $E_\delta$ :
\[S_\delta[\psi]=\tfrac{1}{a}\psi\circ g,\]
de sorte que $\psi_0$ doit être solution de l'équation 
\beqn{sternberg_lin2}{\psi=S_\delta[\psi]+\tfrac1a \phi.}
Quitte à passer à $g^{-1}$, on peut supposer $a<1$, puis on pose
$C(\delta):=\sup_{[-\delta,\delta]}g'$.
Placés maintenant dans le bon cadre analytique pour résoudre \eqref{eq:sternberg_lin2}, on montre que pour $\delta$ suffisamment petit, $S_\delta$ est une contraction. Soient $t$ et $s$ dans $[-\delta,\delta]$, on a
\begin{multline*}
\left \vert S_\delta[\psi]'(t)-S_\delta[\psi]'(s)\right \vert = \frac{1}{a}\left \vert (\psi\circ g)'(t)-(\psi\circ g)'(s)\right \vert \\=\frac{1}{a}\left \vert \psi'\left (g(t)\right )\,g'(t)-\psi'\left (g(s)\right )\,g'(s)\right \vert\\
\le \frac{1}{a}\left \vert \psi'\left (g(t)\right )\,g'(t)-\psi'\left (g(t)\right )\,g'(s)\right \vert+\frac{1}{a}\left \vert \psi'\left (g(t)\right )\,g'(s)-\psi'\left (g(s)\right )\,g'(s)\right \vert\\
\le\frac{\left \vert \psi'\left (g(t)\right )\right \vert}{a}\,\|g'\|_\alpha\,|t-s|^{\alpha}+\frac{g'(s)}{a}\,\|\psi\|\, |g(t)-g(s)|^{\alpha}\\
=\frac{\left \vert \psi'\left (g(t)\right )-\psi'(0)\right \vert}{a}\,\|g'\|_\alpha\,|t-s|^{\alpha}+\frac{g'(s)}{a}\,\|\psi\|\, |g(t)-g(s)|^{\alpha}\\
\le\frac{\|\psi\|\,|g(t)|^\alpha}{a}\,\|g'\|_\alpha\,|t-s|^{\alpha}+\frac{C(\delta)}{a}\,\|\psi\|\, C(\delta)^{\alpha}\,|t-s|^{\alpha}\\
\le\frac{\|\psi\|}{a}\,\left [C(\delta)^\alpha |t|^\a \,\|g'\|_\alpha+C(\delta)^{1+\a}\right ]\,|t-s|^\a\\ \le \frac{\|\psi\|}{a}\,\left [C(\delta)^\alpha \delta^\a \,\|g'\|_\alpha +C(\delta)^{1+\a}\right ]\,|t-s|^\a.
\end{multline*}
Il suffira donc de prendre $\delta$ tel que
\[C(\delta)^\alpha \,\delta^\a \,\|g'\|_\alpha +C(\delta)^{1+\a}<a,\]
choix bien possible car le premier terme tend vers $0$ et $C(\delta)$ tend vers $a<1$ lorsque $\delta$ tend vers $0$.

\medskip

Il reste à montrer l'unicité. Soit $h_1$ un germe $\C{1}$ qui vérifie les bonnes propriétés, alors $g\circ h_1^{-1}=h_1^{-1}\circ M_a$ (où~$M_a$ est la multiplication par $a$), d'où
\[h\circ h_1^{-1}\circ M_a=h\circ g\circ h_1^{-1}=M_a \circ h\circ h_1^{-1},\]
qui veut dire que $h\circ h_1^{-1}$ commute avec $M_a$. Soit $t$ un point suffisamment près de $0$, alors
\begin{align*}
h\left (h_1^{-1}(t)\right )&=\lim_{n\rightarrow \infty}\frac{h\left (h_1^{-1}(a^nt)\right )}{a^n}\\
&=t\,\lim_{n\rightarrow\infty}\frac{h\left (h_1^{-1}(a^nt)\right )}{a^nt}\\
&=t \left (h\circ h_1^{-1}\right )'(0)=t.
\end{align*}
\end{proof}

\begin{cor}
Soit $f$ un difféomorphisme $\C{1+\a}$ de l'intervalle avec $0$ et $1$ seuls points fixes. On suppose $a:=f'(0)\ne 1$, alors il existe un difféomorphisme $h$ de $[0,1[$ qui est de classe $\C{1+\a}$ sur $[0,1[$ et vérifie
\beqn{sternberg2}{
\tag{$\ast\ast$}
h'(0)=1\trm{ et }h(f(x))=a\hspace{0.3pt}h(x)\trm{ pour tout }x\in [0,1[.
}
En outre, si $h_1$ est un difféomorphisme $\C{1}$ qui vérifie \eqref{eq:sternberg2}, alors $h_1$ est de classe $\C{1+\a}$ et coïncide avec $h$.
\end{cor}

\begin{proof}
Par le théorème de linéarisation de Sternberg, il existe un voisinage de $0$ et un difféomorphisme~$\tilde{h}$ défini sur ce voisinage tel que \eqref{eq:sternberg} soit satisfait. Ceci entraîne que la définition du difféomorphisme $h$ souhaité est entièrement déterminée : en effet, tout voisinage de $0$ contient un domaine fondamental pour l'action de~$f$, de la forme $J_0=[x_0,f(x_0)[$ (si l'on suppose, sans perte de généralité, $a>1$). Soit maintenant $x\in [0,1[$ un point quelconque. Il existe un unique entier $n=n(x)$ tel que $f^{n(x)}(x)\in J_0$. On définit $h(x):=a^{-n(x)}\hspace{0.3pt}\tilde{h}(f^{n(x)}(x))$. La fonction~$n=n(x)$ vérifie $n(f(x))=n(x)-1$. Il en découle l'égalité
\[h(f(x))=a^{-n(f(x))}\hspace{0.3pt}\tilde{h}\left (f^{n(f(x))}(f(x))\right )=a^{-n(x)+1}\hspace{0.3pt}\tilde{h}\left (f^{n(x)}(x)\right )=a\hspace{0.3pt}h(x).\]
\end{proof}

Puisque tout difféomorphisme $f$ vérifiant les hypothèses du corollaire précédent est conjugué à $M_a$, la multiplication linéaire par $a=f'(0)$, il s'ensuit, en particulier, que le centralisateur $\C{1}$ de $f$ sur $[0,1[$ est conjugué par le difféomorphisme $h$ au centralisateur $\C{1}$ de la multiplication par $f'(0)$ sur $[0,+\infty[$. La description de ce dernier est assez élémentaire.

Soit $\vf$ un difféomorphisme sur $[0,+\infty[$ qui commute avec $M_a$. Supposons, sans perte de généralité, que $a$ soit inférieur à $1$. Pour tout $x\in[0,+\infty[$ et $n>0$ entier, 
\beqn{centmult}{
\vf(x)=\frac{\vf\left (a^n\hspace{0.3pt}x\right )}{a^n} \trm{ et }\vf'(x)=\vf'\left (a^n\hspace{0.3pt}x\right ).
}
Puisque, pour tout $x$ fixé, la suite $\left (a^n\hspace{0.3pt}x\right )_{n\in \N}$ tend vers $0$, \eqref{eq:centmult} force la dérivée de $\vf$ à être constante.

On vient de démontrer que le centralisateur $\C{1}$ de $M_a$ est égal à $\R^+$, qu'on peut voir comme étant le flot à un paramètre du champ de vecteurs $a\hspace{0.3pt}\frac{\rd}{\rd x}$.

On en déduit que le centralisateur $\C{1}$ de $f$ sur $[0,1[$ est égal au groupe à un paramètre $\{f_t\}$ engendré par le \emph{champ de vecteurs de Szekeres} $X_f$ associé à $f$, défini par $X_f(x)=h(x)\,Dh^{-1}(x)\,\log a$. On remarquera que le flot de $X_f$ au temps $t$ est défini par $f_t(x)=h^{-1}\left (a^t\hspace{0.3pt}h(x)\right )$.

\medskip

Nous nous rapprochons de l'étude de l'\emph{invariant de Mather}. Même si Kopell dans \cite{kopell} avait entièrement expliqué ce dont nous avons principalement besoin, nous allons encore suivre l'exposition de Yoccoz \cite{diviseurs}.

Soit $f$ un difféomorphisme de l'intervalle avec $0$ et $1$ comme uniques points fixes, supposés hyperboliques. On note $a=f'(0)$ et $b=f'(1)$. Prenons en exemple le cas $a>1>b$. Soient $X_f$ et $Y_f$ les champs de vecteurs de Szekeres associés à $f$ sur les intervalles $[0,1[$ et $]0,1]$ respectivement. Si $g\in \Di$ commute avec $f$, d'après ce qui précède, il doit exister deux temps $\tau$ et $\tau'$ tels que $g=f_\tau=f^{\tau'}$, où $\{f_t\}_{t\in \R}$ est le flot engendré par $X_f$ et $\{f^t\}_{t\in \R}$ celui par $Y_f$. Soient $x$, $y\in ]0,1[$, on définit $W^{x,y}$ comme étant la fonction de « raccord » entre $\{f^t\}$ et $\{f_t\}$ :
\beqn{mather}{
f^t(y)=f_{W^{x,y}(t)}(x).
}
Si $h$ et $k$ sont les difféomorphismes qui linéarisent $f$ sur $[0,1[$ et $]0,1]$ respectivement, l'équation \eqref{eq:mather} devient 
\[k^{-1}\left (b^t\hspace{0.3pt}k(y)\right )=h^{-1}\left (a^{W^{x,y}(t)}\hspace{0.3pt}h(x)\right ),\]
d'où
\[W^{x,y}(t)=\log_a\frac{h\left (k^{-1}\left (b^t\hspace{0.3pt}k(y)\right )\right )}{h(x)}.\]
La fonction $W^{x,y}(t)$ est un difféomorphisme $\C{1+\a}$ de $\R$ dans $\R$ qui vérifie de plus
\begin{align*}
W^{x,y}(t+1) & = \log_a\frac{h\left (k^{-1}\left (b\hspace{0.3pt}\,b^t\hspace{0.3pt}k(y)\right )\right )}{h(x)}\\
&=\log_a \frac{h\left (f\left (k^{-1}\left (b^t\hspace{0.3pt}k(y)\right )\right )\right )}{h(x)}\\
&=\log_a\frac{a\hspace{0.3pt}h\left (k^{-1}\left (b^t\hspace{0.3pt}k(y)\right )\right )}{h(x)}=1+W^{x,y}(t) 
\end{align*}
En d'autres termes, $W^{x,y}$ commute avec les translations entières et, par passage au quotient de cette action, il définit un difféomorphisme du cercle de classe $\C{1+\a}$, que l'on note encore $W^{x,y}$.

On remarquera aussi, avec des vérifications analogues, que pour tous $s$, $t$, $t'$,
\[W^{f_t(x),f^{t'}(y)}(s)=W^{x,y}(s+t')-t.\]
Puisque $g=f_\tau=f^{\tau'}$, pour tous $x$, $y\in ]0,1[$, $W^{x,y}$ doit vérifier
\beqn{mather2}{W^{x,y}(s+\tau')-\tau=W^{g(x),g(y)}(s).}

\begin{prop}[\cite{diviseurs}\,]
\label{prop:mather}
Le difféomorphisme $g=f_\tau=f^{\tau'}$ commute avec $f$ sur $[0,1]$ si et seulement si, pour tous~$x$,~$y\in ]0,1[$,~\eqref{eq:mather2} est vérifiée. C'est-à-dire, le centralisateur $\C{1}$ de $f$ dans $\Di$ est isomorphe au stabilisateur de~$W^{x,y}$ sous l'action de $\R^2$ : 
\beqn{mather3}{(t,t')\cdot W^{x,y}(s)=W^{x,y}(s+t')-t.}
\end{prop}

On peut voir que le stabilisateur de $W^{x,y}$ est toujours soit cyclique, engendré par $\left ( \frac{1}{n},\frac{1}{n}\right )$, où $\frac{1}{n}$ correspond à la sous-période de $W^{x,y}-id$, ou bien isomorphe à $\R$ entier.

\begin{rem}
Les difféomorphismes de classe $\C{1+\a}$ qui ont un centralisateur $\C{1}$ trivial forment une partie ouverte et dense.
\end{rem}

Étendons maintenant l'étude précédente aux difféomorphismes hyperboliques avec un nombre arbitraire de points fixes, $0=p_0<q_0<p_1<q_1<\ldots<p_m<q_m=1$. Le raisonnement est valable pour tout intervalle de la forme $[p_i,q_i]$ ou $[q_i,p_{i+1}]$. Remarquons que pour tout point $p_i$ (ou $q_i$), le théorème de linéarisation de Sternberg \ref{thm:sternberg} permet de trouver un difféomorphisme $h_i$ de classe $\C{1+\a}$ défini sur $]q_{i-1},q_i[$ (ou $]p_i,p_{i+1}[$) à valeurs dans $\R$, tel que
\beqn{mather4}{h_i'(p_i)=1\trm{ et }h_i(f(x))=f'(p_i)\,h_i(x)\trm{ pour tout }x\in]q_{i-1},q_i[.}
Cela permet d'associer à $f$ un champ de vecteurs de Szekeres sur l'intervalle $]q_{i-1},q_i[$, de sorte que les flots $f^t$ sur~$]q_{i-1},p_i]$ et $f_t$ sur~$[p_i,q_i[$ se raccordent en $p_i$.

Si $g$ est un difféomorphisme de classe
$\C{1}$ qui centralise $f$, il correspond à un certain $f^t$ sur~$]q_{i-1},p_i]$ et donc, par \eqref{eq:mather4}, à $f_t$ sur~$[p_i,q_i[$. Nous en déduisons que la définition d'un tel $g$ sur un quelconque des intervalles~$[p_i,q_i]$ ou~$[q_i,p_{i+1}]$ force la définition de $g$ sur l'intervalle entier. Ainsi, la proposition \ref{prop:mather} implique que le centralisateur $\C{1}$ de $f$ sur $[0,1]$ est isomorphe au stabilisateur de 
$(W_{0,0}^{x_0,y_0},W_{0,1}^{x_0',y_1},W_{1,1}^{x_1,y_1'},\ldots)$ par l'action diagonale de $\R^2$ définie sur chaque composante par \eqref{eq:mather3} (nous avons noté $W_{i,i}^{x,y}$, $W_{i,i+1}^{x,y}$ les $W^{x,y}$ correspondant aux intervalles $[p_i,q_i]$ et~$[q_i,p_{i+1}]$\,).

\subsection{Centralisateur $\C{1}$ d'un difféomorphisme hyperbolique du cercle}

Soit $f$ un difféomorphisme de classe $\C{1+\a}$ qui possède uniquement des orbites périodiques hyperboliques. On suppose que le nombre de rotation de $f$ est égal à $p/q$. Par souci de simplicité on suppose que les multiplicateurs de chaque orbite sont différents l'un de l'autre. Alors le difféomorphisme $F:=f^q$ est un difféomorphisme du cercle avec uniquement des points fixes hyperboliques. Un résultat de Kopell \cite{kopell} (voir aussi \cite{diviseurs}\,) montre que les centralisateurs $\C{1}$ de $f$ et $F$ coïncident. La preuve de ce fait n'est pas trop difficile, mais elle ne rajoute aucun élément nouveau vraiment instructif pour nous  à présent. Ainsi, nous pouvons nous restreindre au cas où le nombre de rotation de $f$ vaut $0$. Il est clair que pour tout point fixe $x_*\in \T$, le difféomorphisme $f$ définit un difféomorphisme $f_*$ de l'intervalle $I_*:=[x_*,x_*+1]$, et tout élément du centralisateur de $f$ sur $\T$, centralise aussi $f_*$ sur $I_*$, si l'on suppose qu'un tel élément fixe $x_*$. La grande différence dans l'étude des centralisateurs  sur le cercle et sur l'intervalle se manifeste par une partie de torsion : il peut exister un difféomorphisme $g$ qui permute cycliquement les points attractifs et les points répulsifs de $f$, tout en commutant avec $f$. Remarquons que sous notre hypothèse sur les multiplicateurs des orbites (donc sur les valeurs de la dérivée de $f$ aux points fixes), ceci ne peut cependant pas se produire.

\begin{prop}[\cite{diviseurs}\,]
Avec les notations précédentes, le centralisateur $\C{1}$ de $f$ est isomorphe à un sous-groupe fermé de $\R$. De plus, il existe un ouvert dense de $f$ pour lequel le centralisateur $\C{1}$ est trivial.
\end{prop}

\begin{thm}\label{thm:central}
Pour $\mu_{MS}$-presque tout difféomorphisme $f$ du cercle avec nombre de rotation rationnel, le centralisateur $\C{1}$ de~$f$ est trivial.
\end{thm}

Nous utiliserons le lemme suivant dont la preuve est très simple.
\begin{lem}\label{l:simple}
Soit $f^q:[x_0,x_1]\to [x_1,x_2]$ un difféomorphisme aléatoire obtenu en itérant $q$ fois un difféomorphisme MS $f$ avec dérivées aux extrémités fixées. Soient $(a_n)_{n\in \N}\subs ]x_0,x_1[$ et $(b_n)_{n\in \N}\subs ]x_1,x_2[$ deux suites (décroissantes). Alors la probabilité qu'il existe $n$ tel que $f^q(a_n)=b_n$ est nulle.
\end{lem}

Pour montrer le théorème \ref{thm:central}, nous suivons la preuve de Kopell dans \cite{kopell}, du fait que les contractions avec centralisateur $\C{1}$ trivial constituent une partie dense parmi les contractions.

\begin{proof}
Sans perte de généralité, on peut supposer que le nombre de rotation de $f$ soit $0$. Considérons l'ensemble des $f$ tels que
\begin{enumerate}
\item $f$ préserve un intervalle $I=[x,y]$ et sa restriction à $[x,y[$ est une contraction. Fixons $x_0\in ]x,y[$ et posons, pour tout $n\in \Z$, $x_n=f^n(x_0)$.
\item On connait $f$ (et donc $Df$) sur $J=[x,x_{1}]$ et $J'=[x_{0},y]$ et on sait que $f(x_0)=x_1$.
\end{enumerate}
Nous allons conditionner la mesure de \ms par rapport à cet événement.

La deuxième condition implique que l'on connait les champs de vecteurs $X$ et $Y$ associés à $f$ autour de $x$ et de $y$, sur les intervalles $J$ et $J'$. En particulier on connait les flots à un paramètre $f_t$ et $f^t$ sur les intervalles $J$ et $J'$ respectivement.

\medskip

Soit $h$ un difféomorphisme de $I$ tel que $h_{\vert{J\cup J'}}=f_{\vert{J\cup J'}}$ et tel que $h$ appartient à un flot à un paramètre $\{h^t\}$ de difféomorphismes $\C{1}$.

Tout difféomorphisme aléatoire $f$ détermine un homéomorphisme $\psi=\psi_f$ de $]x,y]$ tel que
\begin{enumerate}
\item $\psi_{\vert [x_1,y]}=id$,
\item $\psi h = h\psi$ ,
\item $f=\psi^{-1}h\psi$.
\end{enumerate} 
En effet, il suffit de définir le difféomorphisme $\beta=h\circ f^{-1}$ de $[x_2,x_1]$ dans lui-même, tangent à l'identité en $x_1$ et~$x_2$ : si l'on pose $\psi _{\vert [x_2,x_1]}=\beta$, on détermine $\psi$ entièrement, en utilisant la propriété 2. 

Le lemme \ref{l:simple} implique que presque sûrement, il n'existe pas de $n>0$ tel que $f(h^{1/n}(x_0))=h^{1/n}(x_1)$. Donc presque sûrement, il n'existe pas de $n>0$ tel que
\beqn{comm}{\psi(h^{1/n}(x_1))=h^{1/n}\left (\psi(x_1)\right ),}
puisque $\psi=h\circ f^{-1}$ sur $[x_2,x_1]$.

Supposons que le centralisateur $\C{1}$ de $f$ ne soit pas trivial : il existe $n>0$ tel que $f^{1/n}$ est différentiable en $x$ et par unicité du flot à un paramètre on doit avoir 
$f^{1/n}=\psi^{-1}h^{1/n}\psi$. Alors $f^{1/n}$ commute avec $f=\psi^{-1}h\psi$, qui est égal à $h$ sur $[x,x_1]$. Cela implique que sur l'intervalle $[x,x_1]$, on a $f^{1/n}=h^{1/n}$.

Donc $\psi h^{1/n}=h^{1/n}\psi$ sur $[x,x_1]$, qui est possible uniquement avec probabilité nulle d'après \eqref{eq:comm}.
\end{proof}

\begin{rem}
La preuve du théorème précédent est assez « molle » : elle ne requiert pas d'expliciter la mesure MS, en utilisant seulement ses propriétés qualitatives. Notamment, Les propriétés données par la proposition \ref{prop:noparabolicorbit} et le lemme \ref{l:simple}, avec la propriété de régularité Hölder presque sûre, suffisent.
\end{rem}
\newpage
\section{Exemples de Denjoy}
\label{ssc:Denjoy}

Nous abordons ici l'étude des difféomorphismes avec nombre de rotation irrationnel. On doit rappeler qu'à présent, on ne sait pas si ces difféomorphismes forment une partie de mesure positive, par rapport à la mesure MS (question \ref{q:irrational}\,). Cependant, pour traiter, au sens de la mesure, l'ensemble des difféomorphismes avec nombre de rotation irrationnel $\a$ fixé, nous introduisons des mesures MS «~locales ».

La proposition \ref{prop:arnold} nous a permis de voir $\rho^{-1}(\a)$ comme étant une hypersurface dans $\Dc$. Nous noterons aussi $\mrm{Diff}^{\hspace{0.8pt}r}_{+,\alpha}(\T)$ l'ensemble des difféomorphismes de classe $\C{r}$ et nombre de rotation irrationnel $\a$ fixé.

\begin{dfn}
La mesure MS \emph{locale} sur $\mrm{Diff}^{\hspace{0.8pt}1}_{+,\alpha}(\T)$ est définie comme étant la mesure image de~$\W_{0,\sigma}$ par rapport à la section $s_\a$ qui à $f_0(t)=\frac{\int_0^te^{x(s)}\,ds}{\int_0^1e^{x(s)}\,ds}$ associe $s_{\a}(f_0)=R_{\la(\a)}\circ f_0$, l'unique translatée de $f_0$ avec nombre de rotation $\a$. On notera $\mu_{\sigma,\a}$ la mesure $(s_\a)_*\W_{0,\sigma}$.
\end{dfn}

Par rapport aux mesures $\mu_\sigma$, les mesures locales $\mu_{\sigma,\a}$ ont l'avantage de nous faire travailler avec des ponts browniens. Si l'on conditionnait un difféomorphisme de \ms à avoir nombre de rotation irrationnel, la mesure que l'on obtiendrait sur $\mrm{Diff}^{\hspace{0.8pt}r}_{+,\alpha}(\T)$ ne serait pas pour autant « homogène » : la densité en un point $x=\log Df-\log Df(0)$ dépendrait de la mesure de Lebesgue du Cantor $K_f$ défini dans l'énoncé du théorème de Herman \ref{thm:herman_lebesgue}.

\begin{rem}
Pour tout $\a\in \R-\Q$, la fonction $s_\a$ est un homéomorphisme qui préserve les classes de régularité. Il s'ensuit que presque sûrement, un difféomorphisme $f\in \mrm{Diff}^{\hspace{0.8pt}r}_{+,\alpha}(\T)$  est de classe $\C{1+w}$ ($w$ module de Lévy \eqref{eq:levy0}\,), que sa variation n'est pas bornée, etc.
\end{rem}

D'après la remarque précédente, presque sûrement le théorème de Denjoy (théorème \ref{thm:denjoy} dans l'introduction) ne peut pas s'appliquer. La question à laquelle nous souhaiterions répondre est la suivante :

\begin{q}
Pour chaque $\a$ irrationnel fixé, quelle est la mesure $\mu_{\sigma,\a}$ des difféomorphismes qui sont conjugués topologiquement à la rotation ? De quelle manière cette mesure dépend de la condition diophantienne satisfaite par $\a$ ? Est-il vrai, par exemple, que si $\a$ satisfait une condition diophantienne d'ordre $\delta$, alors, presque sûrement, tout difféomorphisme est $\C{\frac{1}{2}-\ve-\delta}$-conjugué à la rotation ?
\end{q}

\smallskip

Pour tout $\a$ irrationnel et $\tau<1/2$, il est possible de construire des exemples de Denjoy avec nombre de rotation $\a$ et de classe $\C{1+\tau}$. Ces exemples ont été traités systématiquement pour la première fois par Herman \cite[Chapitre X]{herman}. Nous les représentons ici, en nous inspirant plutôt de Norton \cite{norton2}.

On pose $\sigma=1/\tau$ et on choisit $M=M(\sigma)$, entier suffisamment grand tel que $\left (\frac{M+1}{M}\right )^{\sigma}<\frac{5}{4}$. Pour tout $n\in \Z$, on définit $\ell_n:=\frac{A}{(|n|+M)^\sigma}$, où $A$ a été choisi pour que $\sum_{n\in \Z}\ell_n=1$ (cette condition assurera que le Cantor minimal invariant est un ensemble de mesure de Lebesgue nulle).

Pour tout $n$ fixé, soit $I_n$ un intervalle fermé dans $\T$ et de longueur $\ell_n$. On force les intervalles $I_n$ à être ordonnés combinatoirement sur le cercle comme une orbite de la rotation d'angle $\a$ (on peut par exemple faire correspondre $n\a$ à $I_n$).
Ce choix, et la densité de l'orbite $(n\a)_{n\in \Z}$, implique que $K:=\T-\bigcup_{n\in \Z}I_n$ est un ensemble de Cantor (de mesure de Lebesgue nulle).

On passe ensuite à la construction du difféomorphisme. Pour tout $n$, on définit un difféomorphisme $f_n$ de classe $\C{1+\tau}$ de $I_n$ sur $I_{n+1}$ et on pose $f\big\vert_{I_n}=f_n$ ; la densité de $\bigcup_{n\in \Z}I_n$ implique que $f$ s'étend à un homéomorphisme de $\T$ dans $\T$. Afin que $f$ soit globalement de classe $\C{1+\tau}$ sur $\T$, il faut prendre quelques précautions pour définir les $f_n$.

Soit $\vf:[0,1]\to [0,4]$ une fonction $\C{\infty}$ de support $\left [\frac{1}{3},\frac{2}{3}\right ]$ et de moyenne globale $1$. Si $x\in I_n=[a_n,b_n]$, on définit
\[g_n(x)=1+\frac{\ell_{n+1}-\ell_n}{\ell_n}\,\vf\left (\frac{x-a_n}{\ell_n}\right ).\]
Alors $g_n$ est une fonction strictement positive de classe $\C{\infty}$ et $g_n(a_n)=g_n(b_n)=1$. En outre, le choix du $M$ implique que $g_n$ est toujours inférieur, disons, à $2$. Il s'ensuit que tout difféomorphisme
\[f_n(x):=a_{n+1}+\int_{a_n}^{x}g_n(y)\,dy\]
a une dérivée uniformément bornée par $2$ et elle tend vers $1$ lorsque $|n|$ tend vers l'infini.

Ces considérations permettent de définir globalement $f$, comme expliqué plus haut, et alors $f$ sera un difféomorphisme de classe $\C{1}$ en tout point $x\in \T$. Remarquons que la dérivée de $f$ vaut $1$ sur le Cantor $K$.

Il reste à vérifier que $f$ est de classe $\C{1+\tau}$. Pour cela, il suffit de montrer que les normes $\C{\tau}$ des $g_n$ sont uniformément bornées. Des estimations désagréables, mais élémentaires, montrent que pour $n$ assez grand en module, la constante de Hölder $\C{\tau}$ de $g_n$ est bornée par $\frac{2\sigma}{A^\tau}\|\vf\|_{\mrm{Lip}}$ ($\|\vf\|_{\mrm{Lip}}$ est la constante de Lipschitz de $\vf$ sur $[0,1]$).

\medskip

Ces exemples ont de plus la propriété, pour nous remarquable, que la dimension de boîte supérieure du Cantor $K$ est exactement $\tau$. À ce propos, nous rappelons que si $(X,d)$ est un espace métrique compact et $Y\subs X$ une partie bornée, on définit la \emph{dimension de boîte supérieure} de $Y$ de la manière suivante : soit, pour tout~$\ve>0$,~$N(\ve)$ le nombre minimal  de boules de rayon $\ve$ nécessaires pour recouvrir $Y$ ; alors la dimension de boîte supérieure~$\overline{\dim}_B(Y)$ de $Y$ est définie comme étant la limite supérieure
\[\overline{\dim}_B(Y)=\varlimsup_{\ve\rightarrow 0}\frac{\log N(\ve)}{\log(1/\ve)}.\]
La limite à droite dans l'expression n'existe pas forcément. Si cette limite existe, on l'appelle plus simplement \emph{dimension de boîte} de $Y$.

Il est connu (voir \cite{norton2}\,) que pour une partie compacte 
$Y$ sur le cercle $\T$, de mesure de Lebesgue nulle, la dimension de boîte supérieure de $Y$ est égale à l'\emph{infimum} des $s>0$ tel que $\sum \Leb(I)^s<\infty$, où la somme est faite sur les composantes connexes $I$ du complémentaire de $Y$ dans $\T$.

Dans notre exemple, la somme $\sum \Leb(I)^s=\sum \ell_n^s$ est finie si et seulement si $s>\tau$, d'où la validité de notre affirmation. Plus généralement, Norton montre qu'il y a des contraintes pour qu'un Cantor dans $\T$ soit l'ensemble minimal invariant d'un difféomorphisme de classe $\C{1+\tau}$ :

\begin{thm}[Norton]\label{t:norton}
Soit $f$ un difféomorphisme du cercle avec un ensemble minimal qui est un Cantor $K$, tel que la dimension de boîte supérieure de $K$ soit égale à $\sigma\in ]0,1[$. Alors $f$ n'est pas de classe $\C{1+\tau}$ pour tout $\tau>\sigma$.
\end{thm}

Observons que dans notre contexte aléatoire, si $f$ est un difféomorphisme tiré au hasard par rapport à la mesure MS locale, et si l'on suppose les deux propriétés supplémentaires suivantes :
\begin{enumerate}
\item $f$ possède un Cantor invariant $K$,
\item $Df=1$ sur $K$,
\end{enumerate}
alors, la dimension de boîte supérieure de $K$ doit être exactement $1/2$.

En effet, la dimension de boîte des zéros d'un pont brownien est presque sûrement  égale à $1/2$ (voir \cite{peres}\,), et donc il en est de même pour tout ensemble de niveau non-vide. Il n'est pas alors possible d'exclure l'existence d'un ensemble de probabilité non-nulle de difféomorphisme qui vérifient les propriétés 1.~et 2.~ci-dessus.

\medskip

Un raffinement célèbre de la dimension de boîte est la dimension de Hausdorff.
Dans \cite{kra} le résultat de Norton est amélioré, sous la condition que $Df$ soit égal à $1$ partout sur le Cantor invariant.

\begin{thm}[Kra~--~Schmeling]\label{thm:kra-sch}
Pour tout $\tau>0$, si $K$ est un Cantor minimal invariant pour un difféomorphisme du cercle $\C{1+\tau}$ et de nombre de rotation $\a$ qui satisfait une condition diophantienne d'ordre $\delta$, alors la dimension de Hausdorff de $K$ est supérieure ou égale à $\tau/(1+\delta)$.
\end{thm}

Par suite, des exemples aléatoires de difféomorphismes de Denjoy avec $Df=1$ sur le Cantor sont possibles seulement si le nombre de rotation $\a$ est de type Roth. Ces nombres $\a$ forment une partie de mesure pleine par rapport à la mesure de Lebesgue, et donc nous ne rencontrons encore aucune contrainte.
\clearpage
\section{Appendice : Simulations numériques}
\label{app:simulations}
En commençant l'étude des difféomorphismes aléatoires, nous avons eu une sensation de dépaysement : quel type de difféomorphisme devons-nous imaginer ?

Nous avons alors eu recours à une aide « numérique » : l'un des avantages à travailler avec les mouvements browniens est en effet leur simplicité quand on souhaite faire des simulations !

Dans la suite, nous présentons quelques images, en rajoutant après les lignes de code qui nous ont permis de les obtenir : elles présentent principalement deux parties, une pour générer pseudo-aléatoirement un difféomorphisme qui suit la loi de Malliavin-Shavgulidze et une autre pour calculer le nombre de rotation. Le langage est Python. Les simulations ont été effectuées sur les serveurs du PSMN de l'École Normale Supérieure de Lyon.

\newpage

\begin{figure}[!htp]
\begin{minipage}[b]{0.45\linewidth}
\[
\includegraphics[scale=.3]{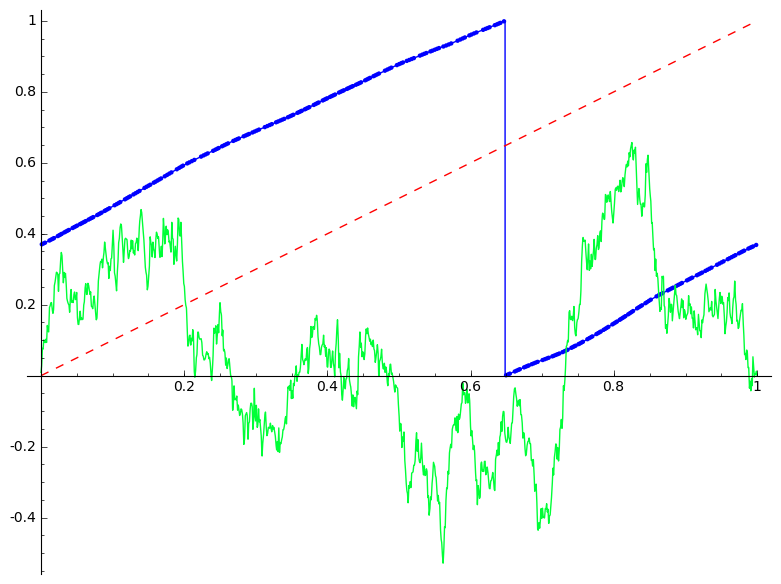}
\]
\caption{Paramètre $\sigma = 0.6$\\ Nombre de rotation $[0,2,1,2,2,1,1,1,\ldots]\sim \frac{27}{73}$}
\end{minipage}
\quad
\begin{minipage}[b]{0.45\linewidth}
\[
\includegraphics[scale=.3]{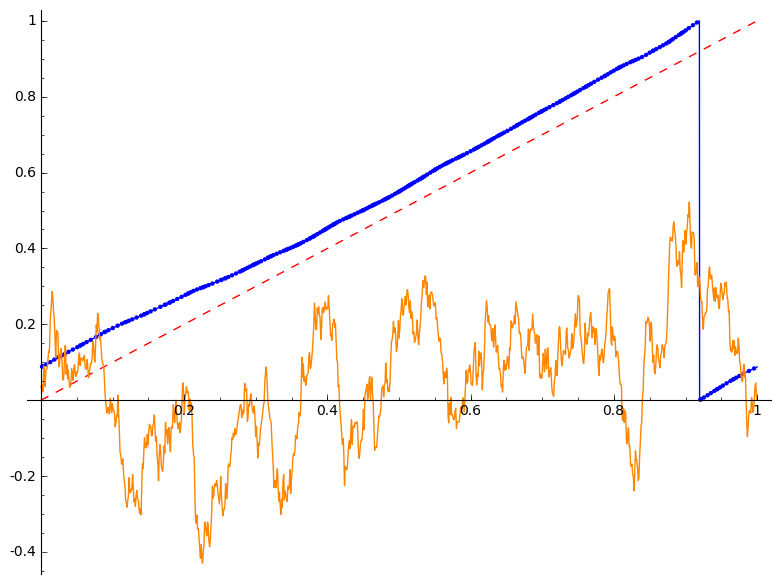}
\]
\caption{Paramètre $\sigma = 0.8$\\ Nombre de rotation $[0,15,3,4,1,3,1,15,\ldots]\sim \frac{1216}{18619}$}
\end{minipage}
\end{figure}

\begin{figure}[ht]
\begin{minipage}[b]{0.45\linewidth}
\[
\includegraphics[scale=.3]{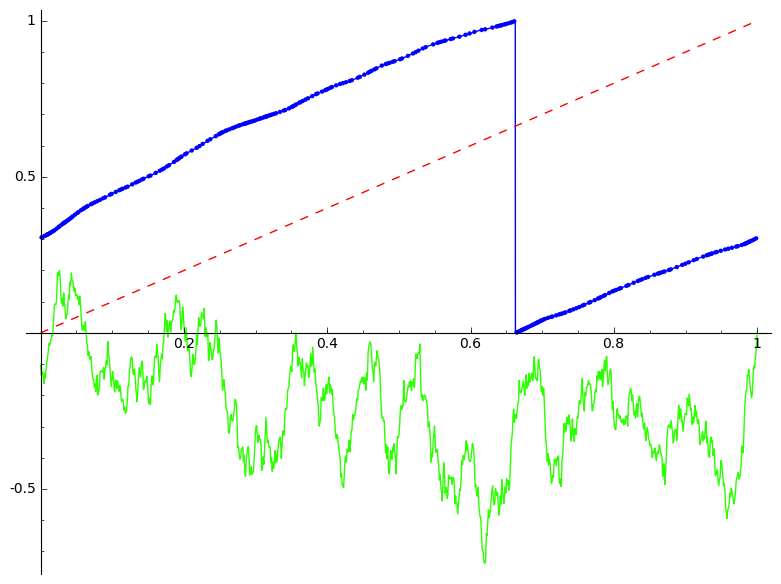}
\]
\caption{Paramètre $\sigma = 1.6$ \\ Nombre de rotation $[0,2,1,6,2,4,1,2,\ldots]\sim \frac{231}{662}$}
\end{minipage}
\quad
\begin{minipage}[b]{0.45\linewidth}
\[
\includegraphics[scale=.3]{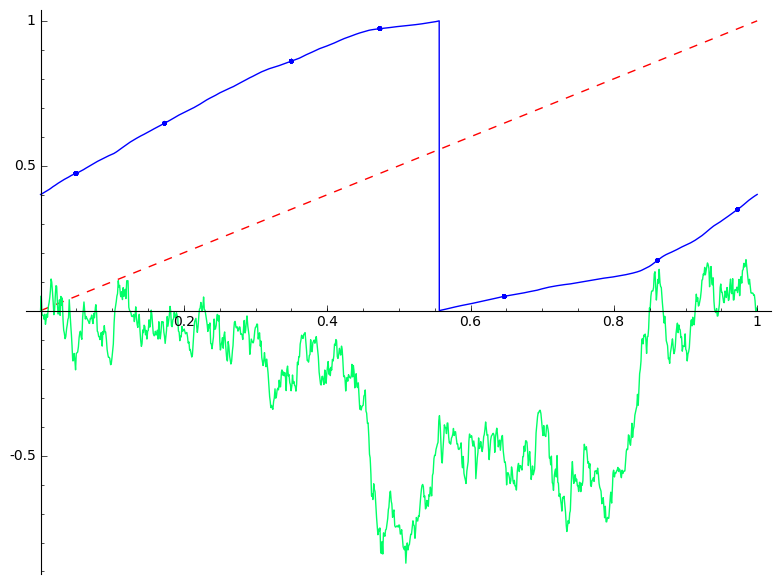}
\]
\caption{Paramètre $\sigma = 2.3$\\
Nombre de rotation $\frac{3}{7}$}
\end{minipage}
\end{figure}

\begin{figure}[ht]
\begin{minipage}[b]{0.45\linewidth}
\[
\includegraphics[scale=.3]{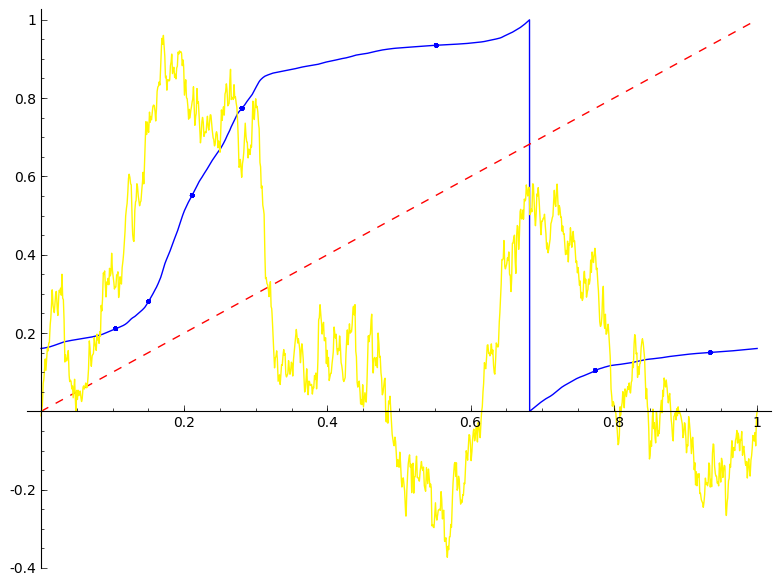}
\]
\caption{Paramètre $\sigma = 3.4$\\ Nombre de rotation $\frac{3}{7}$}
\end{minipage}
\quad
\begin{minipage}[b]{0.45\linewidth}
\[\includegraphics[scale=.3]{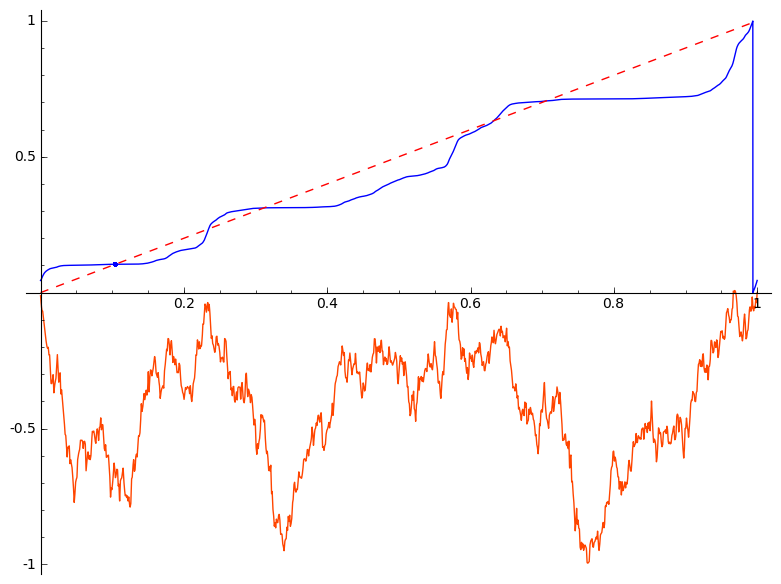} \]
\caption{Paramètre $\sigma = 9.1$\\ Nombre de rotation $0$}
\end{minipage}
\end{figure}

\newpage

Dans les images nous avons tracé le graphe du difféomorphisme pseudo-aléatoire $f$, avec le pont brownien utilisé dans sa définition. Le graphe du pont brownien est colorié par une couleur entre $0$ et $1$ qui correspond à l'image du point $0$ par $f$. Sur le graphe de $f$ nous avons aussi marqué les points sur l'orbite de $0$ : $f^k(0)$, pour~$k\in \{200,\ldots,499\}$.

\bigskip

Cet algorithme nous a permis de générer $1050$ difféomorphismes aléatoires pour chaque loi $\mu_\sigma$, avec~$\sigma\in \{0.25, 0.5, \ldots, 7.75\}$, dont nous avons calculé le nombre de rotation. La figure \ref{fig:rotation} se lit ainsi : chaque ligne correspond à un paramètre $\sigma$, pour lequel nous avons ordonné de manière croissante les nombres de rotation obtenus pour cette valeur. Pour distinguer les rationnels des irrationnels, nous avons reporté chaque nombre de rotation $\rho$ par une couleur entre $0$ et $1$, qui correspond à la partie fractionnaire de $\rho$ multiplié par un très grand nombre rationnel ($2\cdot 3\cdot 4\cdot 5\cdot 6\cdot 7\cdot 8\cdot 9\cdot 10\cdot 11\cdot 13\cdot 14\cdot 15\cdot 17\cdot 19\cdot 23$).

\begin{figure}[ht]
\[
\includegraphics[scale=.65]{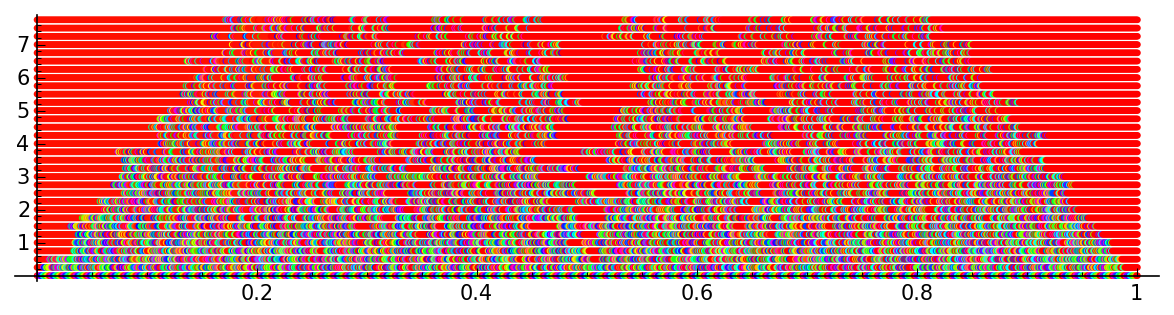}
\]
\caption{\emph{Rouge rationnel}}\label{fig:rotation}
\end{figure}

Même si l'on observe beaucoup de points qui ne sont pas rouges (surtout pour des petites valeurs de $\sigma$), nous ne croyons pas que les nombres de rotation irrationnels aient $\mu_\sigma$-mesure positive : il est numériquement impossible de distinguer un nombre de rotation \emph{irrationnel} d'un nombre de rotation \emph{rationnel} avec très grand dénominateur. Pour faire une comparaison avec ce qui se passe dans le même type de figure provenant de la famille d'Arnol'd (exemple \ref{ex:arnold}\,), nous croyons que les graphes qui décrivent les bords des langues sont donnés par des fonctions non-Lipschitz mais plutôt $1/2$-H\"older.

Dans la figure \ref{fig:escaliers} nous montrons les « escaliers du diable » que l'on trouve en lisant la figure \ref{fig:rotation} ligne par ligne. On peut observer que pour $\sigma=0.5$ la distribution du nombre de rotation est très proche de l'identité, ce qui suggérerait comme conjecture que pour des petites valeurs de $\sigma$, le fait d'avoir nombre de rotation irrationnel est un événement de mesure strictement positive.

\smallskip

Remarquons en outre que notre algorithme permet aussi de déterminer numériquement les distributions des $\{a_n\}_{n\le L}$, lorsque $\sigma=1$. Dans la figure \ref{fig:qdistrib1} on représente les fonctions de répartition des $a_n$ pour $n=1,\ldots,6$, comparées avec la distribution classique de Gauss-Kuzmin. Il semblerait que les distributions soient presque les mêmes : un tel fait devrait impliquer que le nombre de rotation est rationnel presque sûrement. Cependant cette ressemblance est moins évidente pour des différentes valeurs de $\sigma$ (dans la figure \ref{fig:qdistrib2} on peut voir les statistiques pour $\sigma=2$).

\newpage

\begin{figure}[!htp]
\[
\includegraphics[scale=.2]{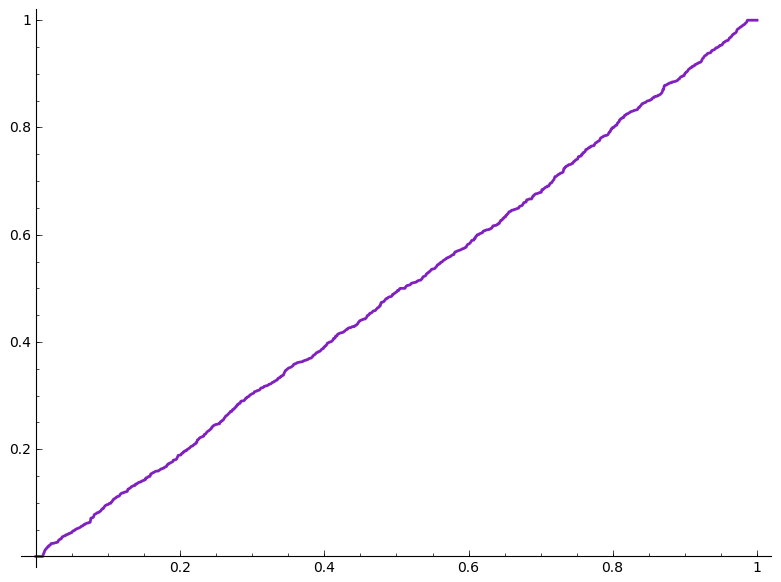}\quad
\includegraphics[scale=.2]{escalier10.png}
\quad\includegraphics[scale=.2]{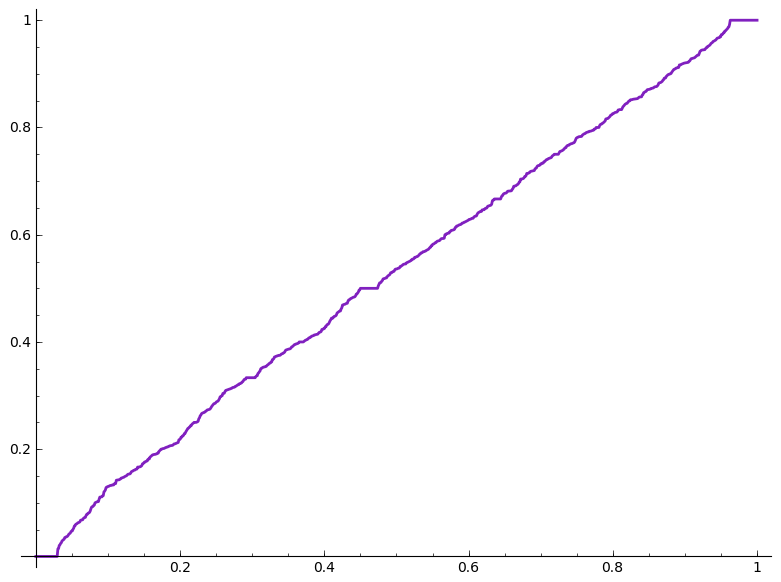}
\]
\[
\includegraphics[scale=.2]{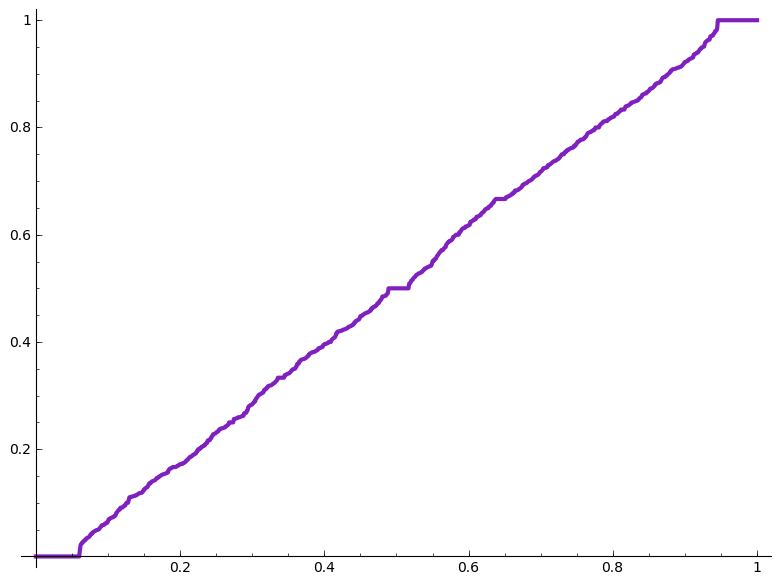}\quad
\includegraphics[scale=.2]{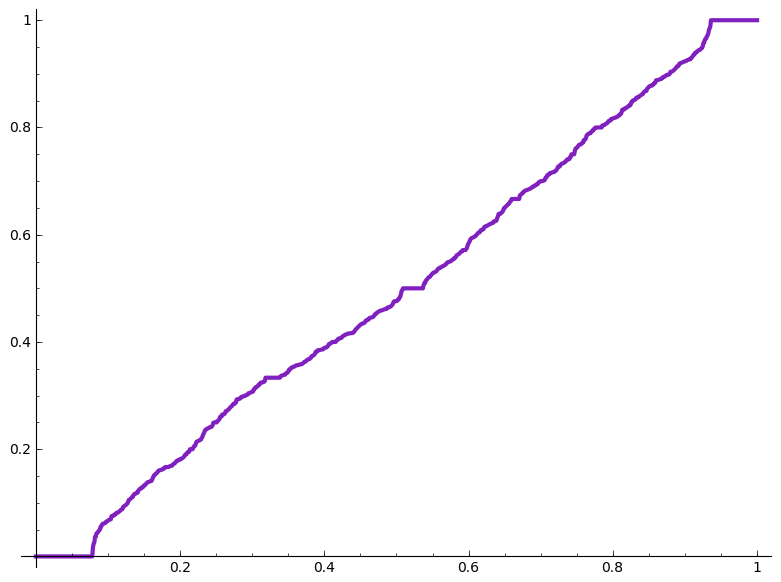}
\quad\includegraphics[scale=.2]{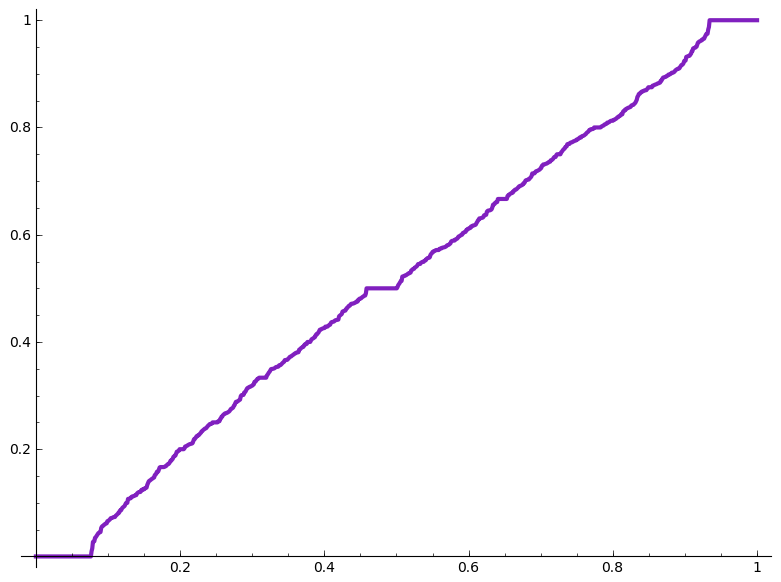}
\]
\[
\includegraphics[scale=.2]{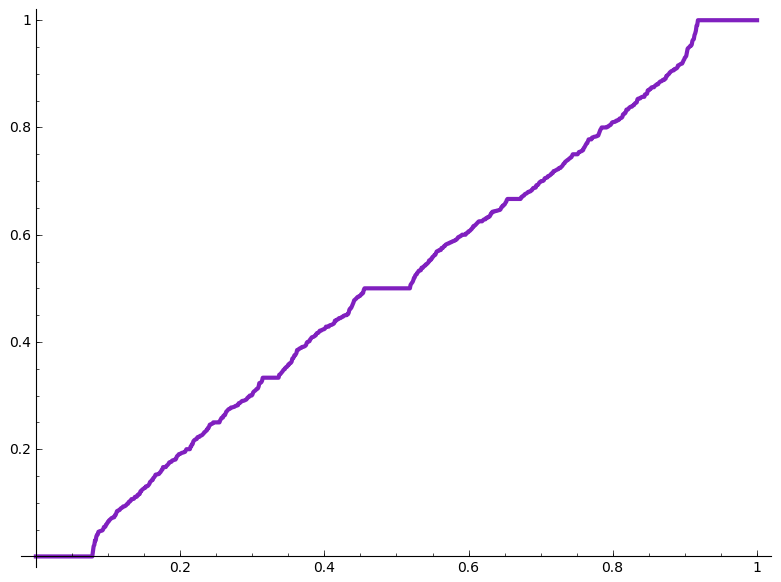}\quad
\includegraphics[scale=.2]{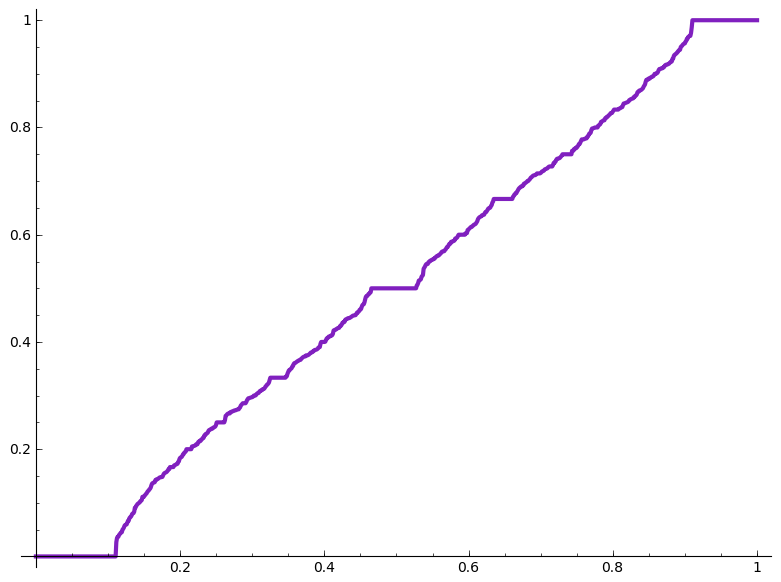}
\quad\includegraphics[scale=.2]{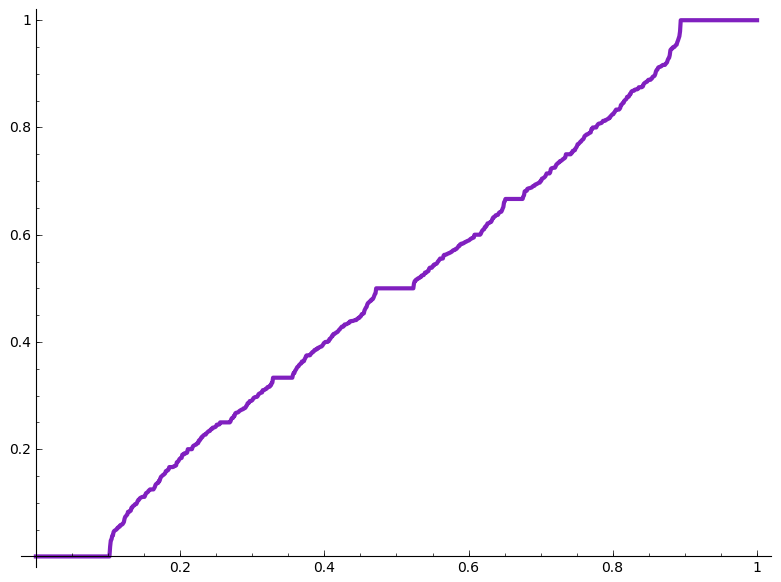}
\]
\[
\includegraphics[scale=.2]{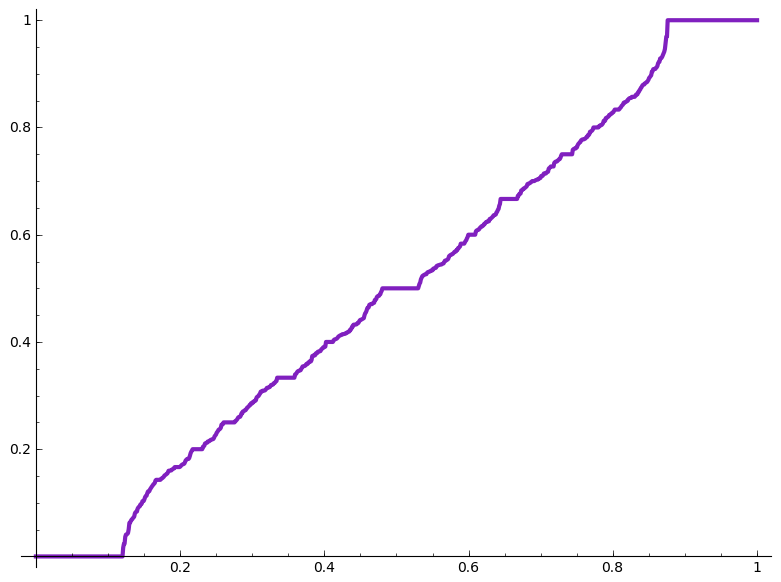}\quad
\includegraphics[scale=.2]{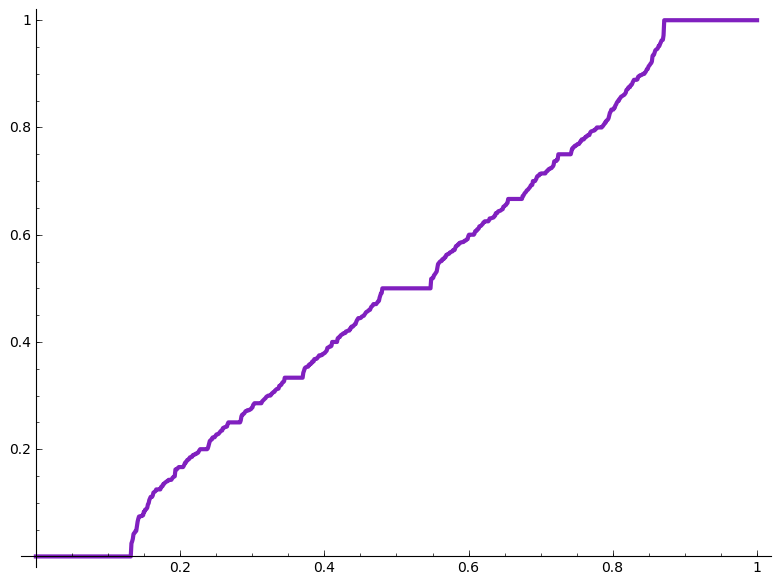}
\quad\includegraphics[scale=.2]{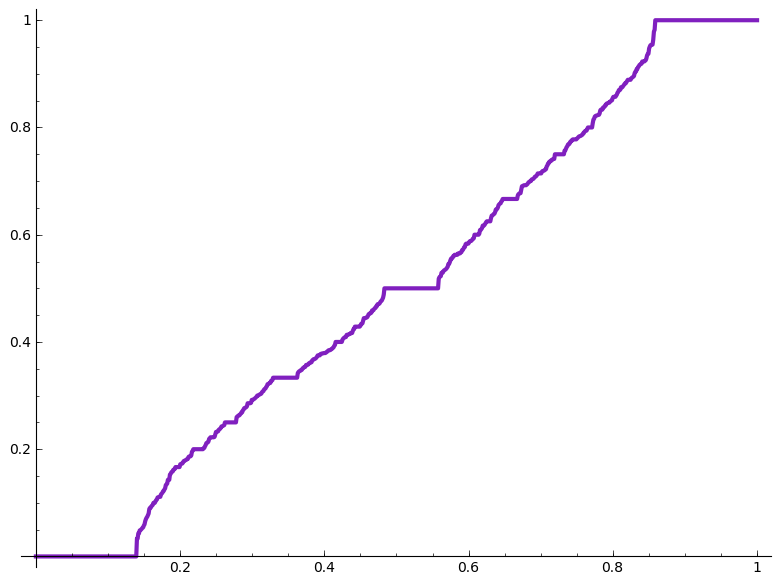}
\]
\[
\includegraphics[scale=.2]{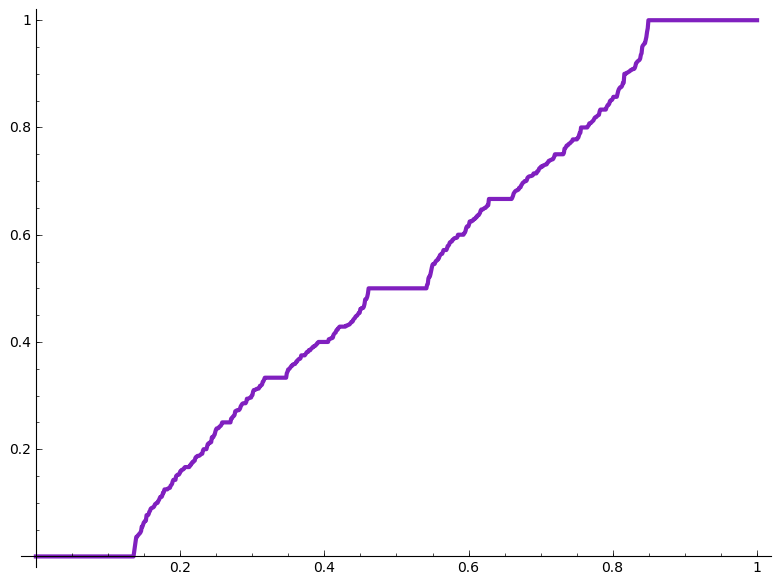}\quad
\includegraphics[scale=.2]{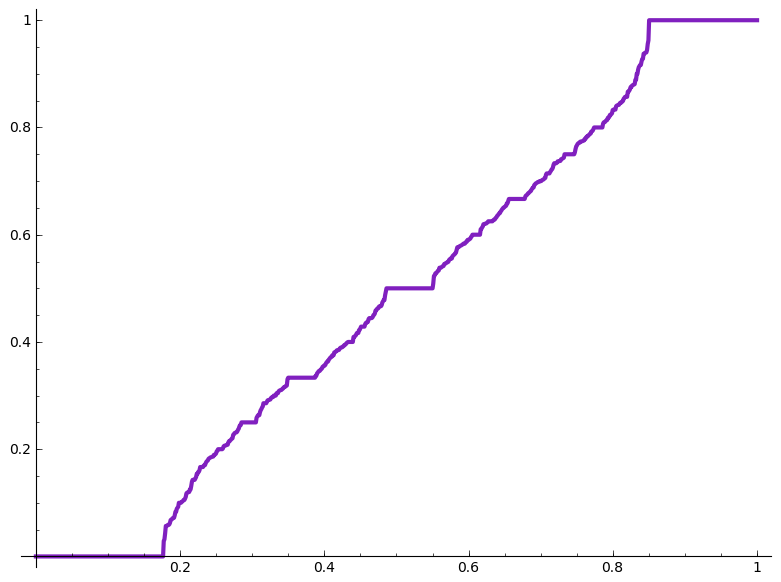}
\quad\includegraphics[scale=.2]{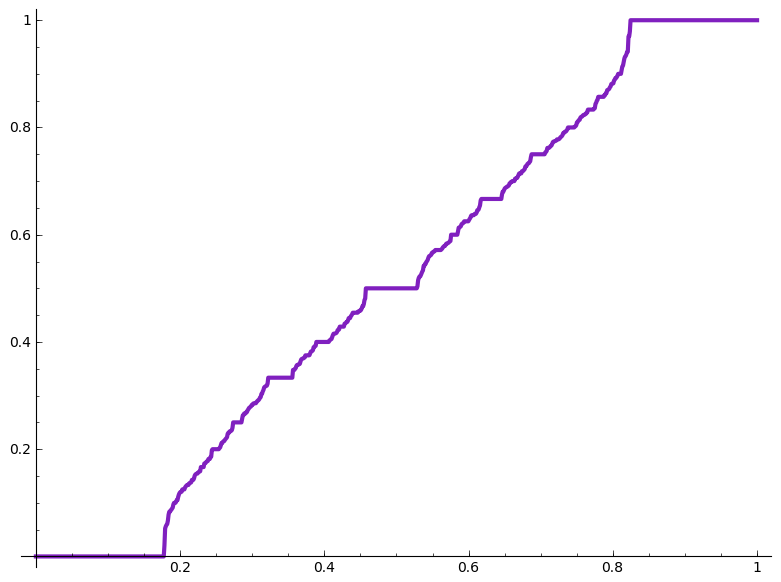}
\]
\caption{Distribution du nombre de rotation selon les mesures $\mu_\sigma$, avec $\sigma\in \{0.5, 1,\ldots, 7.5\}$ (lecture ligne par ligne)}\label{fig:escaliers}
\end{figure}

\begin{figure}[!htp]
\[
\includegraphics[scale=.24]{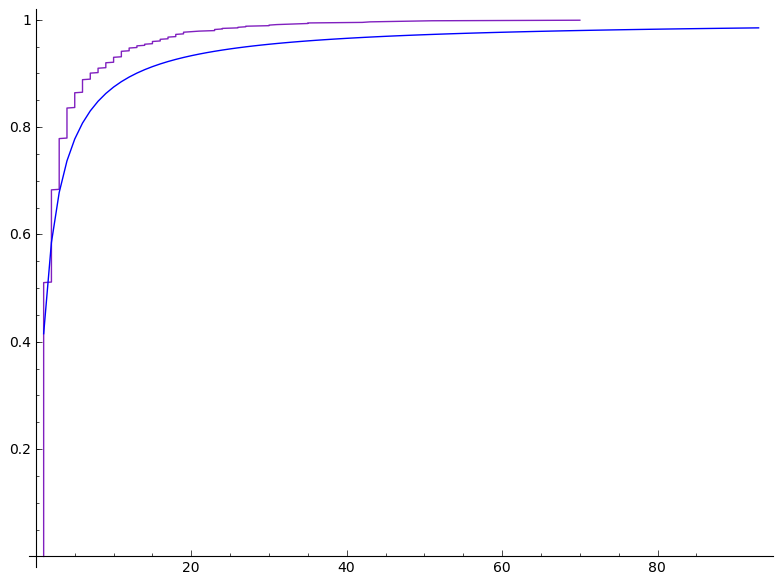}\quad
\includegraphics[scale=.24]{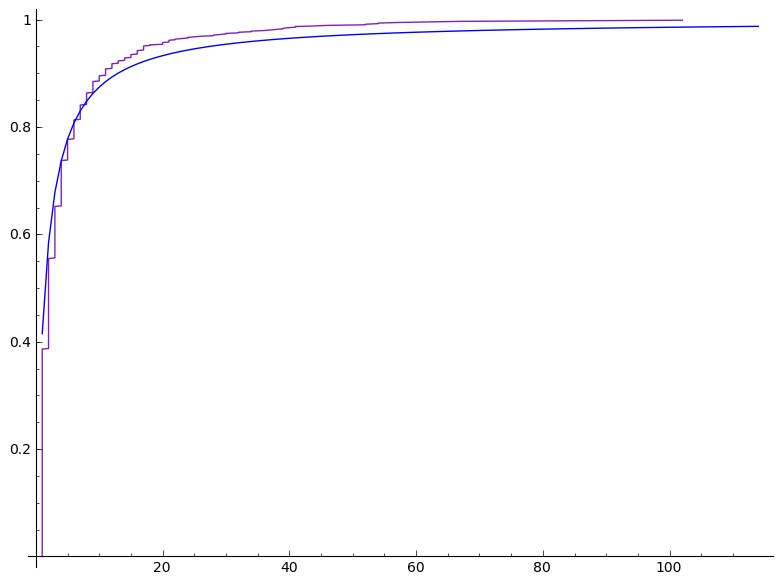}\]\[
\includegraphics[scale=.24]{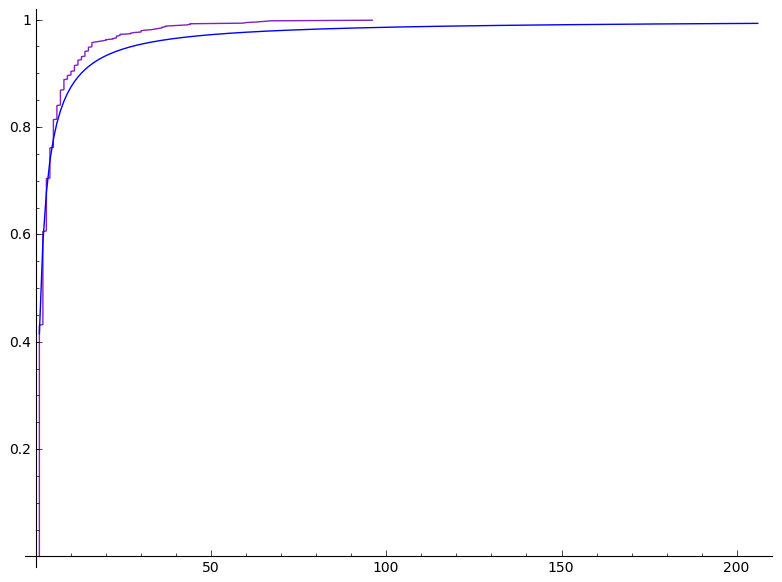}\quad
\includegraphics[scale=.24]{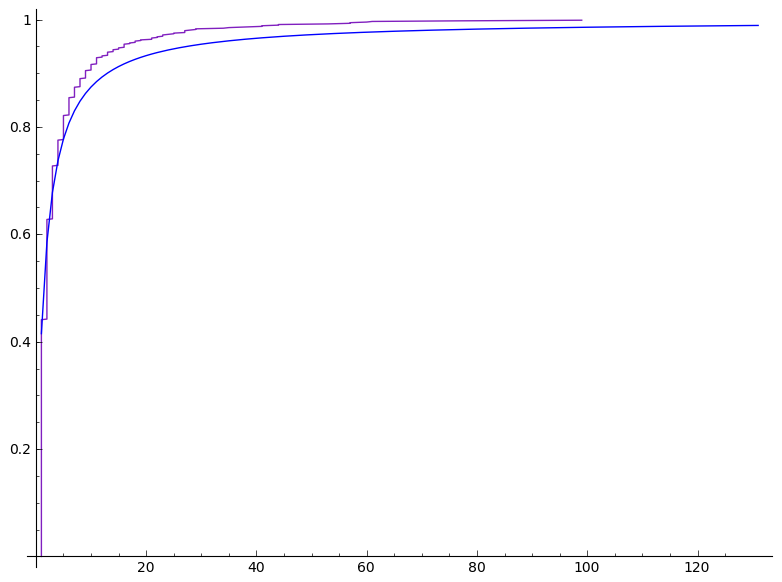}\]\[
\includegraphics[scale=.24]{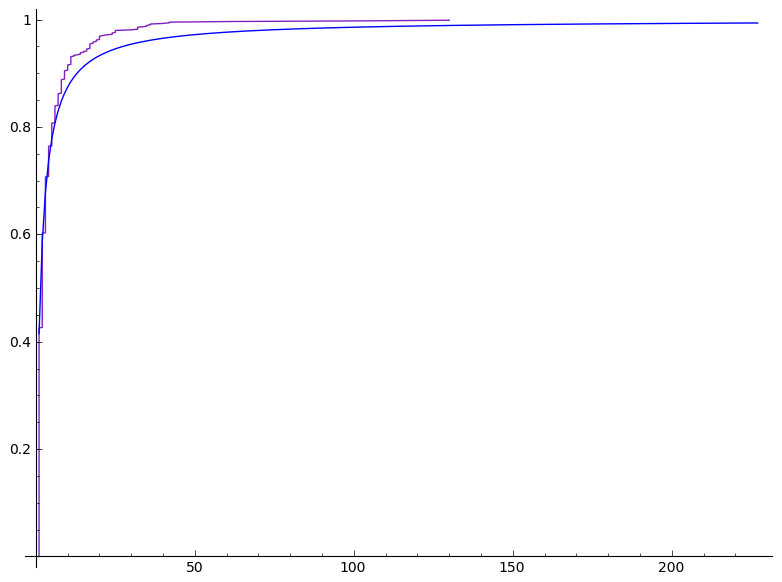}\quad
\includegraphics[scale=.24]{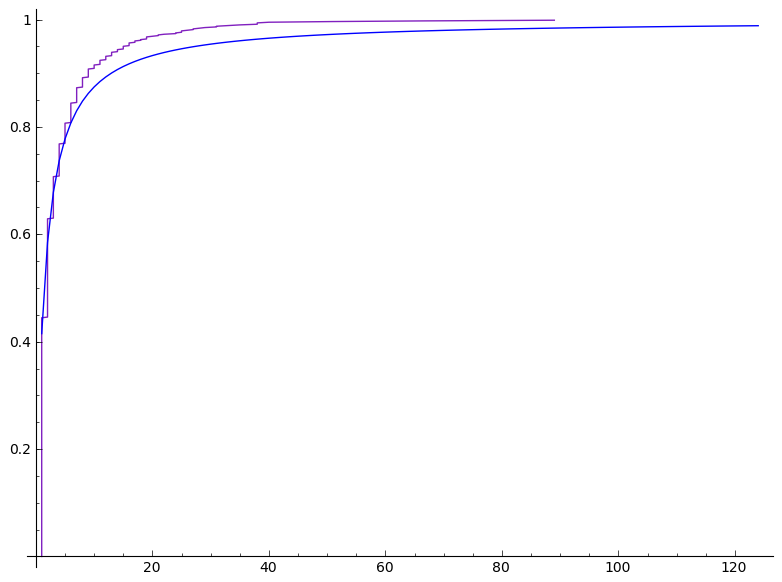}
\]
\caption{Fonctions des répartitions des $a_n$ ($n=1,\ldots,6$) lorsque $\sigma=1$, comparées avec la distribution classique de Gauss-Kuzmin}\label{fig:qdistrib1}
\end{figure}

\begin{figure}
\[
\includegraphics[scale=.24]{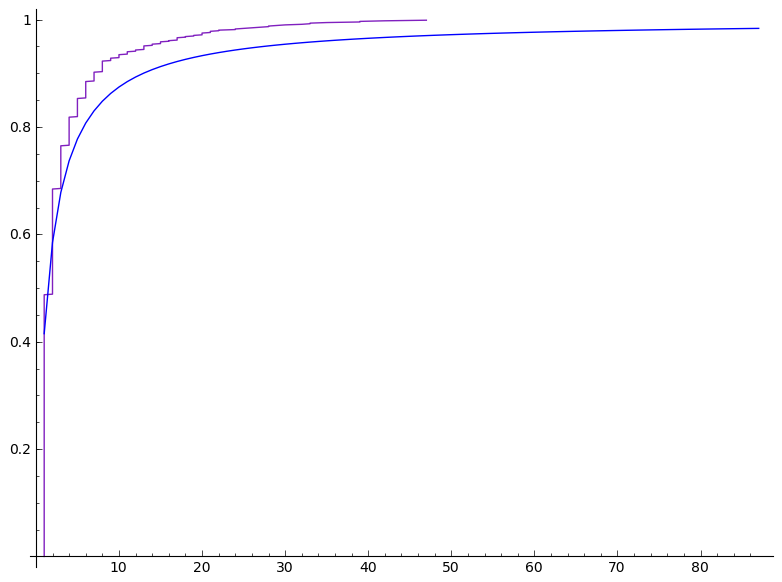}\quad
\includegraphics[scale=.24]{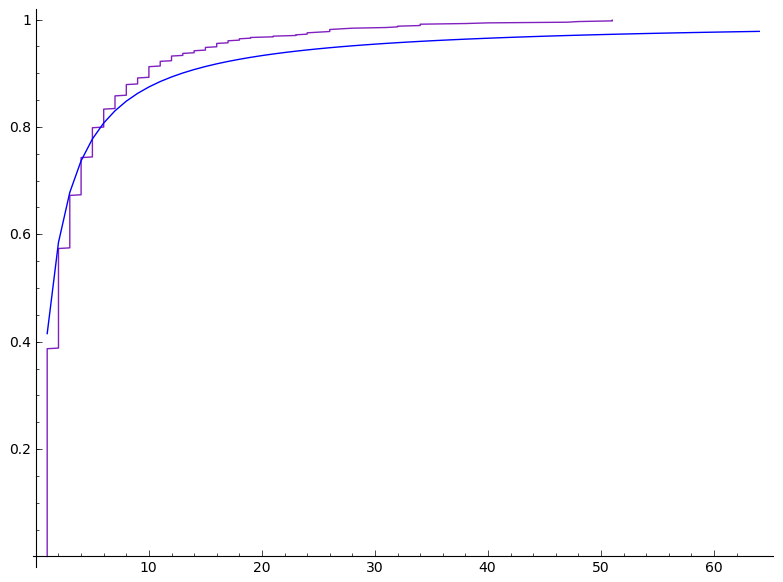}
\]
\[
\includegraphics[scale=.24]{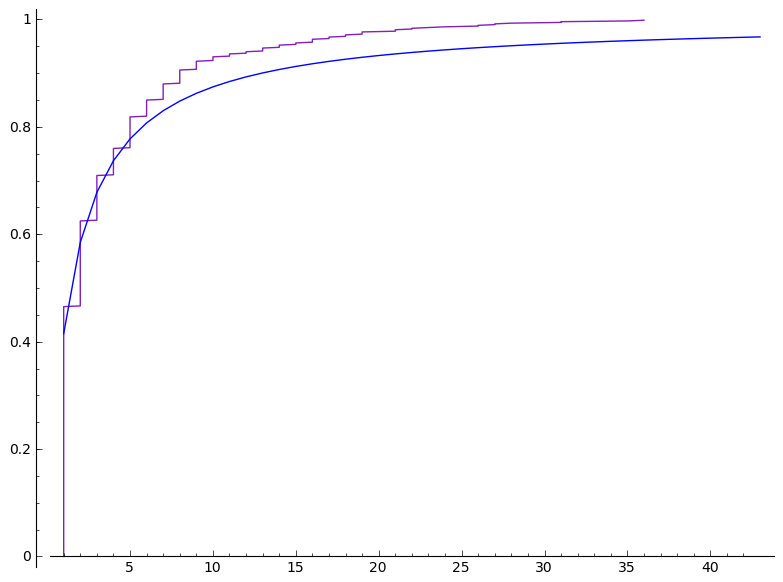}
\quad
\includegraphics[scale=.24]{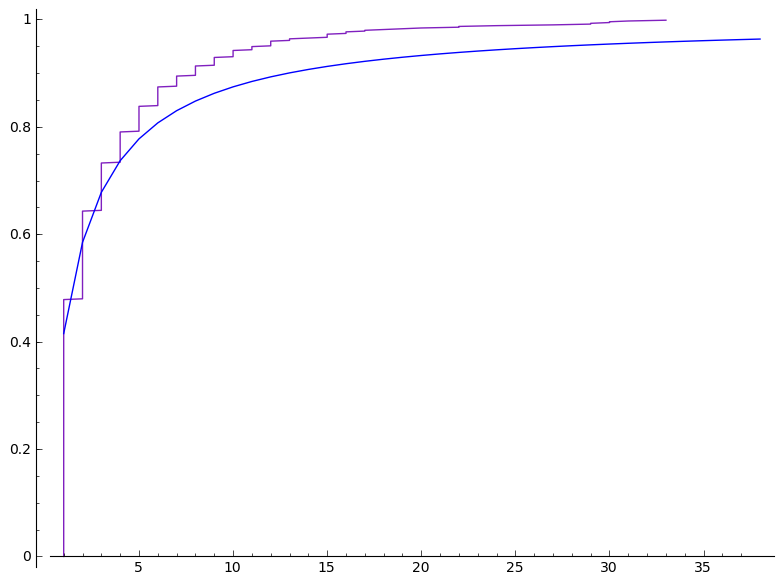}\]\[
\includegraphics[scale=.24]{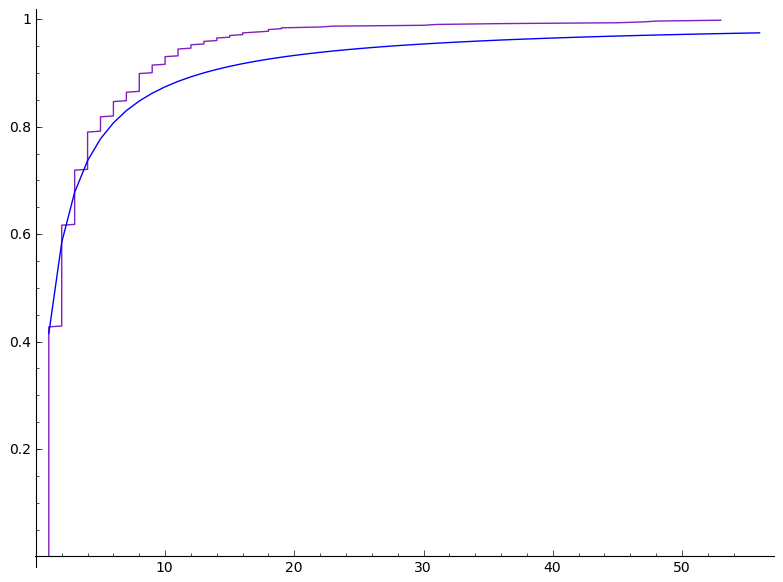}\quad
\includegraphics[scale=.24]{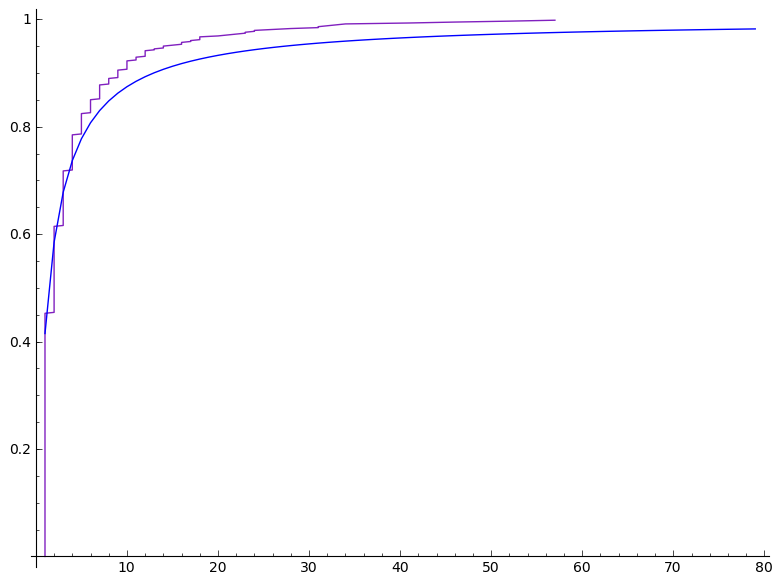}
\]
\caption{Fonctions des répartitions des $a_n$ ($n=1,\ldots,6$) lorsque $\sigma=2$, comparées avec la distribution classique de Gauss-Kuzmin}\label{fig:qdistrib2}
\end{figure}

\newpage

\begin{lstlisting}
from __future__ import division,print_function
import numpy as np
from scipy.interpolate import interp1d
import matplotlib.pyplot as plt
import matplotlib.mlab as mlab
from math import exp,sqrt,floor

n = 10000 #points
K = 10 #tests
L = 8 #length of continued fraction expansion
A = 1500 #maximum size of single element in continued fraction expansion


def partfrac(x):
	return x-floor(x)

#it computes the rational number defined by a continued fraction expansion b
#of length l+1
def rational_approximation(b,l):
	p = [0,1]
	q = [1,b[1]]
	for i in range(1,l+1):
		p.append(b[i+1]*p[i]+p[i-1])
		q.append(b[i+1]*q[i]+q[i-1])
	return p[l+1]/q[l+1]

#this is for computing the rotation number of a given homeomorphism f:
#we use the classical algorithm
#(explained for example in "de Melo, van Strien : One-dimensional dynamics")
#it computes the continued fraction expansion of the rotation number
#and uses rational_approximation to return the L-th rational approximation
#of it (or it stops before if it finds a periodic orbit)
def rotation(f):
	a = [0]
	orbit = []
	orbit.append(partfrac(f(0)))
	if orbit[0]==0 :
		return 0

	def shift(x):
		if partfrac(x)>orbit[0]:
			return partfrac(x)-1
		return partfrac(x)
	
	def first_return(p,pre_p,y):
		x = shift(f(y))
		while x<pre_p or x>p:
			x = shift(f(x))
		return x
	
	a.append(1)
	x = orbit[0]
	
	if shift(f(orbit[0]))==0:
		return 1/2
	
	if shift(f(orbit[0]))<0:
		while shift(f(x))<0:
			a[1] = a[1]+1
			x = shift(f(x))
			if a[1]>A:
				return 0
			if shift(f(x))==0:
				a[1] = a[1]+1
				return 1/a[1]
		orbit.append(shift(x))
		z = shift(f(x))
		a.append(0)
		while z>0:
			y = z
			z = first_return(shift(orbit[0]),shift(orbit[1]),z)
			a[2] = a[2]+1
			if a[2]>A:
				return 1/a[1]
			if z==0:
				a[2] = a[2]+1
				return rational_approximation(a,1)
		orbit.append(y)

	if shift(f(orbit[0]))>0:
		def shift(y):
			if partfrac(y)>=orbit[0]:
				return partfrac(y)-1
			return partfrac(y)
		orbit.append(orbit[0]-1)
		a.append(0)
		while shift(f(x))>0:
			a[2] = a[2]+1
			x = shift(f(x))
			if a[2]>A:
				return 1/a[1]
			if shift(f(x))==0:
				a[2] = a[2]+1
				return rational_approximation(a,1)
		orbit.append(shift(x))
		z = shift(f(x))

	for i in range(1,L):
		a.append(0)
		if shift(orbit[i+1])<shift(orbit[i]):
			while z>0:
				y = z
				z = first_return(shift(orbit[i]),shift(orbit[i+1]),z)
				a[i+2] = a[i+2]+1
				if a[i+2]>A:
					return rational_approximation(a,i)
				if z==0:
					a[i+2]=a[i+2]+1
					return rational_approximation(a,i+1)
		if shift(orbit[i+1])>shift(orbit[i]):
			while z<0:
				y = z
				z = first_return(shift(orbit[i+1]),shift(orbit[i]),z)
				a[i+2] = a[i+2]+1
				if a[i+2]>A:
					return rational_approximation(a,i)
				if z==0:
					a[i+2] = a[i+2]+1
					return rational_approximation(a,i+1)
		orbit.append(y)

	return rational_approximation(a,L)
				

#this generates pseudo-random MS diffeomorphisms
def rotationalea(var):
	s = np.random.normal(0, var, n+2)
	ss = np.cumsum(s)
	T = np.random.rand(1,1)
	w = []
	for i in range(0,n+1):
		w.append(exp(1/sqrt(n)*(ss[i+1]-i*ss[n+1]/n)))
    #exponential of browian motion
	a = [0,w[1]]
	for j in range(2,n+1):
		a.append(a[j-1]+w[j]+w[j-1])
    #we integrate w
	p = []
	q = []
	for i in range(0,n+1):
		p.append(i/n)
		q.append(a[i]/a[n])
	g = interp1d(p,q)
    #linear intepolation on n+1 points
	def G(x):
		return g(x-floor(x)) + floor(x) + T
    #we compose by a uniform rotation T and the compute the rotation number
	return rotation(G)

#here is the "main" part: we simulate K diffeomorphisms of variance parameter h/4,
#where h ranges from 1 to 31 then it prints out a sorted vector
for h in range(1,2**5):
	v = []
	for k in range(0,K):
		v.append(rotationalea(h/2**3))

	print(h/4)
	print(sorted(v))
	print()

\end{lstlisting}
\clearpage
%\cleardoublepage
%!TEX root = master.tex

\phantomsection % Ensures that a PDF bookmark is set here

\bigskip
\bsc{Michele Triestino}\\
Departamento de Matem\'atica PUC-Rio\\
Rua Marqu\^es de S\~ao Vicente, 225\\
G\'avea, Rio de Janeiro CEP 22451-900, Brasil\\
E-mail : mtriestino@mat.puc-rio.br

\addcontentsline{toc}{chapter}{Références}

\begin{thebibliography}{}


%%% A %%%


\myMRbibitem[Arn]{0140699}{arnold-small} \bsc{V.~I.~Arnol'd.}
Small denominators. I. Mapping the circle onto itself.
\emph{Izv. Akad. Nauk SSSR Ser. Mat.} \textbf{25} (1961), 21--86.


\myMRbibitem[Ast-Jon-Kup-Sak]{2892610}{homeoGFF}\bsc{K.~Astala, P.~Jones, A.~Kupiainen \& E.~Saksman.}
Random conformal weldings. 
\emph{Acta Math.} \textbf{207}, no.~\textbf{2} (2011), 203--254. 


%%% B %%%

\myMRbibitem[Bog~1]{1642391}{bogachev-gaussian}
\bsc{V.~I.~Bogachev.}
\emph{Gaussian measures}.
Mathematical Surveys and Monographs \textbf{62}, AMS (1998).

\myMRbibitem[Bog~2]{2663405}{bogachev-differentiable}
\bsc{\bysame.}
\emph{Differentiable measures and the Malliavin calculus}. Mathematical Surveys and Monographs \textbf{164}, AMS (2010).

\myMRbibitem[Bon-Día-Via]{2105774}{beyond}
\bsc{C.~Bonatti, L.~J.~Díaz \& M.~Viana.}
\emph{Dynamics beyond uniform hyperbolicity}.
Encyclop{\ae}dia of Mathematical Sciences \textbf{102}, Springer-Verlag (2005).



\bibitem[Bri${}^+$]{brin}\bsc{M.~G.~Brin et al.}
On {S}havgulidze's Proof of the Amenability of some Discrete Groups of Homeomorphisms of the Unit Interval.
\arxiv{0908.1353}



%%% C %%%

\myMRbibitem[Cam-Mar~1]{0010346}{cameron-martin} \bsc{R.~H.~Cameron \& W.~T.~Martin.}
Transformations of Wiener integrals under translations.
\emph{Ann. of Math. (2)} \textbf{45} (1944), 386--396.


\myMRbibitem[Cam-Mar~2]{0031196}{cameron-martin2}\bsc{\bysame \& \bysame.}  
The transformation of Wiener integrals by nonlinear transformations.
\emph{Trans. Amer. Math. Soc.} \textbf{66} (1949), 253--283.

\myMRbibitem[Cha]{1931043}{chaperon}\bsc{M.~Chaperon.}
Invariant manifolds revisited.
\emph{Proc. Steklov Inst. Math.} \textbf{236}, no.~\textbf{1} (2002), 415--433.


%%% D %%%


\bibitem[Den]{denjoy}\bsc{A.~Denjoy.}
Sur les courbes définies par les équations différentielles à la surface du tore.
\emph{J. Math. Pures Appl. (9)} \textbf{11} (1932), 333--375.


%%% E %%%

\myMRbibitem[Eps]{0281776}{epstein}\bsc{D.~B.~A.~Epstein.}
Almost all subgroups of a Lie group are free. 
\emph{J.~Algebra} \textbf{19} (1971), 261--262. 



%%% F %%%



\myMRbibitem[Fre]{2462372}{fremlin}\bsc{D.~H.~Fremlin.}
\emph{Topological measure spaces.} Part I. Deuxième édition.
Torres Fremlin (2006).



%%% G %%%



\myMRbibitem[Ghy]{1876932}{ghys-circle}\bsc{É.~Ghys.}
Groups acting on the circle.
\emph{Enseign. Math. (2)} \textbf{47}, no.~\textbf{3}-\textbf{4} (2001), 329--407.


\myMRbibitem[Got-Hed]{0074810}{gottschalk-hedlund}\bsc{W.~H.~Gottschalk \& G.~A.~Hedlund.}
\emph{Topological dynamics.}
AMS Colloquium Publications, Vol. \textbf{36} (1955).


%%% H %%%


\myMRbibitem[Hal]{0033869}{halmos}\bsc{P.~R.~Halmos.}
\emph{Measure theory.}
D.~Van Nostrand Company (1950).


\myMRbibitem[Her~1]{0538680}{herman} \bsc{M.~R.~Herman.}
Sur la conjugaison différentiable des difféomorphismes du cercle à des rotations.
\emph{Publ. Math. de l'IH\'ES} \textbf{49} (1979), 5--233.

\myMRbibitem[Her~2]{0458480}{herman_lebesgue} \bysame.
Mesure de Lebesgue et nombre de rotation. Dans \emph{Geometry and topology (Proc. III Latin Amer. School of Math., Inst. Mat. Pura Aplicada CNPq, Rio de Janeiro, 1976)}. Lecture Notes in Math. \textbf{597}, Springer (1977), 271--293. 

\bibitem[Hun-Kal]{prevalence}
\bsc{B.~R.~Hunt \& V.~Yu.~Kaloshin.}
Prevalence. Dans
\emph{Handbook of Dynamical Systems, vol.~3}, éditeurs H.~Broer, F.~Takens \& B.~Hasselblatt. \emph{Elsevier Science} (2010), 43--87.


%%% K %%%


\myMRbibitem[Kar-Shr]{1121940}{karatzas}\bsc{I.~Karatzas \& S.~E.~Shreve.}
\emph{Brownian motion and stochastic calculus}, deuxième édition.
Graduate Texts in Mathematics \textbf{113}, Springer-Verlag (1994).


\myMRbibitem[Kat-Has]{1326374}{katok-hasselblatt}\bsc{A.~B.~Katok \& B.~Hasselblatt.}
\emph{Introduction to the modern theory of dynamical systems.}
Enc. Math. and Appl., Cambridge University Press (1995).


\myMRbibitem[Katz-Orn]{1036902}{katz-ornstein}\bsc{Y.~Katznelson \& D.~Ornstein.}
The differentiability of the conjugacy of certain diffeomorphisms of the circle.
\emph{Ergodic Theory Dynam. Systems} \textbf{9}, no.~\textbf{4} (1989), 643--680.


\myMRbibitem[Kha-Sin~1]{0904139}{khanin-sinai}\bsc{K.~M.~Khanin \& Ya.~G.~Sinai.}
A new proof of M.~Herman's theorem.
\emph{Comm. Math. Phys.} \textbf{112}, no.~\textbf{1} (1987), 89--101.


\myMRbibitem[Kha-Sin~2]{0997684}{khanin-sinai2}\bsc{\bysame \& \bysame.}
Smoothness of conjugacies of diffeomorphisms of the circle with rotations.
\emph{Russian Math. Surveys} \textbf{44}, no.~\textbf{1} (1989), 69--99.

\myMRbibitem[Kha-Tep]{2545684}{herman-revisited}\bsc{\bysame \& A.~Teplins'kyi.}
Herman's theory revisited.
\emph{Invent. Math.} \textbf{178}, no.~\textbf{2} (2009), 333--344.


\bibitem[Kol]{kol54}
\bsc{A.~N.~Kolmogorov.} Ob\v{s}\v{c}aja teorija dinami\v{c}eskih sistem i klassi\v{c}eskaja mehanika (Théorie générale des systèmes dynamiques et mécanique classique).
\emph{Proc.~Intern.~Congr.~Math.} (1954), 315--333.

\myMRbibitem[Kop]{0270396}{kopell}\bsc{N.~Kopell.}
Commuting diffeomorphisms. Dans \emph{Global Analysis}, \emph{Proc. Sympos. Pure Math.} vol.~\textbf{XIV} (1968), 165--184.


\myMRbibitem[Kos]{1297679}{kosyak}\bsc{A.~V.~Kosyak.}
Irreducible regular {G}aussian representations of the groups 
of the interval and circle diffeomorphisms.
\emph{J. Funct. Anal.} \textbf{125}, no.~\textbf{2} (1994), 493--547.

\myMRbibitem[Kra-Sch]{1932566}{kra} \bsc{B.~Kra \& J.~Schmeling.}
Diophantine classes, dimension and Denjoy maps. 
\emph{Acta Arith.}  \textbf{105}, no.~\textbf{4} (2002),  323--340. 


\myMRbibitem[Kuz]{2302957}{kuzm07}\bsc{P.~A.~Kuzmin.}
On circle diffeomorphisms with discontinuous derivatives and
quasi-invariance subgroups of {M}alliavin-{S}havgulidze measures.
\emph{J. Math. Anal. Appl.} \textbf{330}, no.~\textbf{1} (2007), 744--750.


%%% L %%%


\myMRbibitem[Lév]{1188411}{levy}\bsc{P.~Lévy.} 
\emph{Processus stochastiques et mouvement brownien.}
Les Grands Classiques Gauthier-Villars, Éditions Jacques Gabay (1992). Réimpression de la deuxième (1965) édition.


%%% M %%%

\myMRbibitem[Mac]{0089999}{mackey}\bsc{G.~W.~Mackey.}
Borel structure in groups and their duals.
\emph{Trans. Amer. Math. Soc.} \textbf{85} (1957), 134--165.


\myMRbibitem[Mal-Mal~1]{1082629}{mall90} \bsc{M.~P.~Malliavin \& P.~Malliavin.}
Mesures quasi invariantes sur certains groupes de dimension infinie.
\emph{C. R. Acad. Sci. Paris Sér. I Math.} \textbf{311}, no.~\textbf{12} (1990), 765--768.


\myMRbibitem[Mal-Mal~2]{1166819}{mall91} \bsc{ \bysame \& \bysame.}
An infinitesimally quasi-invariant measure on the group of 
diffeomorphisms of the circle.
\emph{Special functions ({O}kayama, 1990)}, ICM-90 Satell. Conf. Proc., Springer (1991), 234--244.


\myMRbibitem[Man-Yor]{2454984}{mayo08} \bsc{R.~Mansuy \& M.~Yor.}
\emph{Aspects of {B}rownian motion.}
Universitext, Springer-Verlag (2008).


\myMRbibitem[Mat-Hir-Yor]{2203675}{mayo05} \bsc{H.~Matsumoto \& M.~Yor.}
Exponential functionals of {B}rownian motion. {I}. {P}robability laws at fixed time.
\emph{Probab.~Surv.} \textbf{2} (2005), 312--347.


\myMRbibitem[dMe-vSt]{1239171}{demelo-strien}\bsc{W.~de Melo \& S.~van Strien.}
\emph{One-dimensional dynamics.}
Erg.~Math.~Grenz. (3), \textbf{25}. Springer-Verlag (1993).


\myMRbibitem[Mör-Per]{2604525}{peres}\bsc{P.~Mörters \& Y.~Peres.}
\emph{Brownian motion.}
Cambridge Series in Statistical and Probabilistic Mathematics, Cambridge University Press (2010).


\myMRbibitem[Mos]{0206461}{moser}\bsc{J.~Moser.}
A rapidly convergent iteration method and non-linear differential equations. II.
\emph{Ann. Scuola Norm. Sup. Pisa (3)} \textbf{20} (1966), 499--535.


%%% N %%%

\myMRbibitem[Nav~1]{2394157}{navasEM} \bsc{A.~Navas.}
\emph{Grupos de difeomorfismos del círculo.}
Ensaios Matemáticos \textbf{13}, Soc.~Brasil.~Mat.~(2007).

\myMRbibitem[Nav~2]{2809110}{booknavas}\bysame.
\emph{Groups of circle diffeomorphisms.}
Chicago Lectures in Mathematics (2011).


\myMRbibitem[Ner]{1311489}{neretin}\bsc{Yu.~A. Neretin.}
Representations of {V}irasoro and affine {L}ie algebras.
\emph{Representation theory and noncommutative harmonic analysis, 
{I}}, Enc. Math. Sci. \textbf{22}, Springer (1994), 157--234.



\myMRbibitem[Nor]{1600129}{norton2}\bsc{A.~Norton.}
Denjoy's theorem with exponents.
\emph{Proc. Amer. Math. Soc.} \textbf{127}, no.~\textbf{10} (1999), 3111--3118. 


%%% O %%%


%%% P %%%

\bibitem[Poi]{poincare}\bsc{H.~Poincaré.}
Mémoire sur les courbes définies par les équations différentielles, IV.
\emph{J. Math. Pures Appl. (3)} \textbf{2} (1886), 151--217.


%%% R %%%

\myMRbibitem[Ram]{0349945}{ramer}
\bsc{R.~Ramer.}
On nonlinear transformations of Gaussian measures.
\emph{J.~Functional Analysis} \textbf{15} (1974), 166--187. 

\myMRbibitem[Rev-Yor]{1725357}{revyor}\bsc{D.~Revuz \& M.~Yor.}
\emph{Continuous martingales and {B}rownian motion.}
Troisième édition, Springer-Verlag (1999).


%%% S %%%


\myMRbibitem[Sha~1]{0509384}{shav78}\bsc{E.~T.~Shavgulidze.}
An example of a measure which is quasi-invariant relative to 
the action of a group of diffeomorphisms of the circle.
\emph{Funktsional. Anal. i Prilozhen.} \textbf{12}, no.~\textbf{3} (1978), 55--60, 96.


\myMRbibitem[Sha~2]{0984644}{shav88}\bysame. 
A measure that is quasi-invariant with respect to the action
of a group of diffeomorphisms of a finite-dimensional manifold.
\emph{Dokl. Akad. Nauk SSSR} \textbf{303}, no.~\textbf{4} (1988), 811--814.


\myMRbibitem[Sha~3]{1632120}{shav97}\bysame. 
Quasi-invariant measures on groups of diffeomorphisms.
\emph{Tr. Mat. Inst. Steklova} \textbf{217} (1997), 189--208.


\myMRbibitem[Sha~4]{1832477}{shav00}\bysame.  
Some properties of quasi-invariant measures on groups of
diffeomorphisms of the circle.
\emph{Russ. J. Math. Phys.} \textbf{7}, no.~\textbf{4} (2000), 464--472.

\bibitem[She]{zipper}\bsc{S.~Sheffield.}
Conformal weldings of random surfaces: SLE and the quantum gravity zipper.
\arxiv{1012.4797}

\myMRbibitem[Sie]{0007044}{siegel}\bsc{C.~L.~Siegel.}
Iteration of analytic functions.
\emph{Ann. of Math. (2)} \textbf{43} (1942), 607--612.


\myMRbibitem[Sta]{0967471}{stark}\bsc{J.~Stark.}
Smooth conjugacy and renormalisation for diffeomorphisms of the circle.
\emph{Nonlinearity} \textbf{1}, no.~\textbf{4} (1988), 541--575.

\myMRbibitem[Ste]{0369665}{steinlage75}\bsc{R.~C.~Steinlage.}
On Haar measure in locally compact $T_2$ spaces. 
\emph{Amer. J. Math.} \textbf{97} (1975), 291–-307. 

%%% T %%%


\myMRbibitem[Tep]{2424643}{teplinsky}\bsc{O.~Yu.~Teplins'kyi.}
On the smoothness of the conjugacy of the circle diffeomorphisms to rigid rotations.
\emph{Ukrainian Math. J.} \textbf{60}, no.~\textbf{2} (2008), 310--326.


\bibitem[Tri]{these}
\bsc{M.~Triestino.}
La dynamique des difféomorphismes du cercle selon le point de vue de la mesure. \emph{Thèse de l'Université de Lyon} (2014).


%%% U %%%

\myMRbibitem[Üst-Zak]{1736980}{zakai}
\bsc{A.~Üstünel \& M.~Zakai.}
\emph{Transformation of measure on Wiener space.}
Springer Monographs in Mathematics, Springer-Verlag (2000).

%%% W %%%

\myMRbibitem[Wei]{0005741}{weil}\bsc{A.~Weil.}
\emph{L'intégration dans les groupes topologiques et ses applications.}
Actual. Sci. Ind., no.~\textbf{869}, Hermann et Cie., Paris (1940).


\bibitem[Wie~1]{wiener23}\bsc{N.~Wiener.}
Differential space.
\emph{Journal Math. Phys.} \textbf{2} (1923), 131--174.


\myMRbibitem[Wie~2]{1555316}{wiener}\bsc{\bysame.}
Generalized harmonic analysis.
\emph{Acta Math.} \textbf{55}, no.~\textbf{1} (1930), 117--258.


%%% Y %%%

\myMRbibitem[Yoc~1]{0777374}{yoccoz} \bsc{J.-C.~Yoccoz.}
Conjugaison différentiable des difféomorphismes du cercle dont le nombre de rotation vérifie une condition diophantienne.
\emph{Ann. Sci. École Norm. Sup. (4)} \textbf{17}, no.~\textbf{3} (1984), 333--359.


\myMRbibitem[Yoc~2]{1367354}{diviseurs} \bysame.
Petits diviseurs en dimension 1.
\emph{Astérisque} no.~\textbf{231}, SMF (1995).


\myMRbibitem[Yor]{1854494}{yor01}\bsc{M.~Yor.}
\emph{Exponential functionals of {B}rownian motion and related processes.}
Springer Finance, Springer-Verlag (2001).


\end{thebibliography}
\end{document}